\documentclass[11pt]{amsart}
\usepackage{amssymb,amsmath}
\usepackage{bm}
\usepackage[]{graphicx}
\usepackage{color}
\usepackage{amssymb, epsfig}
\numberwithin{equation}{section}

\newcommand{\ep}{\epsilon}

\newcommand{\x}{\xi}

\newcommand{\eps}{\varepsilon}

\numberwithin{figure}{section}
\newcommand{\red}{\textcolor{black}}

\newcommand{\meps}{\mu(\varepsilon)}  
\newcommand{\mepsq}{\mu^2(\varepsilon)}  
\newcommand{\deps}{\delta(\varepsilon)}  
\newcommand{\depsinv}{\delta^{-1}(\varepsilon)}  
\newcommand{\depsq}{\delta^2(\varepsilon)}  
\newcommand{\depsqinv}{\delta^{-2}(\varepsilon)}  
\newcommand{\neps}{{{\nu(\varepsilon)}}}

\usepackage{bm}

\usepackage[mathscr]{euscript}
\let\euscr\mathscr \let\mathscr\relax
\usepackage{mathrsfs}

\usepackage{tikz, pgfplots}
\pgfplotsset{compat=1.6}

\usepackage{empheq}

\newtheorem{theorem}{Theorem}[section]
\newtheorem{lemma}[theorem]{Lemma}

\newtheorem{proposition}[theorem]{Proposition}
\newtheorem{remark}[theorem]{Remark}

\usepackage{float}
\usepackage{caption}
\begin{document}

\title[Layer Dynamics for the Cahn-Hilliard / Allen-Cahn Equation]
{Layer Dynamics for the one dimensional $\bm{\eps}$-dependent
Cahn-Hilliard / Allen-Cahn Equation}

\author[D.~C. Antonopoulou]{D.C.~Antonopoulou$^{\dag\ddag}$}
\author[G.~Karali]{G.~Karali$^{*\ddag}$}
\author[K.~Tzirakis]{K.~Tzirakis$^\ddag$}

\thanks
{$^{\dag}$ Department of Mathematical and Physical Sciences,
University of Chester, UK}

\thanks
{$^{*}$ Department of Mathematics and Applied Mathematics,
University of Crete, Greece.}

\thanks
{$^{\ddag}$ Institute of Applied and Computational Mathematics,
FORTH, Heraklion, Greece.}
%
%
%
%

%
%
%

%
%
\begin{abstract}
We study the dynamics of the one-dimensional $\eps$-dependent
Cahn-Hilliard / Allen-Cahn equation within a neighborhood of an
equilibrium of $N$ transition layers, that in general does not
conserve mass. Two different settings are considered which differ
in that, for the second, we impose a mass-conservation constraint
in place of one of the zero-mass flux boundary conditions at
$x=1$. Motivated by the study of Carr and Pego on the layered
metastable patterns of Allen-Cahn in \cite{CarrPego}, and by this
of Bates and Xun in \cite{BatesXunI} for the Cahn-Hilliard
equation, we implement an $N$-dimensional, and a mass-conservative
$ N-1$-dimensional manifold respectively; therein, a metastable
state with $N$ transition layers is approximated. We then
determine, for both cases, the essential dynamics of the layers
(ode systems with the equations of motion), expressed in terms of
local coordinates relative to the manifold used. In particular,
we estimate the spectrum of the linearized Cahn-Hilliard /
Allen-Cahn operator, and specify wide families of
$\eps$-dependent weights $\delta(\eps)$, $\mu(\eps)$, acting at
each part of the operator, for which the dynamics are stable and
rest exponentially small in $\eps$. Our analysis enlightens the
role of mass conservation in the classification of the general
mixed problem into two main categories where the solution has a
profile close to Allen-Cahn, or, when the mass is conserved, close
to the Cahn-Hilliard solution.
\end{abstract}

\maketitle
\pagestyle{myheadings}
\thispagestyle{plain}
%
%

\section{Introduction}
\subsection{The equation}
In this paper, we examine the dynamics of the Cahn-Hilliard /
Allen-Cahn equation
\begin{equation}\label{eq:deltaCH-AC}
 u_t = - \delta(\eps) \, \big(\varepsilon^2 u_{xx} - f(u)\big)_{xx} \; + \; \mu(\eps)(\varepsilon^2 u_{xx}  - f(u)) , \quad x \in
 (0,1),\quad t>0,
\end{equation}
in a neighborhood of a layered equilibrium parameterized by a
small positive constant $\eps$.

The nonlinearity $f(u)=F'(u)$ is the derivative of a double
equal-well potential $F$ taking a non-degenerate global minimum
value zero at $ u =\pm 1, $ in particular
\begin{flalign}
&  & & F(\pm1)  = f(\pm1)=0, &  
\label{eq:W cond1}
\\
& & &
f'(\pm1)  > 0, 
& 
 \label{eq:W cond2}
\\
& & & F(u) > 0 \quad \mbox{for} \quad u\in (-1,1). \label{eq:W
cond3} &
\end{flalign}

We define, for simplicity,
$$f(u):=u^3-u,$$ which is a typical example for a potential
$F(u):=\tfrac{1}{4}(u^2-1)^2$. However, many of the results are
valid for more general nonlinearities.

The initial condition
$$u_0(x;\eps)=:u(x,0),\;x\in(0,1),$$ is assumed layered with respect to
$\eps$ which stands as a measure of the layers width
corresponding to a time scale proportional to $e^{C\eps^{-1}}$
for $C>0$, and therefore, to very long times as $\eps\rightarrow
0$ where the solution is expected to change very slowly; see in
\cite{CarrPego} for the analogous considerations on the Allen-Cahn
equation.

We introduce the positive constant $\delta(\eps)>0$ and the
non-negative one $\mu(\eps)\geq 0$ in order to control the
coexistence of the $2$ operators in terms of $\eps$. Moreover, we
impose the presence of the Cahn-Hilliard part in the combined
model as $\delta(\eps)\neq 0$, while for $\delta(\eps):=1$ and
$\mu(\eps):=0$ the Cahn-Hilliard equation stands as a special
case.

An equation of the form \eqref{eq:deltaCH-AC} has been first
introduced in \cite{KK} as a mean field model of the microscopic
dynamics associated with adsorption and desorption mechanisms in
the context of surface processes; see also in \cite{KV,AKM} for
more details on the physical motivation. The combined
Cahn-Hilliard / Allen-Cahn model in study describes surface
diffusion including particle/particle interactions and adsorption
and desorption from the surface. It is noticeable that the
mobility is completely different from this of Allen-Cahn
equation, which implies that the diffusion speeds up the mean
curvature flow, (\cite{HM,Spohn}).

{\red Equation \eqref{eq:deltaCH-AC} is a gradient flow for the
above energy with respect to an $\eps$-weighted metric. In
particular, the standard Allen-Cahn equation is written as
$$u_t=(\Delta-\eps^{-2}I)u=:\mathcal{A}^\eps(u),$$
the Cahn-Hilliard equation, after rescaling, as
$$u_t=(-\eps\Delta)(\Delta-\eps^{-2}I)u=(-\eps\Delta)(\mathcal{A}^\eps(u)),$$
while the mixed model \eqref{eq:deltaCH-AC}, as
$$u_t=(\eps^2(-\delta(\eps)\Delta+\mu(\eps)I))(\Delta-\eps^{-2}I)u=(\eps^2(-\delta(\eps)\Delta+\mu(\eps)I))(\mathcal{A}^\eps(u)),$$
and the $\eps$-weighted metric is given by
$$<f,g>_\eps:=(f,(\eps^2(-\delta(\eps)\Delta+\mu(\eps)I))^{-1}g),$$
for $(\cdot,\cdot)$ the $L^2((0,1))$ inner product; see also the
discussion in \cite{KK} for the case $\delta(\eps):=1$,
$\mu(\eps)=c_0\eps^{-2}$.

The deterministic Allen-Cahn equation was proposed in \cite{a-c}
as a model for the dynamics of interfaces of crystal structures
in alloys.  The same equation also appears as a model for various
other problems, including population genetics and nerve
conduction. As far as the one-dimensional case is concerned, the
limiting behaviour was analyzed in \cite{CarrPego,chen}. After a
very short time, generation of  many very steep transition layers
is observed. These well developed transition layers then start to
move very slowly, and each time a pair of transition layers meet,
the two layers annihilate each other, and thus the number of
layers decrease gradually. Although those collision-annihilation
process takes place rather quickly, the motion of layers between
the collisions is extremely slow, and the profile of the layers
look nearly unchanged during those slow motion periods; in other
words, the solution exhibits a metastable pattern. The situation
is quite different in the multi-dimensional case, where such
metastable patterns hardly appear because of the curvature effect
on the motion of the interface as $\eps\rightarrow 0$; for
rigorous justification of singular limits, see for example in
\cite{BS}, \cite{C1,C2}, \cite{MS2}.

The deterministic Cahn-Hilliard equation, proposed by
\cite{Cahn}, describes the evolution of transitions (mass
transfer) during the phase separation of alloys.
In the case of only two layers, the exponentially slow dynamics
have been studied in \cite{ABF}, where a one-dimensional invariant
manifold of slowly moving states was constructed. More details of
the phenomenon and the motion towards the boundary can be found in
\cite{FH},  \cite{F}.

Considering the CH/AC equation \eqref{eq:deltaCH-AC}, in higher
dimensions, and for a specific choice of the coefficients
$\mu(\eps):=\eps^{-2},\;\;\delta(\eps):=\mathcal{O}(1)$, motion
by mean curvature was derived on the sharp interface limit
$\eps\rightarrow 0$, in \cite{KK}, as in the Allen-Cahn equation
limiting dynamics. A main result of our work is the analysis of
the spectrum of the linearized operator, where a crucial spectral
condition
$$\eps^{2}\mu(\eps)\geq \mathcal{O}(\delta(\eps)),$$ is determined.
For the existence and the regularity properties of
\eqref{eq:deltaCH-AC}, we refer to \cite{KN}, while the
stochastic version thereof was investigated in \cite{AKM}. As it
has been observed in \cite{AKM}, the operator at the right-hand
side of the mixed equation \eqref{eq:deltaCH-AC} is strongly
parabolic in the sense of Petrovsk\v{i}i and the bi-Laplacian
since existing ($\delta(\eps)>0$) dominates, resulting to
regularity properties identical to the Cahn-Hilliard equation (at
least in the stochastic setting). However, the sharp interface
limit of the deterministic equation may exhibit a different
profile closer to this of the Allen-Cahn, \cite{KK}. The above,
ignites a special interest on the scaling of the chosen parameters
$\delta(\eps),\;\mu(\eps)$, and the motivation of a further
investigation of their influence on the dynamics of the layers.

It is well known that the Cahn-Hilliard equation with the standard
Neumann boundary conditions for $u$ and its Laplacian is mass
conservative in the sense that the integral of the solution in
space is time independent. In contrast, the Allen-Cahn equation
with Neumann or Dirichlet boundary is not satisfying such a
property unless a non-local integral term is added, which is the
case for the mixed equation as well; this is not considered in
this work, however it consists a future plan in progress the
detailed investigation of the dynamics for such a version, i.e.,
\eqref{eq:deltaCH-AC} with the extra integral term.

Our main aim is to obtain the equations of motion and estimate
the dynamics of a fixed number of layers, when $\eps$ is
sufficiently small, for the combined model \eqref{eq:deltaCH-AC},
and in dimension one. For this, when the initial and boundary
value problem involves the Neumann conditions, and so mass
conservation is not holding true, the solution will be
approximated into the manifold constructed and effectively used
for the Allen-Cahn equation in the classical result of Carr and
Pego, \cite{CarrPego}. Then, by imposing mass conservation, not
through the pde but replacing one only of the b.c. with an
integral one, we will apply the mass conserving manifold of Bates
and Xun, \cite{BatesXunI,BatesXunII}, which has been proposed for
the integrated Cahn-Hillliard equation. There, the derived initial
and boundary value problem for the integrated equation is
identical to this of \cite{BatesXunI,BatesXunII}, when
$\mu(\eps):=0$. We note that the problem is of fourth order,
since $\delta(\eps)$ is not vanishing, while in dimension one the
boundary consists of only two well separated points where four
boundary conditions are applied on, and could be therefore of
different type. In both cases we will determine the ode systems of
the dynamics, and investigate the main order terms with respect to
the order in $\eps$ of $\delta(\eps)$ and $\mu(\eps)$, and
stability.

The general approach for deriving the equations of motion, as a
system of odes, consists of specifying the approximate solution
into a proper approximate manifold with a residual orthogonal to
a set of approximate tangent vectors to the manifold.
Differentiating in time the orthogonality condition then yields
the system describing the dynamics of the layers.

\subsection{Main results}


Let $u$ be the solution of \eqref{eq:deltaCH-AC}, with the
standard Neumann conditions, $u_{x}=u_{xxx}=0$ at $x=0,1$, (non
mass conserving case). Given a configuration
$$ h = (h_1, \ldots, h_{\scriptscriptstyle N}), $$ of exactly $ N
$ layer positions, we will construct a function $ u^h=u^h(x) $
approximating a metastable state of $u$ with $ N $ transition
layers. Here, we will use the parameterization of Carr and Pego,
\cite{CarrPego}, for the approximate manifold.

More precisely, the function $ u^h $ will almost satisfy the
steady-state problem of \eqref{eq:deltaCH-AC}, that is
$A_\varepsilon(u^h)$ is very small, where $ A_\varepsilon $ is the
operator given by the right-hand side of \eqref{eq:deltaCH-AC}
\begin{equation}\label{eq:def of Aepsilon2}
 A_\varepsilon(u) :=
 - {\deps}
\big(\varepsilon^2 u_{xx} - f(u)\big)_{xx} +
\mu(\eps)\big(\varepsilon^2 u_{xx}  -
f(u)\big)=:-\delta(\eps)A_{_{1,\varepsilon}}(u)+\mu(\eps)A_{_{2,\varepsilon}}(u),
\end{equation}
for
$$A_{_{1,\varepsilon}}(u):=\Big{(}\varepsilon^2 u_{xx} -
f(u)\Big)_{xx},$$ the negative of the Cahn-Hilliard operator, and
$$A_{_{2,\varepsilon}}(u):=\varepsilon^2 u_{xx} -
f(u),$$ the Allen-Cahn operator.

We shall then define the set of admissible layer positions by
$$\Omega_{\rho} = \bigg\{h = (h_1, \ldots, h_{\scriptscriptstyle N}): \;  \frac{\varepsilon}{2\rho} < h_1 < \cdots < h_N < 1-
\frac{\varepsilon}{2\rho},
 \quad \mbox{and} \quad \min\limits_{\scriptscriptstyle 2 \leq j \leq N}(h_{j}-h_{j-1}) > \frac{\varepsilon}{\rho}\bigg\}, $$
for some $ \rho $ small and independent of $ \varepsilon$, which
will be described in detail in the next section. Moreover, we
shall specify the $N$-dimensional manifold of approximate steady
states
$$ \mathcal M:=\{u^h: h \in \Omega_\rho\}.
$$ The residual $\upsilon$ of the approximation is defined as
orthogonal to $N$ approximate tangent vectors to $\mathcal{M}$ at
$u^h$.

The spectrum of the linearized CH/AC operator is determined at
Section \ref{specsec}, and as well the positive definition of the
induced bilinear form, when applied on the residual $\upsilon$ if
$\eps^2\mu(\eps)\geq C_0\delta(\eps)$ for some $C_0\geq C_{\min}>
0$ sufficiently large and specified through the supremum in
$(0,1)$ of $|\eps^{2}(f'(u^h))_{xx}|=\mathcal{O}(1)$, see Main
Theorem \ref{mainspec}. We also estimate the rest of the spectrum
in Theorem \ref{mainspec2}.

Section \ref{eqmonmc} presents the equations of motion through
the ode system \eqref{eq:2coord ODE 11} for the positions $h_i$,
$i=1,\cdots,N$, and then Main Theorem \ref{prop h velocity}
estimates the velocities $\dot{h}_i$ of the layers; at this
technical part, we followed the approach of Bates and Xun,
\cite{BatesXunI}. Finally at Section \ref{slowch}, after a rather
extensive calculus and using the spectral condition of the
linearized operator, we specify a wide class of $\mu(\eps)$,
$\delta(\eps)$, for which the dynamics are stable, and
exponentially small in $\eps$, cf. \eqref{eq222}, \eqref{eqasumpmd2}.

\vspace{0.6cm}

The second part of this manuscript is devoted, at Section
\ref{section:MSP}, to the mass conserving layer dynamics, and the
strategy applied is analogous to this in \cite{BatesXunI}.

Let $M$ be a fixed mass in $(-1,1)$. Restricting one degree of
freedom we impose a mass conservation property, and define the
second approximate $(N-1)$-dimensional manifold, which is a
submanifold of $\mathcal{M}$, by
$$ \mathcal M_1
:= \Big{\{}u^h \in \mathcal M: \int_0^1 u^h(x)\,
dx=M\Big{\}},
$$ and further define the manifold
$$ \tilde{\mathcal{M}}:=\Big{\{}\tilde{u}^h:\;u^h \in  \mathcal M_1,\;\;\tilde{u}^h(x)=\int_0^x u^h(y)\,
dy\Big{\}}.
$$

We impose the mass conservation condition
$\int_{0}^{1}u(x,t)dx=\int_{0}^{1}u_0(x)dx=$ fixed, in place of
the b.c. $u_{xxx}(1,t)=0$. Then, for the integrated equation, we
derive the CH/AC initial and boundary value problem given by
(IACH)-(IBC2). In the Appendix we discuss the well posedness of
the mass-conserving problem and derive \textit{a priori}
estimates by using the corresponding energy functional.

In Section \ref{eqmmc}, we specify the ode system for the
equations of motion of the $N-1$ layers and estimate their
dynamics, in the mass conservative case, see in Theorem \ref{prop
xi velocity}.

Finally, at Section \ref{slowmc}, we prove the Main Theorem
\ref{mtat}, establishing attractiveness of the manifold and
stability of the dynamics, again for a wide class of $\mu(\eps)$,
$\delta(\eps) $ (see \eqref{eq:a assumption B}), for which the dynamics are stable, and exponentially small in $\eps$.

\vspace{0.6cm}

We have also included, at the end, an Appendix where we collected
and proved various estimates used throughout the text.


%

\section{Non mass conserving layer dynamics}\label{nmcld}
We supplement \eqref{eq:deltaCH-AC} with the standard Neumann
b.c. on $u$ and $u_{xx}$, and consider the following initial and
boundary value problem
\begin{equation}\label{nmcp1}
\begin{split}
&u_t = - \delta(\eps) \, \big(\varepsilon^2 u_{xx} -
f(u)\big)_{xx} \; + \; \mu(\eps)(\varepsilon^2 u_{xx}  - f(u)) ,
\quad x\in
 (0,1),\quad t>0,\\
&u_{x}= u_{xxx}= 0 \quad \mbox{at}\quad x=0,1, \quad t>0,\\
&u(x,0)=u_0(x,\eps), \quad x\in(0,1).
\end{split}
\end{equation}
Let us point out that \eqref{nmcp1} does not conserve mass for any
$\mu(\eps)\neq 0$ since in general
$$\partial_t\int_0^1 u(x,t)dx=\int_0^1u_t(x,t)dx=\int_0^1\mu(\eps)(\varepsilon^2 u_{xx}(x,t)  - f(u(x,t))) dx=
-\mu(\eps)\int_0^1 f(u(x,t))dx,$$ is not the zero function, while
the solutions of the boundary problem
\begin{equation}\label{BVP: steady states}
\begin{split}
&\varepsilon^2 u_{xx} - f(u) = 0, \quad 0<x<1,
\\
&u_x = 0 \quad  \mbox{at} \quad x=0,1,
\end{split}
\end{equation}
are obviously steady states of \eqref{nmcp1}; \eqref{BVP: steady
states} follows by setting $ u_t = 0 $ at the equation of
\eqref{nmcp1} and integrating in space twice using the boundary
conditions.

We assume, therefore, in the context of Section \ref{nmcld}, that
$\mu(\eps)>0$ for all $\eps>0$.

\subsection{Free energy decreasing}
The relevant to the scaling of the standard Allen-Cahn operator
$$\eps^2\Delta u-f(u),$$ free energy functional is defined as
follows
\begin{equation}\label{energy}
E(u):=\int_0^1\Big{(}\frac{\eps^2|\nabla u|^2}{2}+F(u)\Big{)}dx.
\end{equation}

Multiplying both sides of the equation of the i.b.v.p
\eqref{eq:deltaCH-AC} with $\eps^2\Delta u-f(u)$, integrating in
space, and using the boundary conditions, we derive
\begin{equation}\label{d1}
\begin{split}
 (u_t,\eps^2\Delta u-f(u)) = &- \delta(\varepsilon)(\Delta\big(\varepsilon^2 \Delta u - f(u)\big),\eps^2\Delta u-f(u))
  + \mu(\varepsilon)\|\varepsilon^2 \Delta u - f(u)\|^2\\
  =&\delta(\varepsilon)\|\nabla\big(\varepsilon^2 \Delta u -
  f(u)\big)\|^2
  + \mu(\varepsilon)\|\varepsilon^2 \Delta u - f(u)\|^2,
\end{split}
\end{equation}
where here and for the rest of the manuscript, $(\cdot,\cdot)$
denotes the $L^2((0,1))$ inner product, and $\|\cdot\|$ the
induced $L^2((0,1))$ norm.

Differentiating in time, \eqref{energy} yields
\begin{equation}\label{energyder}
\begin{split}
\frac{\partial E(u)}{\partial t} =&\int_0^1\Big{(}\eps^2\nabla
u\nabla
u_t+F'(u)u_t\Big{)}dx\\
=&\int_0^1\Big{(}\eps^2\nabla u\nabla
u_t+f(u)u_t\Big{)}dx\\
=&\int_0^1\Big{(}-\eps^2\Delta
u u_t+f(u)u_t\Big{)}dx\\
=&-(u_t,\eps^2\Delta u-f(u)).
\end{split}
\end{equation}

So, by \eqref{d1} and \eqref{energyder}, we obtain the free
energy decreasing property for the combined model
\begin{equation}\label{d2}
\begin{split}
 \frac{\partial E(u)}{\partial t}= -\delta(\varepsilon)\|\nabla\big(\varepsilon^2 \Delta u -
  f(u)\big)\|^2
  - \mu(\varepsilon)\|\varepsilon^2 \Delta u - f(u)\|^2\leq 0,
\end{split}
\end{equation}
since $\delta(\varepsilon)> 0$ and $\mu(\varepsilon)\geq 0$.

Besides the three homogeneous equilibria $ u = \pm 1 $ and $u=u_0 $ for the zero $ u_0 $ of $ f $ in $(-1,1)$ ($u=0 $ for odd $ f $ as in our special case), problem \eqref{BVP: steady states} has non-constant solutions for all sufficiently small $ \varepsilon; $ see e.g. \cite{FH}. More precisely, if $ \varepsilon_{n+1} \leq \varepsilon < \varepsilon_n $ with $ \varepsilon_i:=(-f'(0))^{1/2}/2\pi i,\,i=1,2,\ldots, $ then problem \eqref{BVP: steady states} has exactly $n$ pairs of non-constant solutions $u^\pm_{\varepsilon_i} ,\, 1\leq i \leq n $ ($u^-_{\varepsilon_i}=-u^+_{\varepsilon_i}$ if $f$ is odd). For each $ i, $ the equilibria $u^-_{\varepsilon_i}, u^+_{\varepsilon_i} $ have exactly $ i $ zeros at $ x=1/2i, 3/2i, \ldots, 1-1/2i. $ The two solutions for $ i=1 $ are monotone, and the other solutions for $ i\geq 2 $ are oscillating taken as rescaled reflections and periodic extensions of monotone solutions, they correspond to periodic orbits around the origin on the phase plane and we speak of solutions with $i$ internal \textit{transition} layers.

Since the internal transition layers of stationary solutions must have periodic spacing, the solutions of \eqref{eq:deltaCH-AC} which reach patterns that are nearly piecewise constant, say with $N$ transition layers, but not periodic, they are close to a stationary state but they are
not solutions of the steady state problem and we do not expect them to remain at these patterns. We
will concern with these solutions yet not with their end-state but rather with their dynamics as long as
they remain at these ``metastable'' $N$-layered patterns.

\subsection{The approximate manifold}\label{section:constr manifold}
First we note that assumption \eqref{eq:W cond2} ensures the
existence of an $ a > 0 $ such that $ f'(u) > 0 \,$ for $\, |u
\pm 1| < a $, so let us fix such an $ a > 0. $

We will follow the strategy of the pioneering works of Carr and
Pego \cite{CarrPego}, and Bates and Xun \cite{BatesXunII} for the
construction of the approximate manifold solutions; we refer also
to the work of Antonopoulou, Bl\"omker, Karali
\cite{AntonBlomkerKarali} for the Cahn-Hilliard equation with
noise where this approach was applied effectively in the
stochastic setting as well.

For initial condition
$u_0(x;\eps)=u(x,0)$ close to the manifold $\mathcal{M}$, 
we will approximate the profile of a metastable state of the
solution $u$ of \eqref{nmcp1} with $ N $ transition layers by
piecing together the stationary solutions of \eqref{eq:deltaCH-AC}
satisfying the following boundary value Dirichlet problem for the
bistable equation,
\begin{equation}\label{eq: Bistabe BVP}
\begin{dcases}
\varepsilon^2 \phi_{xx} - f(\phi) = 0 , \qquad |x| <
\frac{\ell}{2} + \varepsilon,
\\
\phi\Big(\mp\frac{\ell}{2}\Big)  = 0,
\end{dcases}
\end{equation}
with $ \ell > 0 $ denoting the distance between two successive
layer positions.

\begin{remark}\label{ff}
We summarize briefly the properties of the solutions of
\eqref{eq: Bistabe BVP}, as established in \cite{CarrPego}, cf.
Prop. 2.1. therein:

There exists $ \rho_0 > 0 $ such that if $ \varepsilon/\ell <
\rho_0, $ then
\begin{itemize}
\item[(i)] a unique solution $ \phi_{\scriptscriptstyle\varepsilon}(x, \ell , +1) $ of \eqref{eq: Bistabe BVP} exists, with
$$ \phi_{\scriptscriptstyle\varepsilon}(x, \ell , +1) > 0 \quad  \mbox{for} \quad   |x| < \ell/2, \, \qquad  \mbox{and}
\qquad |\phi_{\scriptscriptstyle\varepsilon}(0, \ell , +1) - 1| <
a,$$
\item[(ii)] a unique solution $ \phi_{\scriptscriptstyle\varepsilon}(x, \ell , -1) $ of \eqref{eq: Bistabe
BVP} exists, with $$ \phi_{\scriptscriptstyle\varepsilon}(x, \ell
, -1) < 0 \quad  \mbox{for} \quad  |x| < \ell/2, \, \qquad
\mbox{and} \qquad |\phi_{\scriptscriptstyle\varepsilon}(0, \ell ,
-1) + 1| < a.$$
\end{itemize}

Moreover, the functions $ \phi_{\scriptscriptstyle\varepsilon} $
are smooth and depend on $ \varepsilon $ and $ \ell $ only
through the ratio $ \varepsilon/\ell.$
\end{remark}

\subsection{The approximate solution}

Let us consider a smooth cut-off function satisfying
\begin{equation}\label{eq: cutoff X}
\chi:\mathbb{R} \to [0,1] \qquad\mbox{with}\quad \chi(x)=0
\quad\mbox{for}\quad  x\leq -1, \qquad\mbox{and}\quad  \chi(x)=1
\quad\mbox{for}\quad  x\geq 1.
\end{equation}

Given a choice of admissible layer positions $ h  = (h_1, \ldots,
h_{\scriptscriptstyle N}) \in \Omega_{\rho} \,, $ let
\begin{equation}\label{eq: def ellj}
\begin{split}
&\ell_j = h_j - h_{j-1} \qquad \mbox{for} \quad j = 2, 3, \ldots, N, \qquad \mbox{and} \qquad \ell_1 = 2 h_1, \;\;\ell_{N+1} = 2 (1 - h_N),\\
& m_j =  \frac{h_{j-1}+h_j}{2} \qquad \mbox{for} \quad j = 2, 3, \ldots,
N, \qquad \mbox{and} \qquad
 m_1 = 0, \;\; m_{N+1} = 1,\\
& I_j =  [m_j, \, m_{j+1}] \qquad \mbox{for} \quad j = 1, 2, \ldots, N.
\end{split}
\end{equation}

We define the approximate solution $ u^h $\, for any \, $ x \in
I_j $, by
\begin{equation}\label{eq: def of uh}
\begin{split}
u^h(x) = &\bigg[1-\chi\Big(\frac{x-h_j}{\varepsilon}\Big)\bigg] \,
\phi_{\scriptscriptstyle\varepsilon}(x-m_j ,\, \ell_j , \, (-1)^j)
\\
&+ \chi\Big(\frac{x-h_j}{\varepsilon}\Big) \,
\phi_{\scriptscriptstyle\varepsilon}(x-m_{j+1} , \, \ell_{j+1}
,\, (-1)^{j+1}).
\end{split}
\end{equation}

In order to ease notation, we suppress the dependence of $ u^h $
on $ \varepsilon $ and omit hereafter the subscript in $
\phi_\varepsilon $  by simply writing $ \phi$, and define
\begin{equation}\label{eq: abbrev phij}
 \phi^j(x) := \phi_{\scriptscriptstyle\varepsilon}(x-m_j,\,h_j-h_{j-1},\, (-1)^{j}),
\end{equation}
for $ \phi_{\scriptscriptstyle\varepsilon}$ the steady states
presented analytically in Remark \ref{ff}.

Moreover, we define
\begin{equation}\label{eq: abbrev chij}
\chi^j(x) :=   \chi\big(\frac{x-h_j}{\varepsilon}\big).
\end{equation}

The profile of $u^h$ is presented at Figure \ref{fig: uh}.
\begin{figure}[H]
\begin{center}
\begin{tikzpicture}[scale=0.37]

\draw[mark size=1.5pt, black] (-21,0) node {\tiny$\scriptscriptstyle\mid$};
\draw[] (-21,0) node[below right=-0.2em] {\scriptsize$h_0$};

\draw[mark size=1.5pt, black] (-17.5,0) node {\tiny$\scriptscriptstyle\mid$};
\draw[] (-17.5,0) node[above right=-0.2em] {\scriptsize$m_1$};
\draw[] (-17.5,0) node[below left=-0.2em] {\scriptsize$0$};

\draw[dotted,thin] (-12, pi/2-0.1) -- (-17.5, pi/2-0.1) node[left, black] {\scriptsize{$1$}} ;
\draw[black]  (-17.5, pi/2-0.1) node[] {{$\bm{-}$}};

\draw[dotted,thin] (-12, -pi/2+0.12) -- (-17.5, -pi/2+0.12) node[below left, black] {\scriptsize{$-1$}};
\draw[black]  (-17.5,-pi/2+0.12) node[] {{$\bm{-}$}};

\draw[densely dashed, thin, domain=-22:-19.4,smooth,variable=\x] plot ({\x}, { - rad(atan(6*(\x + 21))) });

\draw[densely dashed, very thin] (-19.4, -pi/2+0.12) -- (-17.5, -pi/2+0.12) ;
\draw[black,thick] (-17.4, -pi/2+0.12) -- (-15.5, -pi/2+0.12) ;

\draw[thick,domain=-15.5:-12.3,smooth,variable=\x] plot ({\x}, { rad(atan(6*(\x + 14)))   }) ;
\draw[mark size=1.3pt, black] (-14,0) node {\tiny$\scriptscriptstyle\mid$};
\draw[] (-14,0) node[below right=-0.3em] {\scriptsize$h_1$};

\draw[black,thick] (-12.3, pi/2-0.1) -- (-10.5, pi/2-0.1) ;

\draw[thick,domain=-10.5:-7,smooth,variable=\x] plot ({\x}, { - rad(atan(6*(\x + 9))) }) ;
\draw[mark size=1.3pt, black] (-11.5,0) node {\tiny$\scriptscriptstyle\mid$};
\draw[] (-11.5,0) node[above=-0.2em] {\scriptsize$m_2$};
\draw[mark size=1.3pt, black] (-9,0) node {\tiny$\scriptscriptstyle\mid$};
\draw[] (-9,0) node[below left=-0.3em] {\scriptsize$h_2$};

\draw[thick] (0.5, pi/2-0.1) -- (4.5, pi/2-0.1) ;

\draw[thick,domain=-3:0.5,smooth,variable=\x] plot ({\x}, {  rad(atan(6*(\x +1)))  }) ;
\draw[mark size=1.5pt, black] (-1,0) node {\tiny$\scriptscriptstyle\mid$};
\draw[] (-1,0) node[below right=-0.3em] {\scriptsize$h_{j-1}$};
\draw[mark size=1.5pt, black] (2.5,0) node {\tiny$\scriptscriptstyle\mid$};
\draw[] (2.5,0) node[above=-0.3em] {\scriptsize$m_j$};

\draw[thick,domain=4.5:7.5,smooth,variable=\x] plot ({\x}, { - rad(atan(6*(\x - 6)))  }) ;
\draw[mark size=1.5pt, black] (6,0) node {\tiny$\scriptscriptstyle\mid$};
\draw[] (6,0) node[below left=-0.3em] {\scriptsize$h_j$};

\draw[thick] (7.5, -pi/2+0.12) -- (8.6, -pi/2+0.12) ;

\draw[thick, domain=8.6:12,smooth,variable=\x] plot ({\x}, {  rad(atan(6*(\x - 10))) });
\draw[mark size=1.5pt, black] (10,0) node {\tiny$\scriptscriptstyle\mid$};
\draw[] (10,0) node[below right=-0.3em] {\scriptsize$h_{j+1}$};
\draw[mark size=1.5pt, black] (8,0) node {\tiny$\scriptscriptstyle\mid$};
\draw[] (8,0) node[above=-0.3em] {\scriptsize$m_{j+1}$};

\draw[thick, domain=15:18.5,smooth,variable=\x] plot ({\x}, {  rad(atan(6*(\x - 17))) });
\draw[mark size=1.3pt, black] (17,0) node {\tiny$\scriptscriptstyle\mid$};
\draw[] (17,0) node[below left=-0.1em] {\scriptsize$h_{N}$};

\draw[thick] (18.5, pi/2-0.11) -- (19.5, pi/2-0.11) ;
\draw[densely dashed, thin] (19.5, pi/2-0.11) -- (20.5, pi/2-0.11) ;

\draw[densely dashed, very thin, domain=20.5:22.1,smooth,variable=\x] plot ({\x}, {  -rad(atan(6*(\x - 22))) });
\draw[densely dashed, very thin, domain=22.5:24,smooth,variable=\x] plot ({\x}, {  -rad(atan(6*(\x - 22))) });

\draw[mark size=1.3pt, black] (19.5,0) node {\tiny$\scriptscriptstyle\mid$};
\draw[] (19.5,0) node[above=-0.2em] {\scriptsize$m_{N+1}$};
\draw[] (19.5,0) node[below=-0.1em] {\scriptsize$1$};
\draw[mark size=1.3pt, black] (22,0) node {\tiny$\scriptscriptstyle\mid$};
\draw[] (22,0) node[below=-0.15em] {\scriptsize$h_{N+1}$};

\draw [draw=none, fill=gray!30!white, opacity=0.6] (-14.5, -pi/2) -- (-13.5, -pi/2) -- (-13.5, pi/2) -- (-14.5, pi/2)  -- cycle; 
\draw [draw=none, fill=gray!30!white, opacity=0.6] (-9.5, -pi/2) -- (-8.5, -pi/2) -- (-8.5, pi/2) -- (-9.5, pi/2)  -- cycle; 
\draw [draw=none, fill=gray!30!white, opacity=0.6] (-1.5, -pi/2) -- (-.5, -pi/2) -- (-.5, pi/2) -- (-1.5, pi/2)  -- cycle; 
\draw [draw=none, fill=gray!30!white, opacity=0.6] (5.5, -pi/2) -- (6.5, -pi/2) -- (6.5, pi/2) -- (5.5, pi/2)  -- cycle; 
\draw [draw=none, fill=gray!30!white, opacity=0.6] (9.5, -pi/2) -- (10.5, -pi/2) -- (10.5, pi/2) -- (9.5, pi/2)  -- cycle; 
\draw [draw=none, fill=gray!30!white, opacity=0.6] (16.5, -pi/2) -- (17.5, -pi/2) -- (17.5, pi/2) -- (16.5, pi/2)  -- cycle; 

\draw[decorate,decoration={brace,raise=1pt}, thin] (-0.5, pi/2) -- (5.5, pi/2) node[pos=0.5,above=2pt]   {\scriptsize$u^h=\phi^j$};
\draw[decorate,decoration={mirror,brace,raise=2pt}, thin] (6.5, -pi/2) -- (9.5, -pi/2) node[pos=0.5,below=3pt]   {\scriptsize$u^h=\phi^{j+1}$};

\draw[decorate,decoration={mirror, brace,raise=2pt} , very thin] (-14.5, -pi/2) -- (-13.5, -pi/2) node[pos=0.5,below=2pt]   {\scriptsize$2\varepsilon$};  

\draw[decorate,decoration={mirror, brace,raise=2pt} , very thin] (-9.5, -pi/2) -- (-8.5, -pi/2) node[pos=0.5,below=2pt]   {\scriptsize$2\varepsilon$}; 

\draw[decorate,decoration={mirror, brace,raise=2pt} , very thin] (-1.5, -pi/2) -- (-0.5, -pi/2) node[pos=0.5,below=2pt]   {\scriptsize$2\varepsilon$}; 


\draw[decorate,decoration={brace,raise=2pt} , very thin] (9.5, pi/2) -- (10.5, pi/2) node[pos=0.5,above=2pt]   {\scriptsize$2\varepsilon$};  

\draw[decorate,decoration={mirror, brace,raise=2pt} , very thin] (16.5, -pi/2) -- (17.5, -pi/2) node[pos=0.5,below=2pt]   {\scriptsize$2\varepsilon$};  

\draw[mark size=1.3pt, black] (-14,0) node {\tiny$\scriptscriptstyle\mid$};
\draw[] (-14,0) node[below right=-0.3em] {\scriptsize$h_1$};

\draw[mark size=1.3pt, black] (-9,0) node {\tiny$\scriptscriptstyle\mid$};
\draw[] (-9,0) node[below left=-0.3em] {\scriptsize$h_2$};

\draw[mark size=1.5pt, black] (-1,0) node {\tiny$\scriptscriptstyle\mid$};
\draw[] (-1,0) node[below right=-0.3em] {\scriptsize$h_{j-1}$};

\draw[mark size=1.5pt, black] (6,0) node {\tiny$\scriptscriptstyle\mid$};
\draw[] (6,0) node[below left=-0.3em] {\scriptsize$h_j$};

\draw[mark size=1.5pt, black] (10,0) node {\tiny$\scriptscriptstyle\mid$};
\draw[] (10,0) node[below right=-0.3em] {\scriptsize$h_{j+1}$};

\draw[mark size=1.3pt, black] (17,0) node {\tiny$\scriptscriptstyle\mid$};
\draw[] (17,0) node[below left=-0.1em] {\scriptsize$h_{N}$};

\draw[black,thick] (-22,0) -- (-7,0) ;
\draw[black, very thick, dotted] (-6.5, 0) -- (-3.5, 0) ;
\draw[black,thick] (-3,0) -- (12,0) ;
\draw[black, very thick, dotted] (12.5, 0) -- (14.5, 0) ;
\draw[mark size=1.5pt, thick, ->] (15,0) -- (24,0) node[above] {$\bm{x}$};
\draw[mark size=1.5pt, thick, ->] (-17.5,-3) -- (-17.5,4) node[right] {$\bm{u^h}$};
\end{tikzpicture}
\end{center}
\caption{Given a configuration $ h=(h_1,\cdots, h_N) $ of $ N $ layer positions, we construct $ u^h $ by piecing together
steady state solutions of \eqref{eq: Bistabe BVP}. 
We set
 $ u^h = (1 - \chi^j) \phi^j + \mathcal \chi^j \phi^{j+1}$  on $ [h_j-\varepsilon,  h_j+\varepsilon] $ (shaded areas). Note that $ h_0=-h_1, \, h_{N+1}=2-h_N, $ and $ m_1=0, m_{N+1}=1.$}
\label{fig: uh}
\end{figure}
\subsection{Properties of $u^h$}
Note that $ u^h $ is a smooth function of $ x $ and $ h. $ In
particular,  for $ x \in [m_j, \, m_{j+1}], $ we have
\begin{equation}\label{eq: derivative of uh}
u_x^h=
\begin{cases}
\; \phi_x^j , &\quad  m_j \leq x \leq h_j-\varepsilon ,
\\
\chi_x^j \, \big(\phi^{j+1}-\phi^j\big)  + \big(1-\chi^j\big) \,
\phi_x^j +  \chi^j \, \phi_x^{j+1}, &\quad |x-h_j| < \varepsilon ,
\\
\; \phi_x^{j+1} , &\quad   h_j +\varepsilon \leq x \leq m_{j+1},
\end{cases}
\end{equation}
and so
\begin{equation}\label{eq: uxxh}
u_{xx}^h =
\begin{cases}
\; \phi_{xx}^j , &\quad  m_j \leq x \leq h_j-\varepsilon ,
\\
\chi_{xx}^j \, \big(\phi^{j+1}-\phi^j\big)  + 2 \chi_x^j \,
\big(\phi_x^{j+1}-\phi_x^j\big) + (1-\chi^j)\phi^j_{xx}
 + \chi^j\phi^{j+1}_{xx}, &\quad   |x-h_j| < \varepsilon ,
\\
\; \phi_{xx}^{j+1} , &\quad   h_j +\varepsilon \leq x \leq m_{j+1},
\end{cases}
\end{equation}
where, from the definition of $ \phi^j  $ (see \eqref{eq: Bistabe BVP}),
\begin{equation}\label{eq: phij solves BS}
\varepsilon^2 \, \phi_{xx}^j \, = \,  f(\phi^j), \qquad
\mbox{on} \quad [h_{j-1} - \varepsilon, \, h_j+\varepsilon].
\end{equation}

It is then straightforward to see that $ u^h $ satisfy the bistable equation
\begin{equation}\label{eq: uh sat BS}
\mathscr L^b(u^h)=0,\qquad \mbox{for} \quad |x-h_j| \geq \varepsilon, \;\; j = 1, 2, \ldots, N.
\end{equation}
where $ \mathscr L^b $ is the bistable operator
\begin{equation}\label{eq:bistable operator}
\mathscr L^b(u) := A_{_{2,\varepsilon}}(u) = \varepsilon^2
u_{xx}  -  f(u).
\end{equation}

Notice also that by reflecting the solutions $ \phi(\cdot, \ell ,
\pm1) $ of \eqref{eq: Bistabe BVP} about the origin we can show
that they are even in $ x, $ and thus
\begin{equation}\label{eq: phix equals 0}
\phi_x(0, \ell , \pm1) = 0,
\end{equation}
which together with \eqref{eq: def of uh} and \eqref{eq: cutoff
X} yields
$$ u^h_x(m_j)=0, \; j = 1, \ldots, N+1.$$ So, we obtain
\begin{equation}\label{eq: zero Neumann BV uh}
u^h_x(0)=u^h_x(1)=0,
\end{equation}
and, therefore,
\begin{equation}\label{eq: zero BV Wprimeuhx}
\frac{\partial}{\partial x}f(u^h) = 0 \qquad\mbox{at \, $ x=0,1$}.
\end{equation}

Moreover, we note that
$$u_j^h \sim -u_x^h, \qquad\mbox{as} \quad r \to 0, \qquad \mbox{uniformly on} \quad I_j :=[m_j,\,
m_{j+1}],$$ see Remark \ref{remujh} for more details thereof.

\subsection{The coordinate system}
Following \cite{CarrPego}, we introduce a local coordinate system
$$ u\mapsto (h, \upsilon), $$ in a tubular neighborhood of the
approximate manifold $ \mathcal M, $ defined by the decomposition
\begin{equation}\label{eq:coord decomp2}
u(x, t) = u^{h(t)}(x) + \upsilon(x, t),
\end{equation}
with $ u^h \in \mathcal M $, and $ \upsilon $ satisfying the
orthogonality condition
 \begin{equation}\label{orthogonality condition2}
\langle \upsilon, \tau_j^h\rangle := \int_0^1 \upsilon \, \tau_j^h \, dx = 0,  \qquad j=1,\ldots, N,
\end{equation}
where $ \tau_j^h $ are approximate tangent vectors  to $ \mathcal
M\, $ at $ u^h $.

More precisely, let
\begin{equation}\label{eq:def gamma j}
\gamma^j(x):=\chi\big(\tfrac{x-m_j-\varepsilon}{\varepsilon}\big)
\Big[1-\chi\big(\tfrac{x-m_{j+1}+\varepsilon}{\varepsilon}\big)\Big],
\end{equation}
which yields that
\begin{equation*}\label{eq:def2 gamma j}
\gamma^j(x)
=
\begin{dcases}
0, & x \notin (m_j, \, m_{j+1}),
\\
1, & x \in [m_j+2\varepsilon, \, m_{j+1}-2\varepsilon].
\end{dcases}
\end{equation*}
The approximate tangent vectors are then defined through
$\gamma^j(x)$ by
\begin{equation}\label{eq:def tau j}
\tau_j^h(x) \, = \, \gamma^j(x) \,  u^h_x(x)
\end{equation}
which are smooth functions of $ x $ and  $ h $.

Considering differentiation, we introduce the notation
$$ \tau_{j,k}^h := \frac{\partial \tau_j^h}{\partial h_k} \qquad\mbox{and}\qquad  \tau_{j,x}^h := \frac{\partial \tau_j^h}{\partial x}, $$
and observe that $ \tau_{j,x}^h=0 $ at $ x=0,1, $ for $
j=1,\ldots, N. $

\subsection{Equations of motion}\label{eqmonmc} For a classical solution $ u=u(x,t) $ of \eqref{eq:deltaCH-AC},
we will establish the ODEs system for the dynamics of $ (h,
\upsilon), $ which is defined by \eqref{eq:coord
decomp2}-\eqref{orthogonality condition2}.

This first order ODE system with unknowns the positions
coordinates $h_k(t)$, $k=1,\cdots,N$ for each one of the $N$
layers (fronts) will be derived by differentiating in time the
orthogonality condition \eqref{orthogonality condition2}. We will
insert the Cahn-Hilliard / Allen-Cahn equation into the
differentiated condition and then we shall use linearization which
will be given by the linear combination through the weights
$\delta(\eps)$, $\mu(\eps)$ of the C-H and A-C linearized
operators respectively.

We differentiate \eqref{orthogonality condition2}, with respect to $ t, $  to get
$$ \big\langle  \partial_t \upsilon \; , \; \tau_j^h  \big\rangle
\,+\, \big\langle \upsilon \; , \;  \partial_t \tau_j^h
\big\rangle =  0 , \qquad j=1,\ldots, N. $$ Using
\eqref{eq:coord decomp2}, and the substituting $u_t$ by the
equation \eqref{eq:deltaCH-AC}, together with the definition
\eqref{eq:def of Aepsilon2} of the operator $A_\varepsilon$, we
obtain
\begin{equation*}
\begin{split}
\partial_t \upsilon=&\partial_t (u - u^h)\\
=&- {\deps}
\Big(\varepsilon^2 u_{xx} - f(u)\Big)_{xx} +
\mu(\eps)(\varepsilon^2 u_{xx}  - f(u))-\partial_t u^h)-\partial_t u^h\\
=& A_\varepsilon u - \partial_t u^h = A_\varepsilon u -
\sum_{k=1}^N (\partial_{\!_{\scriptscriptstyle
h_k}}\!\!u^h)\dot{h}_k,
\end{split}
\end{equation*}
to arrive at the system
\begin{equation}\label{eq:coord ODE 1}
\sum\limits_{k=1}^N a_{jk} \, \dot{h}_k
\; = \;
\big\langle A_{\varepsilon}(u^h +  \upsilon) \; , \; \tau_j^h\; \big\rangle ,\qquad j=1,2,\ldots,
 N,
\end{equation}
where
\begin{equation}\label{eq:1 the coefficient matrix}
a_{jk} := \big\langle u^h_k \; , \; \tau_j^h \big\rangle \; - \;
\big\langle \upsilon  \; , \; \tau^h_{j,k}\, \big\rangle
,\qquad j,k=1,2,\ldots, N.
\end{equation}
In the above, the subscripts $ k $ indicate differentiation
w.r.t. $ h_k, $ i.e.,
$$ u^h_k := \frac{\partial u^h}{\partial h_k}
\qquad \mbox{and} \qquad
\tau^h_{j,k} := \frac{\partial \tau^h_j}{\partial h_k}.$$

We expand
\begin{eqnarray}\label{eq:expand2 Ae}
A_\varepsilon(u^h + \upsilon) & = & A_\varepsilon(u^h) \,+\,
L_\varepsilon^h(\upsilon) \,-\, {\deps} \, \big(f^h \,
\upsilon^2\big)_{xx} \,+\,
 \mu(\eps) f^h \, \upsilon^2,
\end{eqnarray}
where $ L_\varepsilon^h(\upsilon) $ is the linearization of
$A_\eps$ at $u^h$, i.e.,
\begin{equation}\label{eq: linearized operator2}
\begin{split}
L_\varepsilon^h(\upsilon) :=& - \deps \Big( \varepsilon^2
\upsilon_{xx} -  f'(u^h)\upsilon \Big)_{xx}
+\mu(\eps)(\varepsilon^2 \upsilon_{xx}  -  f'(u^h)\upsilon)\\
=:&-\deps
L_{1,\varepsilon}^h(\upsilon)+\mu(\eps)L_{2,\varepsilon}^h(\upsilon),
\end{split}
\end{equation}
and
\begin{equation}\label{eq:definition2 of f2}
 f^h(x) := \int_0^1 (\tau - 1) \; f''(u^h + \tau \upsilon) \; d\tau.
\end{equation}

Using \eqref{eq:expand2 Ae}, then the system \eqref{eq:coord ODE
1} is written as
\begin{equation}\label{eq:2coord ODE 11}
\begin{split}
\sum\limits_{k=1}^N a_{jk} \, \dot{h}_k
 \; =& \;
 \big\langle A_\varepsilon(u^h)  \; , \; \tau_j^h\;  \big\rangle
\; + \;
  \big\langle L_\varepsilon^h(\upsilon)   \; , \; \tau_j^h\;  \big\rangle
\; \\
&- \; {\deps} \,
 \big\langle \big( f^h \, \upsilon^2\big)_{xx}  \; , \; \tau_j^h\;  \big\rangle
\; + \;
 \mu(\eps)\big\langle  f^h \, \upsilon^2\; , \; \tau_j^h\;
 \big\rangle\\
 =&\,
- {\deps} \, \big\langle A_{1,\varepsilon}(u^h)
 +
L_{1,\varepsilon}^h(\upsilon)
 +
\big( f^h \, \upsilon^2\big)_{xx} \; , \; \tau_j^h
\big\rangle \; \\
&+ \mu(\eps)\big\langle A_{2,\varepsilon}(u^h) +
L_{2,\varepsilon}^h(\upsilon)
 +
 f^h \, \upsilon^2
\; , \; \tau_j^h\big\rangle\\
=&\, - {\deps} \, \big\langle
A_{1,\varepsilon}(u^h),\tau_j^h\big\rangle+
\mu(\eps)\big\langle
A_{2,\varepsilon}(u^h),\tau_j^h\big\rangle\\
 &- {\deps} \,
\big\langle L_{1,\varepsilon}^h(\upsilon)
 +
\big( f^h \, \upsilon^2\big)_{xx} \; , \; \tau_j^h
\big\rangle \; \\
&+ \mu(\eps)\big\langle L_{2,\varepsilon}^h(\upsilon)
 +
 f^h \, \upsilon^2
\; , \; \tau_j^h\big\rangle,
 \end{split}
\end{equation}
for $  j=1,2,\ldots, N$.

\begin{remark}\label{remexp2}
We note that as in the Allen-Cahn case, $\upsilon$-independent
exponentially small terms in the dynamics \eqref{eq:2coord ODE
11} will be derived by the term
$$\mu(\eps)\big\langle A_{2,\varepsilon}(u^h), \tau_j^h\big\rangle,$$ due to the second order
operator there. Of course $\mu(\eps)$ will influence the result.



More specifically, as in \cite{CarrPego} Lemma 3.3., when $\rho$
in the definition of $\Omega_\rho$ is sufficiently small, we
observe that
\begin{equation}\label{fine}
\begin{split}
\mu(\eps)\big\langle A_{2,\varepsilon}(u^h) \; , \;
\tau_j^h\big\rangle=&-\mu(\eps)\int_{h_j-\eps}^{h_j+\eps}(\eps^2u_{xx}^h-f(u^h))u_x^hdx\\
=&\mu(\eps)[F(u^h)-\frac{1}{2}\eps^2(u_x^h)^2]_{h_j-\eps}^{h_j+\eps}=:\mu(\eps)(a^{j+1}-a^j),
\end{split}
\end{equation}
where the difference $a^{j+1}-a^j$ is exponentially small in
$\eps$.

In our case, we also have the $\upsilon$-independent term
$$\big\langle
A_{1,\varepsilon}(u^h),\tau_j^h\big\rangle,$$ stemming from
the Cahn-Hilliard part, which will be shown exponentially small
as well.
\end{remark}

Moreover, we apply \eqref{eq:coord decomp2} to \eqref{eq:deltaCH-AC} to get 
\begin{equation}\label{eq:coord ODE 2}
\upsilon_t = A_\varepsilon(u^h + \upsilon) - \sum\limits_{j=1}^N  u_j^{h}\,\dot{h}_j.
\end{equation}
As above, we expand in \eqref{eq:coord ODE 2} the term $ A_\varepsilon(u^h + \upsilon), $ according to \eqref{eq:expand2 Ae}, to get
\begin{equation}\label{eq:2coord ODE 22}
\upsilon_t \, = \, A_\varepsilon(u^h) \,+\,
L_\varepsilon^h(\upsilon) \,-\, {\deps} \, \big( f^h \,
\upsilon^2\big)_{xx} \,+\,\mu(\eps)
 f^h \, \upsilon^2
 \,-\,
 \sum\limits_{j=1}^N u_j^{h}\,\dot{h}_j
\end{equation}
or discriminating between the CH and AC parts,
\begin{equation}\label{eq:2coord ODE 2}
\begin{split}
\upsilon_t \, = \, &- {\deps} \, \bigg[ A_{1,\varepsilon}(u^h) \,
+ \, L_{1,\varepsilon}^h(\upsilon) \, + \, \big( f^h \,
\upsilon^2\big)_{xx} \bigg] \,\\
& + \, \mu(\eps)\bigg[ A_{2,\varepsilon}(u^h) \, + \,
L_{2,\varepsilon}^h(\upsilon) \, + \, f^h \, \upsilon^2\bigg] \,
- \, \sum\limits_{j=1}^N u_j^{h}\,\dot{h}_j.
\end{split}
\end{equation}

According to Proposition 2.3 in \cite{CarrPego}, 
there exist $ \rho_2 > 0, $ constants $ A_0, $ $ C $ and $ b(\rho)
= o(1) $ as $ \rho\to 0^+ $ such that if $ h \in \Omega_{\rho} $
with $ \rho < \rho_2, $ then:
\begin{itemize}
\item[(i)] for $ j=1,2,\cdots, N, $
\begin{eqnarray}
\big(A_0 - b(\rho)\big)^2 \varepsilon^{-1}
\, \leq \,
\big\langle u_j^h \; , \; \tau_j^h \big\rangle
\, \leq \,
\big(A_0 + b(\rho)\big)^2 \varepsilon^{-1}
\label{eq: CP 2.8}
\\
\|\tau_{jj}^h\|  \leq  C \varepsilon^{-3/2}  \qquad \mbox{and} \qquad
\|\tau_{jj}^h\|_{L^1} & \leq & C \varepsilon^{-1}
\label{eq: CP 2.9}
\end{eqnarray}

\item[(ii)] for $ j\neq k, $
\begin{eqnarray}\label{eq: CP 2.6}
\big|
\big\langle u^h_j \; , \; \tau_k^h \big\rangle
\big|
\, \leq \, b(\rho) \varepsilon^{-1},
\label{eq: CP 2.10}
\\
\|\tau_{kj}^h\|  \leq  b(\rho) \varepsilon^{-3/2}  \qquad \mbox{and} \qquad
\|\tau_{kj}^h\|_{L^1} & \leq & b(\rho) \varepsilon^{-1} .
\label{eq: CP 2.11}
\end{eqnarray}
\end{itemize}

\subsection{The spectrum of the linearized operator and transverse coercivity}\label{specsec}
A key point in the analysis of the equations of motion \eqref{eq:2coord ODE 11},\eqref{eq:2coord ODE 22} is the spectral analysis of the linear operator $ L^h_\eps $ on $ \upsilon = u- u^h, $ the part of the solution transverse to the manifold $ \mathcal M.$

We
consider the eigenvalue problem
\begin{equation}\label{The original EVP problem}
\begin{dcases}
L_{\varepsilon}^h(\phi):= - \deps \big(\varepsilon^2 \phi'' -
f'(u^h) \phi \big)'' + \mu(\eps)\big(\varepsilon^2 \phi''  -
f'(u^h)\phi\big) = \lambda(\varepsilon) \phi, \quad 0 < x < 1 ,
\\
\phi'(0) = \phi'(1) = 0,
\\
\phi'''(0) = \phi'''(1) = 0,
\end{dcases}
\tag{EVP}
\end{equation}
and the associated quadratic form
\begin{eqnarray}
\tilde{\mathcal A}_{\varepsilon}[\phi] & :=&  - \langle
L_\varepsilon^h(\phi) \, , \, \phi \rangle \nonumber
\\
& = & \int_0^1
\left[
{\deps}
\left(\varepsilon^2 \phi_{xx}^2
\, + \, f'(u^h)\, \phi_x^2 \,  - \, \tfrac{1}{2} \,
\big(f^{'}(u^h)\big)_{xx} \phi^2 \right)
\, + \mu(\eps)\left(
\varepsilon^2 \phi_x^2 \,+\, f'(u^h)\phi^2\right)
\right] \, dx,
\label{eq: define Atilde}
\end{eqnarray}
where we have applied integration by parts. We note that the operator $ L_{\varepsilon}^h:
\mathcal{D}(L_{\varepsilon}^h) \to L^2(0,1) $, with domain
$$ \mathcal{D}(L_{\varepsilon}^h) := \Big\{\phi\in H^4(0,1):\; \phi'(0) = \phi'(1) = 0 = \phi'''(0) = \phi'''(1)\Big\}, $$
is not symmetric in $ L^2(0,1). $ However \eqref{The original EVP problem} may be recast into a self-adjoint form, as the form $ \tilde{\mathcal A}_{\varepsilon} $ it will be seen in
the sequel 
to be lower semibounded and it is also closable since it is associated with the symmetric operator $
S_\varepsilon^h:\mathcal{D}(S_\varepsilon^h) \to L^2(0,1)$,
defined by
\begin{equation}\label{eq:def of S self adjoint}
S_\varepsilon^h(\phi) := - \deps \Big( \varepsilon^2 \phi''''
\; - \; \big(f'(u^h) \phi' \big)' - \frac{1}{2}
\big(f^{'}(u^h)\big)_{xx} \phi \Big) + \mu(\eps)\Big(
\varepsilon^2 \phi''  -  f'(u^h)\phi\Big),
\end{equation}
with domain $ \mathcal{D}(S_\varepsilon^h) \equiv
\mathcal{D}(L_{\varepsilon}^h), $ in the sense that
\begin{equation}
\langle L_{\varepsilon}^h(\phi) \, , \, \phi \rangle
=
-\tilde{\mathcal A}_{\varepsilon}[\phi]
= \langle S_\varepsilon^h(\phi) \, , \, \phi \rangle, \qquad
\mbox{for any} \quad \phi \in \mathcal{D}(L_{\varepsilon}^h).
\end{equation}
We can then consider the self-adjoint extension (Friedrichs extension) of $ S_\varepsilon^h $ associated with the closure of $ \tilde{\mathcal A}_{\varepsilon} $ which we still denote by $ S_\varepsilon^h. $ The spectrum of $ S_\varepsilon^h $ turns out to consist of a sequence of real eigenvalues $$\cdots
\leq \lambda_{N+1} \leq \lambda_N \leq \dots \leq \lambda_1,$$ satisfying, when
$\eps^{2}\mu(\eps)\geq\mathcal{O}(\delta(\eps)),$
$$\lambda_k\leq C\mu(\eps),$$
and in particular for $ k = N+1, N+2,\cdots$
$$\lambda_k\leq -C=-C(\eps)< 0,$$
for some $ C=C(\eps) $ which will be specified. To this aim, we will use the variational characterization of the eigenvalues for $ S_\varepsilon^h,$
\begin{equation}\label{The minimization problem}
-\lambda_{N+1} := \max\limits_{\substack{V \\\scriptscriptstyle \operatorname{dim}\,V=N}}  \min\limits_{\phi \in V^\bot} \frac{\tilde{\mathcal A}_{\varepsilon}[\phi] }{\|\phi\|^2},
\end{equation}
where the maximum is taken over the linear $N$-dimensional
subspaces $ V $ of
\begin{equation}\label{The spaces V}
\big\{ \phi \in H^2(0,1): \, \phi'(0) = \phi'(1) =  0\big\}.
\end{equation}

Let us decompose the Rayleigh quotient in \eqref{The minimization
problem} into the Cahn-Hilliard and the Allen-Cahn part, i.e.
\begin{equation}\label{decomposition of Reyleigh quotient}
R(\phi) := \frac{\langle -S_\varepsilon^h(\phi), \phi\rangle
}{\|\phi\|^2} =
 {\deps} R_1(\phi)
\, + \mu(\eps) R_2(\phi),
\end{equation}
for
\begin{equation}\label{eq: the summands of Reyleigh quotient}
\begin{split}
&R_1(\phi) := \frac{\displaystyle \int_0^1 (\varepsilon^2
\phi_{xx}^2 \, + \, f'(u^h)\, \phi_x^2 \,  - \, \tfrac{1}{2} \,
\big(f^{'}(u^h)\big)_{xx} \phi^2) \, dx  }{\|\phi\|^2} ,\\
&R_2(\phi) := \frac{ \overbrace{\int_0^1 ( \varepsilon^2 \phi_x^2
\,+\, f'(u^h)\phi^2) \, dx}^{\scriptscriptstyle \langle
L_{2,\varepsilon}^h \phi, \phi\rangle } }{\|\phi\|^2}.\quad
\end{split}
\end{equation}
In the sequel we will assume that $ \delta(\eps)> 0 $ and $ \mu(\eps) > 0 $  satisfy the condition
\begin{equation}\label{eq222}
\eps^2\mu(\eps)\geq C_0\delta(\eps)\qquad\;\forall\eps>0,
\end{equation}
for some $ C_0\geq C_{\min} > 0 $, where $C_{\min}$ depends on the
supremum of $|\eps^{2}(f'(u^h))_{xx}|=\mathcal{O}(1)$ in $(0,1) $ and can be determined implicitly in the context of the proof of following theorem.

We shall show that when $ \delta(\eps), \mu(\eps)$ satisfy \eqref{eq222}, then
$$ - \lambda_{N+1} \geq
C(\eps)=\mathcal{O}(\mu(\eps)-C_{\min}\delta(\eps)\eps^{-2})> 0.
$$
For our purpose, we state the result in the variational form. The main implication for us is the transverse coercivity (cf. Lemma \ref{lemma BleqA}), the coercivity of $ L_{\varepsilon}^h $ on $ \upsilon=u-u^h  $ which recall that is approximately orthogonal to the tangent space (see \eqref{orthogonality condition2}).
\begin{theorem}\label{mainspec}
Let $\delta(\eps)>0$ and $\mu(\eps)> 0$ satisfying
\eqref{eq222}. Then there exist $ \Lambda, \rho_2>0 $ such that for $ h \in \Omega_\rho $ with $ \rho\leq \rho_2 $ it holds that
$$\min\limits_{ \phi \in V^\bot} \frac{-\langle
S_\varepsilon^h(\phi) \, , \, \phi\rangle }{\|\phi\|^2} \geq
\eta(\eps) >
0,$$ for $ V = \operatorname{span}\{\tau_i^h, \; i = 1, 2,
\ldots, N\}, $ with $ \tau_i^h $ are the approximate tangent
vectors defined in \eqref{eq:def tau j}, and
\begin{equation}\label{def of etaeps}
\eta(\eps) := \Lambda (\mu(\eps)-C_{\min}\delta(\eps)\eps^{-2}).
\end{equation}
\end{theorem}
\begin{proof}
Remind that $f(u)=u^3-u$, and so $f'(u)=3u^2-1\geq -1$.

Moreover, it holds that $$|(f^{'}(u^h))_{xx}|\leq C
\eps^{-2}\qquad\mbox{in} \quad [0,1],$$ since there
$$|u^h_{x}| \leq C
\varepsilon^{-1},\qquad|u^h_{xx}| \leq C \varepsilon^{-2},$$ see
in Appendix relations \eqref{eq: uhx leq eps}, and \eqref{eq: uhxx
leq eps}, while $u^h$ is bounded.

Using that $\delta(\eps)>0$, $\mu(\eps)> 0$, and the above, we
get for any $\phi \in V^\bot$
\begin{equation}\label{est1mainspec}
\begin{split}
R(\phi) :=&
\frac{\langle -S_\varepsilon^h(\phi), \phi\rangle
}{\|\phi\|^2}
\\
=&
{\deps} \frac{\displaystyle \int_0^1 (\varepsilon^2
\phi_{xx}^2 +f'(u^h)\phi_x^2- \tfrac{1}{2}
\big(f^{'}(u^h)\big)_{xx} \phi^2)dx  }{\|\phi\|^2}
\\
&+ \mu(\eps)\frac{\int_0^1  (\varepsilon^2 \phi_x^2 +
f'(u^h)\phi^2)dx}{\|\phi\|^2}
\\
\geq &-{\deps} \frac{\displaystyle \int_0^1\phi_x^2 dx
}{\|\phi\|^2}+ \mu(\eps)\frac{\int_0^1  (\varepsilon^2 \phi_x^2 +
f'(u^h)\phi^2)dx}{\|\phi\|^2}
\\
&-{\deps} \frac{\displaystyle \int_0^1 \tfrac{1}{2}
\big(f^{'}(u^h)\big)_{xx} \phi^2dx  }{\|\phi\|^2}
\\
\geq &-{\deps} \frac{\displaystyle \int_0^1\phi_x^2 dx
}{\|\phi\|^2}+ \mu(\eps)\Lambda\frac{\int_0^1  (\varepsilon^2
\phi_x^2 +
\phi^2)dx}{\|\phi\|^2}
\\
&-\frac{{\deps}}{2}\displaystyle{\sup_{(0,1)}}|(f^{'}(u^h))_{xx}|\\
\geq &\frac{\displaystyle \int_0^1 [-{\deps}+
\mu(\eps)\Lambda\eps^2]\phi_x^2 dx }{\|\phi\|^2}+
\mu(\eps)\Lambda-\displaystyle{\sup_{(0,1)}}|\eps^2(f^{'}(u^h))_{xx}|\frac{\deps}{2}\eps^{-2}
\\
= &
\Lambda\Big(\mu(\eps)\eps^2 - \frac{1}{\Lambda} {\deps}\Big)
\frac{\|\phi_x\|^2 }{\|\phi\|^2} + \Lambda\eps^{-2}\Big(\mu(\eps)\eps^2 - \frac{C_1}{2 \Lambda} {\deps}\Big)
\end{split}
\end{equation}
where we applied the Main Theorem 4.2 (i) of \cite{CarrPego}, at
pg. 538, and defined
$$C_1:=\displaystyle{\sup_{(0,1)}}|\eps^2(f^{'}(u^h))_{xx}|=\mathcal{O}(1).$$
More specifically, we used that for some $\Lambda>0 $ and $ \rho_2>0 $
$$ \int_0^1  (\varepsilon^2
\phi_x^2 + f'(u^h)\phi^2)dx\geq \Lambda\int_0^1  (\varepsilon^2
\phi_x^2 +\phi^2)dx$$
for $ h \in \Omega_\rho $ with $ \rho\leq \rho_2. $ Here, $\Lambda$ satisfies
$$\Lambda\leq \frac{\Lambda_*}{\Lambda_*+2},$$ (see \cite{CarrPego}, at pg. 541,
and use the specific $f$), with $\Lambda_*$ arbitrary in
$(0,\lambda_*)$ for some $\lambda_*>0$ such that $$\lambda_*\leq
\min f'(\pm 1)=2,$$ see \cite{CarrPego}, at pg. 536, for the
detailed definition.

Therefore, in \eqref{est1mainspec}, if we choose $\mu(\eps), \delta(\eps)$
such that
$$\eps^2\mu(\eps)\geq C_0\delta(\eps),$$for
$$C_0\geq C_{\min}:=\max\Big{\{}\frac{1}{\Lambda},\;\frac{C_1}{2\Lambda}\Big{\}}+\beta^2,$$
for $\beta=\mathcal{O}(1)\neq 0$ as small, we obtain
$$\frac{\langle -S_\varepsilon^h(\phi), \phi\rangle
}{\|\phi\|^2}\geq
C(\eps)=\Lambda(\mu(\eps)-C_{\min}\delta(\eps)\eps^{-2})> 0,$$
for any $\phi \in V^\bot$, and thus the result.
\end{proof}
The next theorem, gives an upper bound for the spectrum.
\begin{theorem}\label{mainspec2}
Let $\delta(\eps)>0$ and $\mu(\eps)>0$ satisfying \eqref{eq222},
$$\eps^2\mu(\eps)\geq
C_0\delta(\eps)\qquad\;\forall\eps>0,$$ for $C_0\geq C_{\min}> 0$
as defined in Theorem \ref{mainspec}. Then, for any $ k $ it holds that
$$ \lambda_k \leq C\mu(\eps),$$
for some $C>0$.
\end{theorem}
\begin{proof}
Using that $\delta(\eps)>0$, $\mu(\eps)> 0$, and $f(u)=u^3-u$,
$f'(u)=3u^2-1\geq -1$, we get for any $\phi$, and since $C_0\geq
\frac{1}{\Lambda}>1$
\begin{equation}
\begin{split}
 R(\phi) := &\frac{\langle -S_\varepsilon^h(\phi), \phi\rangle
}{\|\phi\|^2}\\
 =&
 {\deps} \frac{\displaystyle \int_0^1 (\varepsilon^2
\phi_{xx}^2 +f'(u^h)\phi_x^2- \tfrac{1}{2}
\big(f^{'}(u^h)\big)_{xx} \phi^2)dx  }{\|\phi\|^2}\\
&+ \mu(\eps)\frac{\int_0^1  (\varepsilon^2 \phi_x^2 +
f'(u^h)\phi^2)dx}{\|\phi\|^2}\\
\geq &{\deps} \frac{\displaystyle \int_0^1 (-\phi_x^2-
\tfrac{1}{2}
\big(f^{'}(u^h)\big)_{xx} \phi^2)dx  }{\|\phi\|^2}\\
&+ \mu(\eps)\frac{\int_0^1  (\varepsilon^2 \phi_x^2
-\phi^2)dx}{\|\phi\|^2}\\
= &\frac{\displaystyle \int_0^1 (-{\deps} +
\mu(\eps)\eps^2)\phi_x^2dx}{\|\phi\|^2}- \frac{\displaystyle
\int_0^1[\mu(\eps)+{\deps} \tfrac{1}{2}\big(f^{'}(u^h)\big)_{xx}]
\phi^2dx}{\|\phi\|^2}\\
\geq&- \frac{\displaystyle \int_0^1[\mu(\eps)+{\deps}
\tfrac{1}{2}\big(f^{'}(u^h)\big)_{xx}] \phi^2dx}{\|\phi\|^2}\\
\geq&-\mu(\eps)-C_1{\deps} \tfrac{1}{2}\eps^{-2}.
\end{split}
\end{equation}

This yields that for all $k=1,\cdots,N$
$$\lambda_k\leq \mu(\eps)+C_1{\deps} \tfrac{1}{2}\eps^{-2}\leq
C\mu(\eps).$$
\end{proof}

For $ \upsilon \in C^2([0,1]) $ with $ \upsilon_x=0 $ at $ x=0,
1, $ we define the forms
\begin{equation}\label{eq: define
Atilde again}
\begin{split}
 \tilde{\mathcal A}_{\varepsilon}[\upsilon] :=
\int_0^1 \Big{[} &{\deps}
 \Big( \varepsilon^2 \upsilon_{xx}^2
+ f'(u^h)\, \upsilon_x^2 \,  - \, \tfrac{1}{2} \,
\big(f^{'}(u^h)\big)_{xx} \upsilon^2 \Big)\\
&+\mu(\eps)\Big( \varepsilon^2 \upsilon_x^2 \,+\,
f'(u^h)\upsilon^2\Big)\Big{]} \, dx,
\end{split}
\end{equation}
(see \eqref{eq:def of S self adjoint}, \eqref{eq: define Atilde})
and
\begin{eqnarray}
\tilde{\mathcal B}_{\varepsilon}[\upsilon] := \int_0^1
\Big{[}{\deps}
 \varepsilon^2 \upsilon_{xx}^2as well as
\, + \, \, \big( {\deps} \, + \mu(\eps)\varepsilon^2 \big)
\upsilon_x^2 \, + \, \big( {\deps} \, + \, \mu(\eps) \big)
\upsilon^2\Big{]} \, dx. \label{eq: define Btilde}
\end{eqnarray}

\begin{lemma}\label{lemma BleqA}
There is a $ \rho_0 > 0 $ such that if $ 0 < \rho < \rho_0 $ and $ h \in \Omega_\rho, $ then for any $ \upsilon \in C^2 $ with $ \upsilon_x=0 $ at $ x=0, 1 $ and $ \langle\upsilon, \tau_j^h\rangle = 0, \, j=1,\cdots, N,$
\begin{equation}\label{ineq: BleqA}
\tilde{\mathcal B}_\eps[\upsilon] \leq C \tilde{\mathcal A}_{\eps}[\upsilon]
\end{equation}
for some positive constant $ C, $ and
\begin{equation}\label{ineq: AleqL1eps}
\eta(\eps)  \tilde{\mathcal A}_{\eps}[\upsilon]
\; \leq \;
\|S_\eps^h(\upsilon)\|^2.
\end{equation}
\end{lemma}
\begin{proof}By Theorem 4.2 in \cite{CarrPego} there exists $ \Lambda>0 $ such that
\begin{equation}
\int_0^1  \eps^2 \upsilon_x^2 \,+\, f'(u^h)\upsilon^2
\, dx
\geq \Lambda
\int_0^1 \eps^2 \upsilon_x^2 \,+\, \upsilon^2.
\end{equation}
For such $ \Lambda, $ and $ c:=\max\{c_1, c_2\} $ for positive constants of the uniform bounds
$$ |\big(f'(u^h)\big)_{xx}| \leq c_1 \eps^{-2}
\qquad\mbox{and}\qquad
|f'(u^h)| \leq c_2
\qquad\mbox{and}\qquad [0,1],
$$as well as
we have
\begin{eqnarray*}
\tilde{\mathcal A}_{\eps}[\upsilon]
& := &
\int_0^1
 {\deps}
 \Big( \eps^2 \upsilon_{xx}^2
\, + \,
f'(u^h)\, \upsilon_x^2 \,  - \, \tfrac{1}{2} \, \big(f'(u^h)\big)_{xx} \upsilon^2
\Big)
\, + \,
\meps \Big(\eps^2 \upsilon_x^2 \,+\, f'(u^h)\upsilon^2\Big)
\, dx
\\[0.5em]as well as
& = &
\int_0^1
{\deps}
\eps^2 \upsilon_{xx}^2
\, dx
\, + \,
{\deps} \int_0^1
f'(u^h)\, \upsilon_x^2 \,  - \, \tfrac{1}{2} \, \big(f'(u^h)\big)_{xx} \upsilon^2
\, dx
\, + \,
\meps \int_0^1  \eps^2 \upsilon_x^2 \,+\, f'(u^h)\upsilon^2
\, dx
\\[0.5em]
& \geq &
\int_0^1
{\deps}
\eps^2 \upsilon_{xx}^2
\, dx
\, - \,
c \,{\deps}
\int_0^1  \upsilon_x^2 \,  + \, \eps^{-2} \upsilon^2
\, dx
\, + \,
\meps \, \Lambda
\int_0^1  \eps^2 \upsilon_x^2 \,+\, \upsilon^2
\, dx
\\[0.5em]
& = &
\int_0^1
{\deps}
 \eps^2 \upsilon_{xx}^2
\, + \,
\,
\big( \meps\, \Lambda \, \eps^2 \, - \, c \, {\deps} \big) \upsilon_x^2
\, + \,
\big(
\meps \, \Lambda \, - \, c \, {\deps} \eps^{-2}
\big) \upsilon^2
\; dx
\\[0.5em]
& \geq &
C \,
\int_0^1
{\deps}
 \eps^2 \upsilon_{xx}^2
\, + \,
\,
\big( {\deps} \, + \, \meps\, \eps^2 \big) \upsilon_x^2
\, + \,
\big( {\deps} \, + \, \meps \big) \upsilon^2
\, dx
\\[0.5em]
& =: &
C \, \tilde{\mathcal B}_{\eps}[\upsilon]
\end{eqnarray*}
for some positive constant $ C $ small enough; the last inequality follows from the assumption \eqref{eq222}, since
$$
\meps\, \Lambda \, \eps^2 \, - \, c \, {\deps}
\geq
C
\left(
{\deps} \, + \, \meps\, \eps^2
\right)
\Longleftrightarrow
\meps\, \eps^2 \, \geq \, \frac{c + C}{\Lambda  - C} \, {\deps}
$$
and
$$
\meps \, \Lambda \, - \, c \, {\deps} \eps^{-2}
\geq
C
\big( {\deps} \, + \, \meps \big)
\Longleftrightarrow
\meps\, \eps^2 \, \geq \, \frac{c + C \eps^2}{\Lambda  - C} \, {\deps}.
$$
Regarding inequality \eqref{ineq: AleqL1eps}, we use the estimate
\begin{equation}\label{ineq: minmax princ tA}
0 < C(\eps) \leq \frac{\tilde{\mathcal A}_\eps[\upsilon]}{\|\upsilon\|^2} 
\qquad
\forall \upsilon \;\;\mbox{with}\;\;\upsilon_x(0)=\upsilon_x(1)=0 \;\;\mbox{and}\;\; \langle \upsilon, \tau_j^h\rangle = 0, \quad j = 1, \cdots, N,
\end{equation}
for $ C(\eps) = \eta(\eps) :=\Lambda (\meps - C_{\rm min}\deps \eps^{-2}), $ to get immediately
\begin{equation}
\bm{\big|} \tilde{\mathcal A}_\eps[\upsilon] \bm{\big|}
\; = \;
\bm{|} \langle S_\eps^h(\upsilon) \, , \, \upsilon\rangle \bm{|}
\; \leq \;
\| S_{\eps}^h(\upsilon) \| \; \| \upsilon \|
\stackrel{\scriptscriptstyle \eqref{ineq: minmax princ tA}}{\leq}
\frac{1}{\sqrt{C}}
\cdot
\| S_{\eps}^h(\upsilon) \|
\cdot
\bm{\big|} \tilde{\mathcal A}_\eps[\upsilon] \bm{\big|}^{1/2}
\end{equation}
and hence \eqref{ineq: AleqL1eps} follows.
\end{proof}

\subsection{Flow near layered equilibria }
The main result of this section is Theorem \ref{prop h velocity} regarding the motion of the layers.
For $ \upsilon \in C^1[0,1] $ with $ \upsilon_x(0)=\upsilon_x(1)=0, $ we introduce the norm
\begin{equation}\label{eq: def of B upsilon}
\mathcal B_\varepsilon[\upsilon] := \int_0^1 \big[ \varepsilon^2 \upsilon_{x}^2 + \upsilon^2 \big] \, dx
\end{equation}
and we will study the orbit 
$ \, u(x, t) = u^{h(t)}(x) + \upsilon(x, t), \, $ of \eqref{eq:deltaCH-AC} as long as
\begin{equation}\label{eq: assumption1 B upsilon}
\mathcal B_\varepsilon[\upsilon] \leq C\eps= \mathcal
O(\varepsilon).
\end{equation}
\begin{lemma}\label{lemma 4.1 BX}For $ \upsilon \in C^1[0,1] $ we have
\begin{equation}\label{eq: 5 proof for Bups}
\| \upsilon \|^2_{L^\infty} \leq \frac{1+\varepsilon}{\varepsilon} \mathcal B_\varepsilon[\upsilon].
\end{equation}
\end{lemma}
\begin{proof}
Let $ x_1\in [0,1] $ be such that
\begin{equation}\label{eq: 1  proof for B upsilon}
\upsilon^2(x_1) = \| \upsilon \|^2_{L^\infty}
\end{equation}
and let $ x_2 \neq x_1 $ be such that
\begin{equation}\label{eq: 2  proof for B upsilon}
\upsilon^2(x_2) \leq \mathcal B_\varepsilon[\upsilon].
\end{equation}
We can assume that $ x_2 \leq x_1 $ without loss of generality, for otherwise we would consider the reflection
of $ \upsilon $ about $ \tfrac{1}{2} $ which would then satisfy this assumption with $ \|\cdot\|_{L^\infty},
\, \mathcal B_\varepsilon[\cdot] $ remaining invariant.

Integrating the inequality
$$ \varepsilon \frac{d}{dx} \upsilon^2 = 2 \varepsilon \upsilon \upsilon_{x} \leq \varepsilon^2 \upsilon^2_{x} + \upsilon^2 $$
on $ [x_2, x_1], $ we obtain
\begin{equation}\label{eq: 4  proof for B upsilon}
\varepsilon \upsilon^2(x_1) - \varepsilon \upsilon^2(x_2) \leq \mathcal B_\varepsilon[\upsilon]
\end{equation}
hence \eqref{eq: 5 proof for Bups} results upon the substitution
of \eqref{eq: 1  proof for B upsilon}-\eqref{eq: 2  proof for B
upsilon} into \eqref{eq: 4  proof for B upsilon}.
\end{proof}

For the coefficients $ a_{jk} $ of $ \dot{h}_k $ in the LHS of
\eqref{eq:coord ODE 1}, defined in \eqref{eq:1 the coefficient
matrix}, we introduce the matrices $ S(h), \hat{S}(h,\upsilon), $
\begin{equation}\label{eq:def of matrices}
 S(h) = \big(S(h)\big)_{jk} := \big\langle u^h_k \; , \; \tau_j^h \big\rangle
\qquad\mbox{and}\qquad
 \hat{S}(h,\upsilon) = \big(\hat{S}(h,\upsilon)\big)_{jk} :=   \big\langle \upsilon  \; , \; \tau^h_{j,k}\, \big\rangle.
\end{equation}

According to Lemma 3.1 of \cite{CarrPego}, 
there exist $ \sigma_1, \rho_1>0 $ such that if $ \sigma <
\sigma_1 $ and $ \rho<\rho_1 $ and $ (h, \upsilon) \in  \mathscr
S_{\!\rho,\sigma} $, then the matrices $ S(h) $ and $ S(h)-
\hat{S}(h, \upsilon) $ are invertible with
\begin{equation}\label{eq:norm of inv matrices}
\|S^{-1}\| \leq 4 A_0^{-2} \varepsilon
\qquad\mbox{and}\qquad
\|(S- \hat{S})^{-1}\| \leq 8 A_0^{-2} \varepsilon.
\end{equation}
Here, $ \|\cdot\| $ denotes the matrix norm induced by the
vector norm $ \|h\|=\max_j|h_j|$ on $ \mathbb R^N, $ and $
A_0 $ is the constant appearing in \eqref{eq: CP 2.8}.

\begin{theorem}\label{prop h velocity}
There exists $ \rho_2 > 0$ and a constant $ C>0 $, such that
\begin{equation}\label{ineq: estimate hdot}
\begin{split}
| \dot{h}_i| \leq &C {\deps} \Big( \varepsilon^{-1} \alpha(r)
\,+\,
\varepsilon^{-5/2} \mathcal B^{1/2}_{\!\varepsilon}[\upsilon]
 \,+\,
 \varepsilon^{-3} \mathcal B_\varepsilon[\upsilon]
\Big)
\\
&+C \mu(\eps)\Big{(}\varepsilon \alpha(r) \,+\,
 \varepsilon^{1/2} \mathcal B_\varepsilon^{1/2}[\upsilon] \big[\alpha(r)+\beta(r)\big]
 \,+\,
\varepsilon^{-1} \mathcal B_\varepsilon[\upsilon]\Big{)},
\end{split}
\end{equation}
as long as $ h \in \Omega_{\rho} $, with $ \rho < \rho_2 $, and
the orbit $ \, u(x, t) = u^{h(t)}(x) + \upsilon(x, t) \, $ of
\eqref{eq:deltaCH-AC} remains sufficiently close to $ \mathcal M $
so that $\mathcal B_\varepsilon[\upsilon] = \mathcal
O(\varepsilon)$.
\end{theorem}
\begin{proof}
The RHS of \eqref{eq:2coord ODE 11} with $$ \tau_j^h(x) =
-\gamma^j(x) u^h_x(x),\; j=1,2,\ldots,N, $$ is written as
\begin{equation}\label{eq: 1 flow in channel}
\begin{split}
&- {\deps} \, \big\langle A_{1,\varepsilon}(u^h)
 +
L_{1,\varepsilon}^h(\upsilon)
 +
\big( f^h \, \upsilon^2\big)_{xx} \; , \; \gamma^j u^h_x
\big\rangle \; \\
&+ \mu(\eps) \big\langle A_{2,\varepsilon}(u^h) +
L_{2,\varepsilon}^h(\upsilon)
 +
 f^h \, \upsilon^2
\; , \; \gamma^j u^h_x\big\rangle.
\end{split}
\end{equation}

We will consider the first summand in \eqref{eq: 1 flow in
channel} that comes from the CH part, and shall estimate the terms
\begin{equation}\label{eq: 2 flow in channel}
\underbrace{ \hspace{-0.4em} \big\langle
A_{1,\varepsilon}(u^h) \; , \; \gamma^j u^h_x \big\rangle
\hspace{-0.4em} }_{T_1} \qquad + \qquad \underbrace{
\hspace{-0.4em} \big\langle L_{1,\varepsilon}^h(\upsilon)
\;,\; \gamma^j u^h_x \big\rangle \hspace{-0.4em} }_{T_2}
\qquad + \qquad \underbrace{ \hspace{-0.4em} \big\langle
\big( f^h \, \upsilon^2\big)_{xx} \; , \; \gamma^j u^h_x
\big\rangle, \hspace{-0.4em} }_{T_3}
\end{equation}
for $ j=1,2,\ldots, N, \, $ where $ f^h $ is given in
\eqref{eq:definition2 of f2}; here, recall the notation
$$
A_{1,\varepsilon}(u) = \big(\varepsilon^2 u_{xx} -
f(u)\big)_{xx}, \qquad\mbox{and}\qquad
 L_{1,\varepsilon}^h(\upsilon) = \big(
\varepsilon^2 \upsilon_{xx} -  f'(u^h)\upsilon \big)_{xx}.
$$
\\
\vspace{0.5cm}
\\
\textit{Estimate of $ \mathit{T_1} $:}
\\
For all $ j, $ we have $$ \mathscr L^b(\phi^j)=0, $$ which in view
of the definition \eqref{eq: def of uh}, implies that $$
A_{1,\varepsilon}(u^h) = 0 \quad \mbox{for}\quad |x-h_j| \geq
\varepsilon.$$ Using that $$ \gamma^j(x) = 1 \quad \mbox{for} \quad
|x-h_j| \leq \varepsilon, $$ we get
\begin{equation}\label{eq: 3 flow in channel}
T_1 := \big\langle A_{1,\varepsilon}(u^h) \; , \; \gamma^j
u^h_x \big\rangle = \int_{h_j-\varepsilon}^{h_j+\varepsilon}
\Big( \underbrace{\varepsilon^2 u^h_{xx} - f(u^h)
}_{\scriptscriptstyle \mathscr L^b(u^h)}\Big)_{xx} \, u^h_x \, dx.
\end{equation}
In the above, we integrate by parts twice, and obtain
\begin{equation}\label{eq: 4 flow in channel}
T_1
=
\int_{h_j-\varepsilon}^{h_j+\varepsilon}
\mathscr L^b(u^h)  \, u^h_{xxx} \, dx.
\end{equation}
We apply \eqref{ineq: uh3x leq eps3} together with \eqref{ineq:
Lbuh leq alpha} in \eqref{eq: 4 flow in channel}, and derive
\begin{equation}\label{eq: estimate of T1}
|T_1| \, \leq \, C \, \varepsilon^{-2} \, \alpha(r).
\end{equation}
\\
\vspace{0.5cm}
\\
\textit{Estimate of $ \mathit{T_2} $:}
\\
Using $ \, \tau_j^h:=\gamma^j u^h_x\, $ we integrate by parts
four times the first term and twice the second one, in the
definition of $ \mathit{T_2} $, to get
\begin{subequations}\label{eq:100}
\begin{align}
T_2
&
:=
\big\langle
L_{1,\varepsilon}^h(\upsilon)
\;,\;
\tau_j^h
\big\rangle
\nonumber
\\
&
=
-\varepsilon^2
\big\langle
 \upsilon_{xxxx}
\;,\; \tau_j^h \big\rangle \;+\; \big\langle
(f'(u^h)\upsilon)_{xx} \;,\; \tau_j^h \big\rangle \nonumber
\\
&
=
-\varepsilon^2
\big\langle
\upsilon
\;,\;
 (\tau_j^h)_{xxxx}
\big\rangle \;+\; \big\langle \upsilon \;,\;
f'(u^h)(\tau_j^h)_{xx} \big\rangle \nonumber
\\
\label{eq:100a}
&
=
-\varepsilon^2
\Big[
\big\langle
\upsilon
\;,\;
 \gamma^j (u^h_x)_{xxxx}
\big\rangle
\;
\begin{aligned}[t]
&
+ \;
4 \big\langle
\upsilon
\;,\;
\gamma^j_x (u^h_x)_{xxx}
\big\rangle
\; + \;
 6 \big\langle
\upsilon
\;,\;
\gamma^j_{xx} (u^h_x)_{xx}
\big\rangle
\\
&
+ \;4 \big\langle
\upsilon
\;,\;
\gamma^j_{xxx} (u^h_x)_x
\big\rangle
\; + \;
\big\langle
\upsilon
\;,\;
\gamma^j_{xxxx} u^h_x
\big\rangle
\Big]
\end{aligned}
\\
& \phantom{=} \; + \; \big\langle \upsilon \;,\; \gamma^j
f'(u^h) u^h_{xxx} \big\rangle \;+\; 2 \big\langle
\upsilon \;,\; \gamma^j_{x} f'(u^h) u^h_{xx} \big\rangle \;
\; + \; \big\langle \upsilon \;,\; \gamma^j_{xx}
f'(u^h)u^h_x \big\rangle. \label{eq:100b}
\end{align}
\end{subequations}
Here, we have used that $ \tau^h_j, (\tau^h_j)_x,(\tau^h_j)_{xx}
$ vanish at $ x=0,1.$

We now proceed to pointwise estimates for the terms involving $
u^h $ in \eqref{eq:100}, within the interval $ [m_j, m_{j+1}] $
for a fixed yet arbitrary $ j=1,\ldots,N. $

First notice that
\begin{equation}\label{eq: 100.3}
\varepsilon^2 (u^h_x)_{xx}\, - \, f'(u^h)u^h_x \;=\;
\frac{d}{dx}\mathscr L^b(u^h),
\end{equation}
and by \eqref{eq: uh sat BS},
\begin{equation}\label{eq: 100.5}
\frac{d}{dx}\mathscr L^b(u^h) = 0\qquad\mbox{except in} \quad [h_j -
\varepsilon , \, h_j + \varepsilon],
\end{equation}
while, in view of \eqref{eq:def2 gamma j},
\begin{equation}\label{eq: 100.7}
\gamma^j_x=\gamma^j_{xx} = \gamma^j_{xxx} = 0
\qquad \mbox{in} \quad [m_j + 2 \varepsilon , \, m_{j+1} - 2
\varepsilon] \supset [h_j - \varepsilon , \, h_j + \varepsilon].
\end{equation}
Therefore the last term in \eqref{eq:100b} is canceled out by the last one in \eqref{eq:100a}, and \eqref{eq:100} becomes
\begin{align}
\label{eq:100apali}\tag{\ref{eq:100a}$'$}
T_2
& =
-\varepsilon^2
\Big[
\big\langle
\upsilon
\;,\;
 \gamma^j (u^h_x)_{xxxx}
\big\rangle
\;
\begin{aligned}[t]
&
+ \;
4 \big\langle
\upsilon
\;,\;
\gamma^j_x (u^h_x)_{xxx}
\big\rangle
\; + \;
5
\big\langle
\upsilon
\;,\;
\gamma^j_{xx} (u^h_x)_{xx}
\big\rangle
\nonumber
\\
&
+ \;
4
\big\langle
\upsilon
\;,\;
\gamma^j_{xxx} (u^h_x)_x
\big\rangle
\; + \;
\big\langle
\upsilon
\;,\;
\gamma^j_{xxxx} u^h_x
\big\rangle
\Big]
\end{aligned}
\\
& \phantom{=} \; \; + \; \big\langle \upsilon \;,\;
\gamma^j  f'(u^h) u^h_{xxx} \big\rangle \;+\; 2
\big\langle \upsilon \;,\; \gamma^j_{x} f'(u^h) u^h_{xx}
\big\rangle . \label{eq:100bpali}\tag{\ref{eq:100b}$'$}
\end{align}

We have
\begin{equation}\label{eq: gamma derivative order}
\Big| \frac{d^n \gamma^j}{dx^n} \Big|  \leq C \varepsilon^{-n}.
\end{equation}
Moreover, differentiating \eqref{eq: phij solves BS} twice, and
then using \eqref{eq: gamma derivative order}, \eqref{eq: 100.7},
\eqref{ineq: cons2 of lemma 7.5 CP}, we get  for $
|x-m_j|<2\varepsilon,$
\begin{equation}\label{eq:104.5 estimate T2}
\varepsilon^2 | \gamma^j_x (u^h_x)_{xxx} | \,=\, \varepsilon^2 \,
| \gamma^j_x| \, | f''(u^h) (u^h_x)^2 \,+\, f'(u^h) u^h_{xx} |
\leq C \varepsilon^{-1}\beta(r).
\end{equation}
By \eqref{eq: gamma derivative order}, \eqref{eq: def of uh},
\eqref{eq: phij solves BS} and for $ x \in
[m_j,m_j+2\varepsilon]\cup [m_{j+1} - 2\varepsilon, m_{j+1}], $
we derive
\begin{equation}\label{eq:106}
\varepsilon^{2} |\gamma^j_{xx} (u^h_x)_{xx}| \,= \,
\varepsilon^{2} \, |\gamma^j_{xx}|\, |f'(u^h)u^h_x| \leq C
\varepsilon^{-1}\beta(r).
\end{equation}
Using  \eqref{eq: gamma derivative order}, \eqref{eq: 100.7} and
the estimate \eqref{ineq: cons2 of lemma 7.5 CP} presented in the
Appendix, we get
\begin{equation}\label{eq:106.5}
\varepsilon^2 |\gamma^j_{xxx} (u^h_x)_x| \,= \,
|\gamma^j_{xxx}|\, |f(u^h)| \leq C \varepsilon^{-3}\beta(r),
\end{equation}
and
\begin{equation}\label{eq:106.8}
|\gamma^j_{x} \, f'(u^h) \, u^h_{xx}| \leq C
\varepsilon^{-3}\beta^2(r).
\end{equation}
By \eqref{eq: gamma derivative order}, \eqref{eq: derivative of
uh},  \eqref{eq: 100.7} and the estimate \eqref{ineq: lemma 7.5
CP} (see Appendix), we obtain
\begin{equation}\label{eq:107}
\varepsilon^{2} |\gamma^j_{xxxx} u^h_x| \leq
C \varepsilon^{-3}\beta(r)
.
\end{equation}

We will see that the dominant asymptotics of $ \mathit{T_2} $
comes from the term
\begin{equation}\label{eq:100.1 estimate T2}
\big\langle \upsilon \;,\;
 \gamma^j \big(-\varepsilon^2 (u^h_x)_{xxxx} + f'(u^h) u^h_{xxx}\big)
\big\rangle,
\end{equation}
which is the leading term of $ T_2. $ This term as we will
see has order $\mathcal{O}(\eps^{-5/2}\|\upsilon\|)$.

We may proceed as previously, to get pointwise estimates on $
[m_j, m_j+2\varepsilon]\cup [m_{j+1}-2\varepsilon, m_{j+1}] $
 wherein the contribution is in minor order in the asymptotics of $ T_2, $ while, as we shall see, the main order comes from the term
\begin{eqnarray}\label{eq:109 estimate T2}
T_{2,1} \; := \; -\varepsilon^2 (u^h_x)_{xxxx} + f'(u^h)
u^h_{xxx} \qquad \mbox{in} \quad [m_j+2\varepsilon, \,
m_{j+1}-2\varepsilon].
\end{eqnarray}
Differentiating the third derivative of $u^h$, given by \eqref{eq:
u xxx form}, twice, we obtain
\begin{flalign}\label{eq: u 5x form}
\frac{\partial^5 u^h }{\partial x^5} =
\begin{cases}
&\phi_{xxxxx}^j, \qquad\mbox{for}\qquad m_j \leq x \leq
h_j-\varepsilon ,
\\\\
&\chi_{xxxxx}^j \, \big(\phi^{j+1}-\phi^j\big) + 5 \chi_{xxxx}^j
\, \big(\phi^{j+1}_x-\phi^j_x\big) + 10 \chi_{xxx}^j \,
\big(\phi_{xx}^{j+1}-\phi_{xx}^j\big)
\\
&+ 10 \chi_{xx}^j \, \big(\phi_{xxx}^{j+1}-\phi_{xxx}^j\big) + 10
\chi_{x}^j \, \big(\phi_{xxxx}^{j+1}-\phi_{xxxx}^j\big)
\\
&+ (1-\chi^j)\phi^j_{xxxxx}
 + \chi^j\phi^{j+1}_{xxxxx},\qquad\mbox{for}\qquad |x-h_j| < \varepsilon ,
\\\\
&\phi_{xxxxx}^{j+1},\qquad \mbox{for}\qquad h_j +\varepsilon \leq
x \leq m_{j+1}.
\end{cases}
\end{flalign}
We use first \eqref{eq: u 5x form}, \eqref{eq: lemma 8.2 CP},
\eqref{ineq: estimate diff phij}, \eqref{ineq: estim gen diff
phij}, \eqref{eq: deriv cutoff}, \eqref{eq: u xxx form},  and get
\begin{equation}\label{eq:114 estimate T2}
T_{2,1} \, = \, \mathcal O\big(\varepsilon^{-3}\alpha(r)\big) \;
- \; \varepsilon^2 \left[ (1-\chi^j)\phi^j_{xxxxx} +
\chi^j\phi^{j+1}_{xxxxx} \right] \;+\; f'(u^h) \Big[
(1-\chi^j)\phi^j_{xxx}
 + \chi^j\phi^{j+1}_{xxx}
\Big],
\end{equation}
and then, after differentiating \eqref{eq: phij solves BS} three
times, we obtain
\begin{subequations}\label{eq:114.2 estimate T2}
\begin{eqnarray}
T_{2,1}
&=&
\mathcal O\big(\varepsilon^{-3}\alpha(r)\big)
\nonumber
\\
& & \; - \; (1-\chi^j) \,  f'(\phi^j) \phi^j_{xxx} \; - \; \chi^j
\, f'(\phi^{j+1}) \phi^{j+1}_{xxx} \;+\; f'(u^h) \Big[
(1-\chi^j)\phi^j_{xxx}
 + \chi^j\phi^{j+1}_{xxx}
\Big]
\label{eq:115 estimate T2a}
\\
& & \; - \; (1-\chi^j) \, \big(f'(\phi^j)\big)_{xx} \phi^j_x
\;-\; \chi^j \, \big(f'(\phi^{j+1})\big)_{xx} \phi^{j+1}_x
\nonumber
\\
& & \; - \; 2 (1-\chi^j) \big(f'(\phi^j)\big)_x \phi^j_{xx} \; -
\; 2 \chi^j \big(f'(\phi^{j+1})\big)_x \phi^{j+1}_{xx} \nonumber
\\
& \stackrel{\scriptscriptstyle\eqref{eq: phij solves BS}}{=\joinrel=}&
\mathcal O\big(\varepsilon^{-3}\alpha(r)\big)
\; - \; \eqref{eq:115 estimate T2a}
\; - \; (1-\chi^j)\Big[ f^{^{(3)}}\!(\phi^j) (\phi_x^j)^3 \, + \,
\varepsilon^{-2} f''(\phi^j)  f(\phi^j) \phi^j_x \Big] \nonumber
\\
& & \; - \; \chi^j \Big[f^{^{(3)}}\!(\phi^{j+1})  (\phi^{j+1}_x)^3
\, + \, \varepsilon^{-2} f''(\phi^{j+1}) f(\phi^{j+1})
\phi^{j+1}_x \Big] \nonumber
\\
& & \; - \; 2\varepsilon^{-2} \Big[ (1-\chi^j) f''(\phi^j)
f(\phi^j)  \phi^j_x \; + \; \chi^j f''(\phi^{j+1}) f(\phi^{j+1})
\phi^{j+1}_x \Big]. \nonumber
\end{eqnarray}
\end{subequations}
So, this yields
\begin{subequations}\label{eq:114.5 estimate T2}
\begin{eqnarray}
T_{2,1}
& = &
\mathcal O\big(\varepsilon^{-3}\alpha(r)\big)
 -
\eqref{eq:115 estimate T2a} \; - \; \Big[ (1-\chi^j)
f^{^{(3)}}\!(\phi^j) (\phi_x^j)^3 \; + \; \chi^j
f^{^{(3)}}\!(\phi^{j+1})  (\phi^{j+1}_x)^3 \Big] \nonumber
\\
& &
\phantom{
\mathcal O\big(\varepsilon^{-3}\alpha(r)\big)
}
 -
3 \varepsilon^{-2} \Big[ (1-\chi^j) f''(\phi^j)f(\phi^j) \phi^j_x
\; + \; \chi^j f''(\phi^{j+1}) f(\phi^{j+1}) \phi^{j+1}_x \Big]
\nonumber 
\\
& = &
\mathcal O\big(\varepsilon^{-3}\alpha(r)\big)
 -
\eqref{eq:115 estimate T2a} -  f^{^{(3)}}\!(\phi^j) (\phi_x^j)^3
 -
3 \varepsilon^{-2} f''(\phi^j) f(\phi^j)  \phi^j_x \nonumber
\\
&  & \; + \; \chi^j \Big[ f^{^{(3)}}\!(\phi^j) (\phi_x^j)^3 \; -
\; f^{^{(3)}}\!(\phi^{j+1})  (\phi^{j+1}_x)^3 \Big]
\label{eq:114.5b estimate T2}
\\
& & \; + \; 3 \, \varepsilon^{-2}  \, \chi^j \, \Big[
 f''(\phi^j) f(\phi^j)  \phi^j_x
\; - \;  f''(\phi^{j+1}) f(\phi^{j+1})  \phi^{j+1}_x \Big].
\label{eq:114.5c estimate T2}
\end{eqnarray}
\end{subequations}

To estimate the term \eqref{eq:115 estimate T2a}, we use
\eqref{eq: Lagrange interp}-\eqref{eq: Lagrange interp2} with $
\chi:=\chi^j $ and
\begin{equation}\label{eq:115 estimate T2}
F(s) \, := \, f'\big(\theta_1(s)\big) \, \theta_2(s) ,\qquad s
\in [0,1],
\end{equation}
where
\begin{equation}\label{eq:116 estimate T2}
\theta_1(s) \, := \,
(1-s) \phi^j + s \phi^{j+1},
\qquad
\theta_2(s) \, := \,
(1-s)\phi^j_{xxx}
 + s \phi^{j+1}_{xxx},
\qquad s \in [0,1],
\end{equation}
so \eqref{eq:115 estimate T2a} equals to $ R(\chi^j), $ (cf.
\eqref{eq: Lagrange interp}-\eqref{eq: Lagrange interp2} for the
detailed definition of the remainder $R$), and hence, it is given
by
\begin{flalign}
\eqref{eq:115 estimate T2a} & = (1-\chi^j) \Big[
(\phi^{j+1}-\phi^j)^2 \int_0^{\chi^j} s \, \theta_2(s)
f^{(3)}(\theta_1(s))\, \, ds \nonumber
\\
& \phantom{ = (1-\chi^j) \Big[ } \, + \, 2 \,
(\phi_{xxx}^{j+1}-\phi_{xxx}^j) \, (\phi^{j+1}-\phi^j)\,
\int_0^{\chi^j} \!\! s \, f''(\theta_1(s)) \, ds \Big] \nonumber
\\
&  \phantom{=}\; + \chi^j \Big[ (\phi^{j+1}-\phi^j)^2
\int_{\chi^j}^1 (1-s) \, \theta_2(s) f^{(3)}(\theta_1(s))\, \, ds
\nonumber
\\
& \phantom{ = \; + \chi^j \Big[ } \, + \, 2 \,
(\phi_{xxx}^{j+1}-\phi_{xxx}^j) \, (\phi^{j+1}-\phi^j)\,
\int_{\chi^j}^1 \!\! (1-s) \, f''(\theta_1(s)) \, ds.
\Big]\label{eq:118 estimate T2}
\end{flalign}
The above, in view of \eqref{eq: a lemma 8.2 CP}, \eqref{eq:
11.3a flow in channel} is bounded as follows
\begin{equation}\label{eq:119 estimate T3} \big|\eqref{eq:115
estimate T2a}\big| \;\leq \; C \varepsilon^{-3} \alpha(r),
\qquad\mbox{for}\quad |x-h_j| < \varepsilon.
\end{equation}

We also exploit \eqref{eq: phi x leq ceps}, \eqref{eq: lemma 8.2
CP}, to obtain
\begin{flalign}
\eqref{eq:114.5b estimate T2}
& = \big| f^{^{(3)}}\!(\phi^j) (\phi_x^j)^3 \; \pm \;
f^{^{(3)}}\!(\phi^j) (\phi_x^{j+1})^3 \; - \;
f^{^{(3)}}\!(\phi^{j+1})  (\phi^{j+1}_x)^3 \big| \nonumber
\\
& \leq \big| f^{^{(3)}}\!(\phi^j) \big| \big| (\phi_x^j)^3 \, - \,
(\phi_x^{j+1})^3 \big| \, + \, \big|\phi^{j+1}_x\big|^3 \big|
f^{^{(3)}}\!(\phi^{j+1}) \, - \, f^{^{(3)}}\!(\phi^j) \big|
\nonumber
\\
& \leq C \varepsilon^{-2} \big| \phi_x^j \, - \, \phi_x^{j+1}
\big| \, + \, C \varepsilon^{-3} \big| f^{^{(3)}}\!(\phi^{j+1}) \,
- \, f^{^{(3)}}\!(\phi^j) \big| \nonumber
\\
& \leq
C \varepsilon^{-2}
\big|
\phi_x^j
\, - \, \phi_x^{j+1}
\big|
\, + \,
C \varepsilon^{-3}
\big|
\phi^{j+1}
\, - \,
\phi^j
\big|
\nonumber
\\
& \leq
C \varepsilon^{-3} \alpha(r), \qquad\mbox{for}\quad |x-h_j| < \varepsilon.
\label{eq:119.1 estimate T3}
\end{flalign}

A similar argument yields
\begin{equation}\label{eq:119.3 estimate T3}
\eqref{eq:114.5c estimate T2}
\leq
C \varepsilon^{-3} \alpha(r), \qquad\mbox{for}\quad |x-h_j| < \varepsilon.
\end{equation}

Combining \eqref{eq: phi nx leq ceps}, \eqref{eq: 5 proof for
Bups}, \eqref{eq:100apali}, \eqref{eq:104.5 estimate
T2}-\eqref{eq:107}, \eqref{eq:114.5 estimate T2}, \eqref{eq:119
estimate T3}, \eqref{eq:119.1 estimate T3}, and \eqref{eq:119.3
estimate T3}, we get the final estimate for $T_2$, given by
\begin{equation}\label{eq: estimate of T2}
|T_2| \, \leq \, C \, \varepsilon^{-5/2} \, \mathcal
B_\varepsilon^{1/2}[\upsilon].
\end{equation}
\\
\vspace{0.5cm}
\\
\textit{Estimate of $ \mathit{T_3} $:}
\\
Integrating by parts twice, we obtain
\begin{equation}\label{eq:120 estimate T3}
T_3 =  \big\langle f^h  \upsilon^2 \; , \; \big( \gamma^j
u^h_x \big)_{xx} \big\rangle,
\end{equation}
where $ f^h $ is defined in \eqref{eq:definition2 of f2}.

By \eqref{eq: assumption1 B upsilon}, \eqref{eq: 5 proof for Bups}
and the fact that $$ u^h =\mathcal O(1), $$ we get
\begin{equation}\label{eq:121 estimate T3}
u^h + \tau \upsilon = \mathcal O(1),
\end{equation}
and thus, the integrand $ |f''(u^h + \tau \upsilon)| $ in the
definition of $ f^h $ is uniformly bounded; so, using again
\eqref{eq: 5 proof for Bups}, we derive
\begin{equation}\label{eq:122 estimate T3}
| f^h | \upsilon^2 \leq C \varepsilon^{-1} \mathcal B_\varepsilon[\upsilon].
\end{equation}

Then \eqref{eq:120 estimate T3}, by employing \eqref{eq:
derivative of uh}, \eqref{eq: uxxh}, \eqref{eq: phij solves BS},
\eqref{eq: u xxx form}, \eqref{eq: phi nx leq ceps}, \eqref{eq:
100.7}, \eqref{ineq: cons1 of lemma 7.5 CP},
 \eqref{ineq: cons2 of lemma 7.5 CP}, yields
\begin{equation}\label{eq: estimate of T3}
|T_3| \, \leq \, C \, \varepsilon^{-4} \, \mathcal B_\varepsilon[\upsilon].
\end{equation}
\vspace{0.5cm}

Gathering \eqref{eq: estimate of T1}, \eqref{eq: estimate of T2},
and \eqref{eq: estimate of T3}, we arrive at the final estimate
\begin{equation}\label{ineq:estimate  T123}
|T_1+T_2+T_3| \leq C \Big( \varepsilon^{-2} \alpha(r) \,+\,
\varepsilon^{-5/2} \mathcal B^{1/2}_{\varepsilon}[\upsilon]
\,+\, \varepsilon^{-4} \mathcal B_\varepsilon[\upsilon] \Big).
\end{equation}

\vspace{0.6cm}

Let us proceed now with the second summand in \eqref{eq: 1 flow in
channel} that comes from the AC part.

In particular, we have to estimate the terms
\begin{equation}\label{eq: 1 AC part}
\underbrace{ \hspace{-0.4em}
\big\langle
A_{2,\varepsilon}(u^h)
\; , \;
\gamma^j u^h_x
\big\rangle
\hspace{-0.4em} }_{I_1}
\qquad + \qquad
\underbrace{ \hspace{-0.4em}
\big\langle
L_{2,\varepsilon}^h(\upsilon)
\;,\;
\gamma^j u^h_x
\big\rangle
\hspace{-0.4em} }_{I_2}
\qquad + \qquad
\underbrace{ \hspace{-0.4em}
\big\langle
f^h \, \upsilon^2
\; , \;
\gamma^j u^h_x
\big\rangle
\hspace{-0.4em} }_{I_3}
\end{equation}
for $ j=1,2,\ldots, N, \, $ where $ f^h $ is given in
\eqref{eq:definition2 of f2}; here, recall the notation
$$
A_{2,\varepsilon}(u) =\varepsilon^2 u_{xx} - f(u) = \mathscr
L^b(u) \qquad\mbox{and}\qquad
 L_{2,\varepsilon}^h(\upsilon)
= \varepsilon^2 \upsilon_{xx} -  f'(u^h)\upsilon.
$$
\\
\vspace{0.3cm}
\textit{Estimate of} $ \mathit{I_1:}$
\\
For all $ j, $ we have $$ \mathscr L^b(\phi^j)=0, $$ which by the
definition \eqref{eq: def of uh}, implies that $$
A_{2,\varepsilon}(u^h) = 0\qquad\mbox{for}\qquad |x-h_j| \geq
\varepsilon. $$ Using the fact that $$ \gamma^j(x) =
1\qquad\mbox{for}\qquad |x-h_j| \leq \varepsilon, $$ we get
\begin{equation}\label{eq: 3 AC part}
I_1 := \big\langle A_{2,\varepsilon}(u^h) \; , \; \gamma^j
u^h_x \big\rangle = \int_{h_j-\varepsilon}^{h_j+\varepsilon}
\underbrace{\bigl[\varepsilon^2 u^h_{xx} - f(u^h)
\bigr]}_{\scriptscriptstyle \mathscr L^b(u^h)} \, u^h_x \, dx.
\end{equation}
We apply \eqref{eq: uhx leq eps}, \eqref{ineq: Lbuh leq alpha}
into \eqref{eq: 3 AC part}, and obtain
\begin{equation}\label{eq: estimate of I1}
|I_1| \, \leq \, C \, \alpha(r).
\end{equation}
See also the analytical calculation at Remark \ref{remexp2}.
\\
\vspace{0.3cm}
\\
\textit{Estimate of} $ \mathit{I_2:}$
\\
With $ \, \tau_j^h:=\gamma^j u^h_x\, $ we integrate by parts
twice the first term, to get
\begin{eqnarray}
I_2
& := &
\big\langle
L_{2,\varepsilon}^h(\upsilon)
\;,\;
\tau_j^h
\big\rangle
\nonumber
\\
& = &
\varepsilon^2
\big\langle
 \upsilon_{xx}
\;,\; \tau_j^h \big\rangle \;-\; \big\langle
f'(u^h)\upsilon \;,\; \tau_j^h \big\rangle \nonumber
\\
& = &
\varepsilon^2
\big\langle
 \upsilon
\;,\; (\tau_j^h)_{xx} \big\rangle \;-\; \big\langle
f'(u^h)\upsilon \;,\; \tau_j^h \big\rangle \nonumber
\\
& = &
\varepsilon^2
\big\langle
 \upsilon
\;,\;
\gamma_{xx}^j u^h_x
\big\rangle
\,+\,
2
\varepsilon^2
\big\langle
 \upsilon
\;,\; \gamma_x^j (u^h_x)_x \big\rangle \,+\, \big\langle
\upsilon \;,\; \gamma^j \big[\varepsilon^2  \big(u^h_x)_{xx}
\,-\, f'(u^h) u^h_x\big] \big\rangle. \label{eq: 5 AC part}
\end{eqnarray}
In the above, we used that $ \tau^h_j, (\tau^h_j)_x $ vanish at $
x=0,1.$

By \eqref{eq: derivative of uh}, \eqref{eq: uxxh}, \eqref{ineq:b lemma 7.5 CP}, \eqref{ineq: cons2 of lemma 7.5 CP},
\eqref{eq: 100.7}, \eqref{eq: gamma derivative order} we have
\begin{equation}\label{eq: 6 AC part}
\big| \varepsilon^2 \gamma_{xx}^j u^h_x \big|
\leq C \varepsilon^{-1} \beta(r)
\qquad \mbox{and} \qquad
\big|\varepsilon^2
\gamma_x^j (u^h_x)_x  \big|
\leq C \varepsilon^{-1} \beta(r).
\end{equation}

Considering third term in \eqref{eq: 5 AC part}, by \eqref{eq:
100.3} and \eqref{eq: 100.5} we have
\begin{equation}\label{eq: 7 AC part}
\varepsilon^2 (u^h_x)_{xx}\, - \, f'(u^h)u^h_x \;=\;
0\qquad\mbox{except  in}\quad [h_j - \varepsilon , \, h_j +
\varepsilon].
\end{equation}
By \eqref{eq: derivative of uh}, \eqref{eq: u xxx form} we have,
for $  |x-h_j| < \varepsilon$
\begin{flalign*}
I_{2,1}&:= \gamma^j \big[\varepsilon^2  \big(u^h_x)_{xx} \,-\,
f'(u^h) u^h_x\big]
\\
&= \varepsilon^2 \Big[ \chi_{xxx}^j \, \big(\phi^{j+1}-\phi^j\big)
+ 3 \chi_{xx}^j \, \big(\phi^{j+1}_x-\phi^j_x\big) + 3 \chi_x^j
\, \big(\phi_{xx}^{j+1}-\phi_{xx}^j\big) + (1-\chi^j)\phi^j_{xxx}
+ \chi^j\phi^{j+1}_{xxx} \Big]
\\
& \qquad - f'\big(\big(1 - \chi^j\big) \, \phi^j +
\chi^j \, \phi^{j+1}\big) \big[\chi_x^j \,
\big(\phi^{j+1}-\phi^j\big)  + \big(1-\chi^j\big) \, \phi_x^j +
\chi^j \, \phi_x^{j+1} \big].
\end{flalign*}
Using \eqref{eq: phij solves BS}, it follows that
\begin{multline}\label{eq: 9 AC part}
I_{2,1} = \varepsilon^2 \Big[ \chi_{xxx}^j \,
\big(\phi^{j+1}-\phi^j\big) + 3 \chi_{xx}^j \,
\big(\phi^{j+1}_x-\phi^j_x\big) + 3 \chi_x^j \,
\big(\phi_{xx}^{j+1}-\phi_{xx}^j\big) \Big]
\\ + f'\big(u^h\big)
\chi_x^j \, \big(\phi^{j+1}-\phi^j\big) + I_{2,2 },
\end{multline}
where
\begin{equation}\label{eq: 10 AC part}
I_{2,2 } = (1-\chi^j) f'(\phi^j) \phi^j_{x} + \chi^j
f'(\phi^{j+1}) \phi^{j+1}_{x} \,-\, f'\big((1 - \chi^j) \, \phi^j
+ \chi^j \, \phi^{j+1}\big) \big[ \big(1-\chi^j\big) \, \phi_x^j
+ \chi^j \, \phi_x^{j+1} \big].
\end{equation}

To estimate the term $ I_{2,2 }, $ we employ \eqref{eq: Lagrange
interp}-\eqref{eq: Lagrange interp2}, for $ \chi:=\chi^j $, and
\begin{equation}\label{eq: 11 AC part}
F(s) \, := \, f'\big(\theta_1(s)\big) \, \theta_3(s) ,\qquad s
\in [0,1],
\end{equation}
where
\begin{equation}\label{eq: 12 AC part}
\theta_1(s) \, := \,
(1-s) \phi^j + s \phi^{j+1},
\qquad
\theta_3(s) \, := \,
(1-s)\phi^j_{x}
 + s \phi^{j+1}_{x},
\qquad s \in [0,1],
\end{equation}
to obtain
\begin{equation}\label{eq: 13 AC part}
\begin{split}
I_{2,2 }  = &(1-\chi^j) \Big[ (\phi^{j+1}-\phi^j)^2
\int_0^{\chi^j} s \, \theta_2(s) f^{(3)}(\theta_1(s))\, \,
ds\\
&+ \; 2 \, (\phi_{x}^{j+1}-\phi_{x}^j) \, (\phi^{j+1}-\phi^j)\,
\int_0^{\chi^j} \!\! s \, f''(\theta_1(s)) \, ds \Big]
\\
&  \phantom{=}\; + \chi^j \Big[ (\phi^{j+1}-\phi^j)^2
\int_{\chi^j}^1 (1-s) \, \theta_2(s) f^{(3)}(\theta_1(s))\, \, ds
\\
& \phantom{ = \; + \chi^j \Big[ } \, + \, 2 \,
(\phi_{x}^{j+1}-\phi_{x}^j) \, (\phi^{j+1}-\phi^j)\,
\int_{\chi^j}^1 \!\! (1-s) \, f''(\theta_1(s)) \, ds \Big].
\end{split}
\end{equation}

Combining \eqref{eq: lemma 8.2 CP}, \eqref{ineq: estimate diff
phij}, \eqref{eq: deriv cutoff}, \eqref{eq: 5 proof for Bups},
\eqref{eq: 5 AC part}-\eqref{eq: 9 AC part}, \eqref{eq: 13 AC
part}, and taking into account that each of the integrals in
\eqref{eq: 5 AC part} is taken over an interval of length $
\mathcal O(\varepsilon), $ we conclude that
\begin{equation}\label{eq: estimate of I2}
|I_2|
\, \leq \, C \, \varepsilon^{-1/2} \, \mathcal B_\varepsilon^{1/2}[\upsilon] \, \big[\alpha(r) \, + \, \beta(r)\big].
\end{equation}
\\
\vspace{0.3cm}
\textit{Estimate of} $ \mathit{I_3:}$
\\
In view of \eqref{eq: derivative of uh}, \eqref{eq: phi nx leq ceps}, \eqref{eq:122 estimate T3}, we obtain
\begin{equation}\label{eq: estimate of I3}
|I_3|
\, \leq \, C \, \varepsilon^{-2} \, \mathcal B_\varepsilon[\upsilon].
\end{equation}

\vspace{0.5cm}

Gathering together \eqref{eq: estimate of I1}, \eqref{eq:
estimate of I2}, \eqref{eq: estimate of I3}, we get
\begin{equation}\label{eq: 2 AC part}
|I_1+I_2+I_3|
\leq C
\Big(
\alpha(r)
\,+\,
 \varepsilon^{-1/2} \mathcal B_\varepsilon^{1/2}[\upsilon] \, \big[\alpha(r)+\beta(r)\big]
 \,+\,
\varepsilon^{-2} \mathcal B_\varepsilon[\upsilon]
\Big).
\end{equation}

\vspace{0.5cm}

Combining \eqref{eq:2coord ODE 11}, \eqref{eq:norm of inv
matrices}, \eqref{ineq:estimate  T123}, and \eqref{eq: 2 AC part},
the final estimate \eqref{ineq: estimate hdot} follows.
\end{proof}

\begin{remark} In view of the estimate of Main Theorem \ref{prop h velocity} for the dynamics,
considering the main order of the Allen-Cahn part (which follows
from \eqref{eq: 2 AC part})
in comparison with this of Carr-Pego \cite{CarrPego}, 
let us emphasize the difference in our approach. We estimated separately the contributions of $ (S - \hat{S})^{-1} $ and the right side of the equations of motion, while the analysis in \cite{CarrPego} is carried out only after having applied the inverse $ (S - \hat{S})^{-1} $ of the coefficient matrix on the right side of the system; see in particular the first equations of (3.1) and (3.4) therein. Nevertheless, our result is analogous even if only the Allen-Cahn part is considered.

The estimate \eqref{ineq: estimate hdot} shows that the main
order in the dynamics will be given by the contribution of
$B_{\varepsilon}[\upsilon]$ at the terms where the exponentially
small quantities $\alpha(r)$, $\beta(r)$ do not act.

Using the spectrum of the linearized Cahn-Hilliard / Allen-Cahn
operator, and a quite wide class of weights $\delta(\eps)>0$,
$\mu(\eps)\geq 0$, we shall show that for initial data close
enough to the manifold $ \mathcal M$ (through a form
$\tilde{\mathcal A}_\varepsilon[\upsilon]$) the layer dynamics
will be stable, and will remain exponentially small in
$\varepsilon$ if the initial data are exponentially small. (The
form $\tilde{\mathcal A}_\varepsilon[\upsilon]$ will involve $u^h$
and up to second derivatives of $\upsilon$, and as we shall prove,
satisfies $\tilde{\mathcal A}_\varepsilon[\upsilon]\geq c\mathcal
B_{\varepsilon}[\upsilon]$.) This stability profile is in
agreement to Sections 1.2-1.3 and (1.2)-(1.3) of \cite{CarrPego}.
\end{remark}

\subsection{The slow channel}\label{slowch}
We define the slow channel for \eqref{eq:deltaCH-AC} to be
\begin{equation}\label{eq: def ourslow ch0}
\Gamma_\rho
:=
\Big\{
u(x): \; u = u^h + \upsilon, \; \; \tilde{\mathcal A}_\eps[\upsilon] \leq c \, \gamma(\eps) \, \alpha^2(r)
\Big\}
\end{equation}
with
\begin{equation}\label{eq:def gammaeps0}
\gamma(\eps)
:=
\big({\depsq} \eps^{-3} \, + \, {\mepsq} \eps
\; + \;
\eps^{-4}
\big)/\eta(\eps).
\end{equation}
We will study the orbit 
$ \, u(x, t) = u^{h(t)}(x) + \upsilon(x, t) \, $ of \eqref{eq:integrated IACH} as long as
\begin{equation}\label{eq:2 assump tB}
\Big(
\frac{{\depsq} \, \eps^{-7}}{ \neps^2}
\,+ \,
\frac{{\mepsq}  \, \eps^{-3}}{ \neps^2}
\; + \;
\frac{\deps \, \eps^{-5}}{\neps}
\; + \;
\frac{\deps \, \eps^{-2} }{\deps + \meps \eps^2}
\Big)
\;\tilde{\mathcal B}_\eps[\upsilon] \ll \eta(\eps)
\end{equation}
(conditions \eqref{eqasumpmd2} and \eqref{eq:2 assump tB} arise in \eqref{estA55}-\eqref{eq: assumption tildeB} further below); for instance, for the condition \eqref{eq:2 assump tB} it suffices to have
$$
\tilde{\mathcal B}_\eps[\upsilon] \ll \eps^7 \eta(\eps).
$$
For later use, notice that clearly $ \, \tilde{\mathcal B}_{\varepsilon} \geq ({\deps} + {\meps} ) \mathcal B_{\varepsilon} \, $ and the estimate  \eqref{eq: 5 proof for Bups} directly yields \begin{equation}\label{ineq: bound tildeBups} \| \upsilon \|^2_{L^\infty} \leq \frac{1}{{\deps} + {\meps}} \, \frac{1 + \varepsilon}{\varepsilon}  \, \tilde{\mathcal B}_\varepsilon[\upsilon]. \end{equation}
The following Lemma will be also useful.
\begin{lemma}\label{lemma my similar to 4.1 BX}For $ \upsilon \in C^2[0,1] $ with $ \upsilon_x(0)  = \upsilon_x(1) = 0 $ we have
\begin{equation}\label{eq: 1 mylemma 4.1 BX}
\| \upsilon_x \|_{L^\infty}
\leq
\deps^{-1/2} \, \varepsilon^{-1} \, \tilde{\mathcal B}_\varepsilon^{1/2}[\upsilon].
\end{equation}
\end{lemma}
\begin{proof} We have
\begin{equation}\label{eq: 3 mylemma 4.1 BX}
|\upsilon_x(x)| = \left| \int_0^x \upsilon_{xx} \, dy \right|  \leq  \| \upsilon_{xx} \|
\leq \deps^{-1/2} \, \varepsilon^{-1} \, \tilde{\mathcal B}_\varepsilon^{1/2}[\upsilon].
\end{equation}
\end{proof}
Next, besides the condition \eqref{eq222} we assume that the coefficients $ \delta(\eps),  \mu(\eps) $ satisfy the condition
\begin{equation}\label{eqasumpmd2}
{\depsq} \, \eps^{-6} \leq c \eta(\eps) \neps
\end{equation}
for some $ c>0 $ small enough, with $\neps:= \deps+\meps $ and the $ \eta(\eps) $ given in \eqref{def of etaeps}.

The result about the attractiveness and the slow evolution of states within the channel \eqref{eq: def ourslow ch0} is stated in the following theorem.
\begin{theorem}
Let $ \, u(x, t) = u^{h(t)}(x) + \upsilon(x, t) \, $ be an orbit of \eqref{eq:deltaCH-AC} starting outside but near the slow channel \eqref{eq: def ourslow ch0} in the sense that $ \upsilon(\cdot, 0) $ satisfies condition \eqref{eq:2 assump tB}. Then $ \tilde{\mathcal B}_\eps[\upsilon] $ will decrease exponentially until $ u $ enters the channel and will remain in the channel following the approximate manifold $ \mathcal M $ with speed $ \mathcal O(e^{-c/r}), $ thus staying in the channel for an exponentially long time. It can leave $ \Gamma_\rho $ only through the ends of the channel i.e at a time that $ (h_j-h_{j-1}) $ is reduced to $ \frac{\eps}{\rho} $ for some $ j.$
\end{theorem}
\begin{proof}
Applying \eqref{ineq: BleqA} combined with the estimate $ \, \mathcal B_{\eps} \leq \neps^{-1} \tilde{\mathcal B}_{\eps} \, $ into \eqref{ineq: estimate hdot} we immediately get
\begin{multline}\label{ineq:hdotAantiB}
| \dot{h}_i|
\leq C
{\deps}
\big(
\eps^{-5/2}  \neps^{-1/2} \tilde{\mathcal A}^{1/2}_{\eps}[\upsilon]
 \,+\,
 \eps^{-3}  \neps^{-1} \tilde{\mathcal A}_{\eps}[\upsilon]
\big)
\,+\,
C \big(\deps + \meps \eps^2\big) \eps^{-1} \alpha(r)
\\
\;+\;
C
{\meps}
\bigl(
\big(\alpha(r)+\beta(r)\big) \, \eps^{1/2} \neps^{-1/2} \tilde{\mathcal A}^{1/2}_{\eps}[\upsilon]
\,+\,
\eps^{-1} \neps^{-1} \tilde{\mathcal A}_{\eps}[\upsilon]
\bigr).
\end{multline}
We have to estimate the growth of $ \tilde{\mathcal A}_{\eps}[\upsilon(\cdot,t)], $ so that to prove the attractiveness of the slow channel, and then combined with \eqref{ineq:hdotAantiB}, we will get an upper bound of the layers' speed within the channel;
see \eqref{ineqq:hdotAantiB22}. To this end, we set
\begin{eqnarray}
I_\eps[\upsilon]
& := &
\frac{1}{2} \frac{d}{dt} \tilde{\mathcal A}_\eps[\upsilon]
\nonumber
\\
& = &
\frac{1}{2} \frac{d}{dt} \bm{\big\langle} - L_{\eps}^h(\upsilon), \, \upsilon \bm{\big\rangle}
\nonumber
\\
& = &
\bm{\big\langle} -  \frac{1}{2} \frac{\partial}{\partial t}  L_{\eps}^h(\upsilon), \; \upsilon \bm{\big\rangle}
\;-\;
\frac{1}{2} \bm{\big\langle} L_{\eps}^h(\upsilon), \; \upsilon_t \bm{\big\rangle}\label{eq:b abbrev I}
\end{eqnarray}
where, we recall by \eqref{eq: linearized operator2},
\begin{flalign}\label{eq:the linearized operator2pali}
& &L_\eps^h(\upsilon)
:=
- \deps \;
\underbrace{
\bigl(
\eps^2 \upsilon_{xx} -  f'(u^h)\upsilon
\bigr)_{xx}
}_{\scriptscriptstyle L_{1,\eps}^h(\upsilon)  = \mbox{\tiny the linearized CH part}}
 + \;
{\meps} \underbrace{
\bigl(
\eps^2 \upsilon_{xx}  -  f'(u^h)\upsilon
\bigr)
}_{\scriptscriptstyle L_{2,\eps}^h(\upsilon)  = \mbox{\tiny the linearized AC part}}.
& &
\tag{\ref{eq: linearized operator2}}
\end{flalign}
In order to write $ I_{\eps}[\upsilon] $ in a more convenient form, we first notice the pointwise estimate
\begin{eqnarray}
\frac{\partial}{\partial t}  L_{\eps}^h(\upsilon)
&=&
\frac{\partial}{\partial t}
\left[
- \deps \,
\big(
\eps^2 \upsilon_{xx} -  f'(u^h)\upsilon
\big)_{xx}
\; + \;
{\meps} \,
\big(\eps^2 \upsilon_{xx}  -  f'(u^h) \upsilon \big)
\right]
\nonumber
\\
&=&
L_{\eps}^h(\upsilon_t)
\;+\; \deps \big(\big(f'(u^h)\big)_t \upsilon \big)_{xx}
\;-\; {\meps} \big(f'(u^h)\big)_t \upsilon
\label{eq: pointwise estim 1 slow chan}
\end{eqnarray}
and note also that integration by parts yields
\begin{equation}\label{eq: int by parts slow channel 1}
\bm{\big\langle} L_{\eps}^h(\upsilon_t), \; \upsilon \bm{\big\rangle}
\; = \;
\bm{\big\langle}  \upsilon_t , \;  L_{\eps}^h(\upsilon) \bm{\big\rangle}
\;-\;
\deps \bm{\big\langle}  \upsilon_t , \; \big(f'(u^h)\upsilon
\big)_{xx} \bm{\big\rangle}
\;+\;
\deps \bm{\big\langle}  \upsilon_t , \;  f'(u^h)\upsilon_{xx} \bm{\big\rangle}
\end{equation}
where the boundary terms vanish due to the zero Neumann conditions on $ \upsilon_x, \upsilon_{xxx} $ and \eqref{eq: zero BV Wprimeuhx}.

Therefore, by \eqref{eq:b abbrev I}, \eqref{eq: pointwise estim 1 slow chan}, \eqref{eq: int by parts slow channel 1} we get
\begin{eqnarray}\label{eq:1 estimate I}
I_\eps[\upsilon]
& = &
- \bm{\big\langle} L_{\eps}^h(\upsilon), \, \upsilon_t \bm{\big\rangle}
\; + \;
\frac{\deps}{2} \bm{\Big\langle} \big(f'(u^h)\upsilon
\big)_{xx} , \; \upsilon_t  \bm{\Big\rangle}
\;-\;
\frac{\deps}{2} \bm{\Big\langle}   f'(u^h)\upsilon_{xx} , \; \upsilon_t   \bm{\Big\rangle}
\nonumber
\\
&  &
\phantom{ - \bm{\big\langle} L_{\eps}^h(\upsilon), \, \upsilon_t \bm{\big\rangle}  }
\;-\;
\frac{\deps}{2}
\bm{\Big\langle} \Big(\big(f'(u^h)\big)_t \upsilon \Big)_{xx} , \; \upsilon
\bm{\Big\rangle}
\;+\;
\frac{\meps}{2}
\bm{\left\langle}
\big(f'(u^h)\big)_t  \upsilon
, \; \upsilon
\bm{\right\rangle}
\nonumber
\\
&=&
T_\eps[\upsilon]
\;-\;
\frac{\deps}{2}  \bm{\Big\langle} \Big(\big(f'(u^h)\big)_t \upsilon \Big)_{xx} , \; \upsilon \bm{\Big\rangle}
\;+\;
\frac{\meps}{2}  \bm{\Big\langle}
\big(f'(u^h)\big)_t  \upsilon
, \; \upsilon
\bm{\Big\rangle}
\end{eqnarray}
for
\begin{equation}\label{eq: def 1st term I}
T_\eps[\upsilon]
\; := \;
- \bm{\big\langle} S_{\eps}^h(\upsilon), \, \upsilon_t \bm{\big\rangle}
\end{equation}
with $  S_{\eps}^h $ the symmetric operator (corresponding to $L_{\eps}^h$) given in \eqref{eq:def of S self adjoint}.

Regarding the last term in \eqref{eq:1 estimate I}, we use Lemma \ref{lemma 4.1 BX}, and notice that the support of each $ u_j^h $ is contained in an interval of length $2 \eps \, $ where $ \, |u_j^h| \leq c \eps^{-1}, \,$ therefore
\begin{equation}\label{ineq:ujhL1}
\|u_j^h\|_{_{L^1}} = \mathcal O(1),
\end{equation}
so we obtain
\begin{eqnarray}
{\meps} \, \bm{\Big\langle}
\big(f'(u^h)\big)_t  \upsilon
, \; \upsilon
\bm{\Big\rangle}
& \leq &
{\meps} \, \|\upsilon\|^2_{_{L^\infty}} \, \|f''(u^h)\|_{_{L^\infty}} \,  \sum_{j=1}^N \|u_j^h\|_{_{L^1}}  |\dot{h}_j|
\nonumber
\\
& \leq &
C\, \eps^{-1} \, {\meps} \, \mathcal B_\eps[\upsilon] \;  \sum_{j=1}^N  \, |\dot{h}_j|
\nonumber
\\
& \leq &
C\, \eps^{-1} \,
\big(
{\mepsq} \mathcal B^2_\eps[\upsilon] \; + \; \max_{j} |\dot{h}_j|^2
\big)
\nonumber
\\
& \leq &
C\, \eps^{-1} \,
\Big(
\frac{\mepsq}{\neps^2} \, \tilde{\mathcal B}_{\eps}^2[\upsilon]
\; + \; \max_{j} |\dot{h}_j|^2
\Big).
\label{eq:2b estimate I}
\end{eqnarray}
As for the middle term in \eqref{eq:1 estimate I}, after integrating by parts and using again Lemma \ref{lemma 4.1 BX}, it can be similarly seen that
\begin{eqnarray*}
\bm{\Big\langle} \Big(\big(f'(u^h)\big)_t \upsilon \Big)_{xx} , \; \upsilon \bm{\Big\rangle}
& = &
-
\bm{\Big\langle} \big(f'(u^h)\big)_t \upsilon_{x} , \; \upsilon_x \bm{\Big\rangle}
\; - \;
\bm{\Big\langle} \big(f''(u^h)u^h_x\big)_t \upsilon, \; \upsilon_x \bm{\Big\rangle}
\nonumber
\\
& \leq &
\|\upsilon_{x}\|^2 \, \|f''(u^h)\|_{_{L^\infty}} \,  \sum_{j=1}^N \|u_j^h\|_{_{L^\infty}}  \, |\dot{h}_j|
\nonumber
\\
& & \;+\;
\|\upsilon\|_{_{L^\infty}} \, \|\upsilon_x\|_{_{L^1}} \, \|f'''(u^h)\|_{_{L^\infty}} \,    \|u^h_x\|_{_{L^\infty}}   \,  \sum_{j=1}^N \|u_{j}^h\|_{_{L^\infty}}  \, |\dot{h}_j|
\nonumber
\\
& & \;+\;
\|\upsilon\|_{_{L^\infty}} \, \|\upsilon_x\|_{_{L^1}} \, \|f''(u^h)\|_{_{L^\infty}} \,  \sum_{j=1}^N \|\partial_{h_j}(u_x^h)\|_{_{L^\infty}}   \, |\dot{h}_j|
\nonumber
\\
& \leq &
C  \,
\big(\eps^{-3} + \eps^{-7/2}\big) \, \mathcal B_\eps[\upsilon] \,  \sum_{j=1}^N  \, |\dot{h}_j|
\nonumber
\\
& \leq &
C  \,
\eps^{-7/2} \, \mathcal B_\eps[\upsilon] \;  \sum_{j=1}^N  \, |\dot{h}_j| .
\end{eqnarray*}
In the last inequality we applied \eqref{eq: Lemma 7.10 CP}, \eqref{eq: phi nx leq ceps}, \eqref{eq: lemma 8.2 CP}, \eqref{ineq: estimate diff phij} into \eqref{eq:9 estimate T2} to get
$$ \|\partial_{h_j}(u_x^h)\|_{_{L^\infty}}  = \mathcal O(\eps^{-2}), $$
and additionally \eqref{eq: 1 Lemma 7.9 CP}-\eqref{eq: 2 Lemma 7.9 CP} into \eqref{eq: derivative uh wrt h} to get
$$ \|u_j^h\|_{_{L^\infty}}  = \mathcal O(\eps^{-1}). $$
We also used
\eqref{eq: 5 proof for Bups} 
and as well the estimate
$$
\|\upsilon_x\|_{_{L^1}} <  \eps^{-1} \mathcal B^{1/2}_\eps[\upsilon].
 $$
Therefore,
\begin{eqnarray}\label{eq:3 estimate I}
\deps\,\bm{\Big\langle} \Big(\big(f'(u^h)\big)_t \upsilon \Big)_{xx} , \; \upsilon \bm{\Big\rangle}
& \leq &
C  \, \deps\, \,
\eps^{- 3 - \frac{1}{2}} \, \mathcal B_\eps[\upsilon] \;  \sum_{j=1}^N  \, |\dot{h}_j|
\nonumber\\
& \leq &
C \,
\big(\depsq\, \, \eps^{-6} \,
\mathcal B_\eps^2[\upsilon] \; + \; \eps^{-1} \, \max_{j} |\dot{h}_j|^2
\big)
\nonumber
\\
& \leq &
C
\,
\Big(
\frac{ \depsq\, \, \eps^{-6}}{\neps^2} \;
\tilde{\mathcal B}^2_{\eps}[\upsilon]
\; + \; \eps^{-1} \, \max_{j} |\dot{h}_j|^2
\Big).
\end{eqnarray}
We next want to estimate the first term in \eqref{eq:1 estimate I}. In view of the equation of motion \eqref{eq:2coord ODE 22}
$$\upsilon_t \, = \, A_\varepsilon(u^h) \,+\,
L_\varepsilon^h(\upsilon) \,-\, {\deps} \, \big( f^h \,
\upsilon^2\big)_{xx} \, + \,\mu(\eps) f^h \, \upsilon^2
\,-\, \sum\limits_{j=1}^N u_j^{h}\,\dot{h}_j$$
with (see \eqref{eq:definition2 of f2})
$$
f^h(x) := \int_0^1 (\tau - 1) \; f''(u^h + \tau \upsilon) \; d\tau
$$
we may write $ T_\eps[\upsilon] $ as follows,
\allowdisplaybreaks
\begin{eqnarray*}\label{eq: calc of Teps}
T_\eps[\upsilon]
& = &
- \bm{\big\langle} S_{\eps}^h(\upsilon), \, \upsilon_t \bm{\big\rangle}
\nonumber
\\
& = &
- \, \bm{\Big\langle} S_{\eps}^h(\upsilon) \, , \,
L_{\eps}^h(\upsilon)
\bm{\Big\rangle}
\, - \,
\bm{\Big\langle} S_{\eps}^h(\upsilon) \, , \,
A_{\eps}(u^h)
\bm{\Big\rangle}
\, + \, \deps \,
\bm{\Big\langle} S_{\eps}^h(\upsilon) \, , \,
\big( f^h \, \upsilon^2\big)_{xx}
\bm{\Big\rangle}
\nonumber
\\
&  &
\, - \,
{\meps} \, \bm{\Big\langle} S_{\eps}^h(\upsilon) \, , \,
f^h \, \upsilon^2
\bm{\Big\rangle}
\, + \,
\sum\limits_{j=1}^N
\bm{\Big\langle} S_{\eps}^h(\upsilon) \, , \,
u_j^{h}
\bm{\Big\rangle}
\,\dot{h}_j
\nonumber
\\[0.5em]
& = &
 - \bm{\Big\langle} S_{\eps}^h(\upsilon) \, , \,
S_{\eps}^h(\upsilon)
\bm{\Big\rangle}
\, - \,
T_{2,\eps}[\upsilon]
\,+ \, \deps
T_{3,\eps}[\upsilon]
\, - \,
{\meps} \, T_{4,\eps}[\upsilon]
\,+ \,
T_{5,\eps}[\upsilon]
\nonumber
\\[0.5em]
& &
 -
\deps
\Big[
\bm{\Big\langle}  S_{\eps}^h(\upsilon) \, , \,  \big(f'(u^h)\big)_{x} \upsilon_x \bm{\Big\rangle}
\; + \;
\frac{1}{2} \bm{\Big\langle} S_{\eps}^h(\upsilon)  \, , \, \big(f'(u^h)\big)_{xx} \upsilon  \bm{\Big\rangle}
\Big]
\end{eqnarray*}
where in in the last equality, we substituted the following relation into the first term of the left side,
\begin{equation*}
L_{\eps}^h(\upsilon)
\; = \;
S_{\eps}^h(\upsilon)
\; + \;
\deps
\Big[
\big(f'(u^h)\big)_{x} \upsilon_x
\; + \;
\frac{1}{2} \big(f'(u^h)\big)_{xx} \upsilon
\Big].
\end{equation*}
So, we obtain
\begin{multline}
T_\eps[\upsilon]
=
- \|S_{\eps}^h(\upsilon)\|^2
\, - \,
T_{2,\eps}[\upsilon]
\,+ \, \deps
T_{3,\eps}[\upsilon]
\, - \,
{\meps} \, T_{4,\eps}[\upsilon]
\,+ \,
T_{5,\eps}[\upsilon]
\\
\,-\,
\deps
\Big[
\bm{\Big\langle}  S_{\eps}^h(\upsilon) \, , \,  \big(f'(u^h)\big)_{x} \upsilon_x \bm{\Big\rangle}
\; + \;
\frac{1}{2} \bm{\Big\langle} S_{\eps}^h(\upsilon)  \, , \, \big(f'(u^h)\big)_{xx} \upsilon  \bm{\Big\rangle}
\Big]
\label{eq:d calc of Teps}
\end{multline}
for
\begin{eqnarray*}
T_{2,\varepsilon}[\upsilon]&:=&
\Big\langle S_{\varepsilon}^h(\upsilon) \, , \,
A_{\varepsilon}(u^h) \Big\rangle,
\\
T_{3,\varepsilon}[\upsilon]&:=&
\Big\langle S_{\varepsilon}^h(\upsilon) \, , \, \big(f^h \,
\upsilon^2\big)_{xx} \Big\rangle,
\\
T_{4,\varepsilon}[\upsilon]&:=&
\Big\langle S_{\varepsilon}^h(\upsilon) \, , \, f^h \,
\upsilon^2 \Big\rangle,
\\
T_{5,\varepsilon}[\upsilon]&:=&
\sum\limits_{j=1}^N \Big\langle S_{\varepsilon}^h(\upsilon)
\, , \, u_j^{h} \Big\rangle \,\dot{h}_j.
\end{eqnarray*}
We have
\begin{flalign}\label{ineq: Teps2}
\big|
T_{2,\eps}[\upsilon]
\big|
& =
\left|
\bm{\Big\langle} S_{\eps}^h(\upsilon) \, , \,
A_{\eps}(u^h)
\bm{\Big\rangle}
\right|
\; \leq \;
\| S_{\eps}^h(\upsilon) \| \;
\| A_{\eps}(u^h) \|
\; \leq \;
\frac{1}{4} \| S_{\eps}^h(\upsilon) \|^2
\; + \;
\| A_{\eps}(u^h) \|^2
\nonumber
\\
& \leq
\frac{1}{4} \| S_{\eps}^h(\upsilon) \|^2
\; + \; C \eps^{-4} \alpha^2(r) ,
\tag{\ref{eq:d calc of Teps}a}
\end{flalign}
where we combined \eqref{eq: uh sat BS}, \eqref{eq: uhxx leq eps},
\eqref{ineq: Lbuh leq alpha} to estimate the term $ \| A_{\eps}(u^h) \|.  $ Also we have
\begin{flalign*}
\big(f^h \, \upsilon^2\big)_{xx}
& =
\upsilon^2
\int_0^1 (\tau-1) \; f^{(4)}(u^h + \tau \upsilon)
(u^h_x + \tau \upsilon_x)^2 \; d\tau
\, + \,
\upsilon^2
\int_0^1 (\tau-1) \; f''(u^h + \tau \upsilon)
(u^h_{xx} + \tau \upsilon_{xx}) \; d\tau
\\
&
+ \;
4  \,  \upsilon \, \upsilon_x
\;
\int_0^1 (\tau-1) \; f^{(3)}(u^h + \tau \upsilon)
(u^h_x + \tau \upsilon_x) \; d\tau
\, + \,
 2  (\upsilon_{xx}\upsilon + \upsilon_x^2) \int_0^1 (\tau-1) \; f''(u^h + \tau \upsilon)  \; d\tau
\end{flalign*}
therefore
\begin{flalign*}
\deps \big|T_{3,\eps}[\upsilon]\big|
& =
\deps \left|
\bm{\Big\langle} S_{\eps}^h(\upsilon) \, , \,
\big(f^h \, \upsilon^2\big)_{xx}
\bm{\Big\rangle}
\right|
\\[0.5em]
& \leq
\deps
\;
\| S_{\eps}^h(\upsilon) \|
\;
\| \big(f^h \, \upsilon^2\big)_{xx}  \|
\\[0.5em]
& \leq
\epsilon \| S_{\eps}^h(\upsilon) \|^2
\; + \;
\frac{\depsq}{4\epsilon} \| \big(f^h \, \upsilon^2\big)_{xx}  \|^2
\\[0.5em]
& \leq
\epsilon \| S_{\eps}^h(\upsilon) \|^2
\; + \;
\frac{C \, \depsq}{\epsilon} \,
\Big[
\|\upsilon\|_{L^\infty}^4
\big(
\frac{1}{\eps^4} + \|\upsilon_x\|^4
+ \|\upsilon_{xx}\|^2
\big)
\, + \,
\|\upsilon\|_{L^\infty}^2  \,
\|\upsilon_x\|_{L^\infty}^2 \,
\big(
\frac{1}{\eps^2} + \|\upsilon_{x}\|^2
\big)
\\[0.2em]
&
\hspace{10em}
\, + \,
\|\upsilon\|_{L^\infty}^2  \,
\|\upsilon_{xx}\|^2 \,
\,+\,
\|\upsilon_x\|_{L^\infty}^2 \, \|\upsilon_{x}\|^2
\Big]
\\[0.5em]
& \leq
\epsilon \| S_{\eps}^h(\upsilon) \|^2
+
\frac{C \, \depsq}{\epsilon} \,
\Big(
\|\upsilon\|_{L^\infty}^4 \frac{1}{\eps^4}
+
\|\upsilon\|_{L^\infty}^2  \,
\|\upsilon_{xx}\|^2
+
\|\upsilon\|_{L^\infty}^2  \,
\|\upsilon_x\|_{L^\infty}^2 \frac{1}{\eps^2}
+
\|\upsilon_x\|_{L^\infty}^2 \, \|\upsilon_{x}\|^2
\Big)
\end{flalign*}
for some $ \epsilon $ small enough which is to be determined later on. In the last inequality we used that  $ \|\upsilon\|_{L^\infty}^2  = \mathcal O(1) \, $ as follows by \eqref{eq: assumption1 B upsilon}, \eqref{eq: 5 proof for Bups}.

Then, applying into the above inequality, the estimates \eqref{ineq: bound tildeBups}, \eqref{eq: 1 mylemma 4.1 BX} and the following two estimates that follow directly from the definition \eqref{eq: define Btilde},
\begin{eqnarray}
\|\upsilon_{xx}\|^2
& \leq &
\deps^{\color{red!70!black}-1} \, \eps^{-2} \, \tilde{\mathcal B}_\eps[\upsilon]
\label{ineq:upsxx}
\\[0.5em]
\|\upsilon_x\|^2
& \leq &
(\deps + \meps \eps^2)^{-1} \, \tilde{\mathcal B}_\eps[\upsilon],
\label{ineq:upsx}
\end{eqnarray}
we get
\begin{equation}\label{eq: T3eps}
{\deps}\,
\big|T_{3,\eps}[\upsilon]\big|
 \leq
\epsilon \| S_{\eps}^h(\upsilon) \|^2
\; + \;
\frac{C}{\epsilon}
\Big[
\frac{\depsq \;  \eps^{-6}}{\neps^2}
\; + \;
\frac{\deps \, \eps^{-5}}{\neps}
\; + \;
\frac{\deps \; \eps^{-2} }{\deps + \meps \eps^2}
\Big] \;
\tilde{\mathcal B}^2_\eps[\upsilon].
\tag{\ref{eq:d calc of Teps}b}
\end{equation}
Also we clearly have
\begin{flalign}
{\meps}
\big|T_{4,\eps}[\upsilon]\big|
& = \;
{\meps}
\left|
\bm{\Big\langle}
S_{\eps}^h(\upsilon) \, , \,
f^h \, \upsilon^2
\bm{\Big\rangle}
\right|
\; \leq \;
{\meps}
\| S_{\eps}^h(\upsilon) \| \;
\| f^h \, \upsilon^2  \|
\nonumber
\\[0.5em]
& \leq \;
\epsilon \| S_{\eps}^h(\upsilon) \|^2
\; + \;
\frac{{\mepsq}}{4\epsilon} \| f^h \, \upsilon^2  \|^2
\nonumber
\\[0.5em]
& \leq \;
\epsilon \| S_{\eps}^h(\upsilon) \|^2
\; + \;
\frac{C \, {\mepsq}}{\epsilon} \|\upsilon\|_{L^\infty}^2  \,
\|\upsilon\|^2
\nonumber
\\[0.5em]
& \leq \;
\epsilon \| S_{\eps}^h(\upsilon) \|^2
\; + \;
\frac{C}{ \epsilon}
\; \frac{\mepsq}{\neps^2} \;
\eps^{-1} \; \tilde{\mathcal B}^2_\eps[\upsilon].
\tag{\ref{eq:d calc of Teps}c}
\end{flalign}
In the last inequality we applied the estimate \eqref{ineq: bound tildeBups} for the term $ \|\upsilon\|_{L^\infty} $ and the estimate
\begin{equation}\label{ineq:ups}
\|\upsilon\|^2 \leq \neps^{-1} \tilde{\mathcal B}_\eps[\upsilon]
\end{equation}
which follows immediately from the definition \eqref{eq: define Btilde}.

Using \eqref{ineq:ujhL1}, we get
\begin{flalign}
\big|T_{5,\eps}[\upsilon]\big|
& = \;
\left|
\sum\limits_{j=1}^N
\bm{\Big\langle} S_{\eps}^h(\upsilon) \, , \,
u_j^{h}
\bm{\Big\rangle}
\,\dot{h}_j
\right|
\; \leq \;
\| S_{\eps}^h(\upsilon) \|
\sum\limits_{j=1}^N
\| u_j^{h} \|
\cdot
| \dot{h}_j |
\; \leq \;
C \, \| S_{\eps}^h(\upsilon) \|
\eps^{-1/2}
\Big(
\sum\limits_{j=1}^N
| \dot{h}_j |
\Big)
\nonumber
\\
& \leq \;
\epsilon \,
\| S_{\eps}^h(\upsilon) \|^2
\; + \;
\frac{C}{\epsilon}
\,
\eps^{-1}
\,
\Big(\sum\limits_{j=1}^N | \dot{h}_j |\Big)^2
\nonumber
\\
& \leq \;
\epsilon \,
\| S_{\eps}^h(\upsilon) \|^2
\; + \;
\frac{C}{\epsilon}
\,
\eps^{-1} \max\limits_j | \dot{h}_j |^2.
\tag{\ref{eq:d calc of Teps}d}
\end{flalign}
As for the last two terms in \eqref{eq:d calc of Teps}, we use \eqref{eq: uhx leq eps},\eqref{eq: uhxx leq eps} and then \eqref{ineq:upsx} and \eqref{ineq:ups} respectively, to get:
\begin{flalign}
{\deps}\,
\bm{\Big\langle}  S_{\eps}^h(\upsilon) \, , \,  \big(f'(u^h)\big)_{x} \upsilon_x \bm{\Big\rangle}
& \leq \;
\epsilon \,
\| S_{\eps}^h(\upsilon) \|^2
\; + \;
\frac{C \, {\depsq}}{\epsilon}
\,
\eps^{-2} \,\| \upsilon_x \|^2
\nonumber
\\[0.5em]
& \leq \;
\epsilon \, \| S_{\eps}^h(\upsilon) \|^2
\; +
\;
\frac{C }{\epsilon}
\;
\frac{{\depsq \;
\eps^{-2}}}{\deps + \meps \eps^2}
\; \tilde{\mathcal B}_\eps[\upsilon]
\tag{\ref{eq:d calc of Teps}e}
\label{ineq: Teps plast term}
\\[0.5em]
{\deps} \,
\bm{\Big\langle} S_{\eps}^h(\upsilon)  \, , \, \big(f'(u^h)\big)_{xx} \upsilon  \bm{\Big\rangle}
& \leq \;
\epsilon \,
\| S_{\eps}^h(\upsilon) \|^2
\; + \;
\frac{C \, {\depsq}}{\epsilon}
\,
\eps^{-4} \,\| \upsilon \|^2
\nonumber
\\[0.5em]
& \leq \;
\epsilon \, \| S_{\eps}^h(\upsilon) \|^2
\; + \;
\frac{C }{\epsilon}
\;
\frac{{\depsq \;
\eps^{-4}}}{\neps } \, \tilde{\mathcal B}_\eps[\upsilon]
\tag{\ref{eq:d calc of Teps}f}
\label{ineq: Teps last term}.
\end{flalign}
We then apply \eqref{ineq: Teps2}-\eqref{ineq: Teps last term} into \eqref{eq:d calc of Teps} and next apply the resulted estimate together with \eqref{eq:2b estimate I}, \eqref{eq:3 estimate I} into \eqref{eq:1 estimate I} and use the assumption \eqref{eq222}  to conclude that
\begin{multline}\label{ineq: Aeps est1}
\frac{1}{2} \frac{d}{dt} \tilde{\mathcal A}_\eps[\upsilon]
\; + \;
\big(1 - \frac{1}{4} - 5 \epsilon\big)
\;
\| S_{\eps}^h(\upsilon) \|^2
\; \leq \;
C \;
\Big[ \eps^{-1} \max_{j} |\dot{h}_j|^2
\; + \;
\eps^{-4} \alpha^2(r)
\Big]
\\[0.5em]
\; + \;
C \big[ \frac{{\depsq}}{\neps}
\;
\eps^{-4}
\;+\;
\Big(
\frac{\mepsq \; \eps^{-1}}{\neps^2}
\; + \;
\frac{\depsq \, \eps^{-6}}{\neps^2}
\; + \;
\frac{\deps \, \eps^{-5}}{\neps}
\; + \;
\frac{\deps \, \eps^{-2} }{\deps + \meps \eps^2}
\Big)
\;
\tilde{\mathcal B}_\eps[\upsilon]
\Big]
\, \tilde{\mathcal B}_\eps[\upsilon].
\end{multline}
We apply the estimate \eqref{ineq: estimate hdot} for the term $ \, \max_{j} |\dot{h}_j|^2 $ of the RHS of \eqref{ineq: Aeps est1}, and then substitute the estimate $ \; \mathcal B_\eps[\upsilon] \leq \neps^{-1}\tilde{\mathcal B}_\eps[\upsilon]\; $ and fix any $ \epsilon <\tfrac{3}{20} $ to get
\begin{multline}\label{estA5}
\frac{1}{2} \frac{d}{dt} \tilde{\mathcal A}_\eps[\upsilon]
\; + \;
C \;
\| S_{\eps}^h(\upsilon) \|^2
 \leq
C \;
\big({\depsq} \eps^{-3} \, + \, {\mepsq} \eps
\; + \;
\eps^{-4}\big) \alpha^2(r)
\\
 + \;
C
\Big[
\frac{{\depsq} \, \eps^{-6} }{\neps}
\; + \;
\Big(
\frac{{\depsq} \, \eps^{-7}}{ \neps^2}
\,+ \,
\frac{{\mepsq}  \, \eps^{-3}}{ \neps^2}
\; + \;
\frac{\deps \, \eps^{-5}}{\neps}
\; + \;
\frac{\deps \, \eps^{-2} }{\deps + \meps \eps^2}
\Big)
\;\tilde{\mathcal B}_\eps[\upsilon]
\Big]
\; \tilde{\mathcal B}_\eps[\upsilon].
\end{multline}
Applying the estimates \eqref{ineq: BleqA}-\eqref{ineq: AleqL1eps} into \eqref{estA5} we obtain
\begin{multline}\label{estA55}
\frac{1}{2} \frac{d}{dt} \tilde{\mathcal A}_\eps[\upsilon]
\; + \;
C \;
\eta(\eps) \,\tilde{\mathcal A}_\eps[\upsilon]
\; \leq \;
C \;
\big({\depsq} \eps^{-3} \, + \, {\mepsq} \eps
\; + \;
\eps^{-4}\big) \alpha^2(r)
\\
 + \;
C
\Big[
\frac{{\depsq} \, \eps^{-6} }{\neps}
\; + \;
\Big(
\frac{{\depsq} \, \eps^{-7}}{ \neps^2}
\,+ \,
\frac{{\mepsq}  \, \eps^{-3}}{ \neps^2}
\; + \;
\frac{\deps \, \eps^{-5}}{\neps}
\; + \;
\frac{\deps \, \eps^{-2} }{\deps + \meps \eps^2}
\Big)
\;\tilde{\mathcal B}_\eps[\upsilon]
\Big]
\,
\tilde{\mathcal A}_\eps[\upsilon]
\end{multline}
and in view of the assumptions
\eqref{eq222},\eqref{eqasumpmd2} on $ \deps,\meps, $ as well as the condition \eqref{eq:2 assump tB} for $ \tilde{\mathcal B}_\eps[\upsilon(\cdot,t)], $
we have
\begin{equation}\label{eq: assumption tildeB}
\frac{{\depsq} \, \eps^{-6} }{\neps}
\; + \;
\Big(
\frac{{\depsq} \, \eps^{-7}}{ \neps^2}
\,+ \,
\frac{{\mepsq}  \, \eps^{-3}}{ \neps^2}
\; + \;
\frac{\deps \, \eps^{-5}}{\neps}
\; + \;
\frac{\deps \, \eps^{-2} }{\deps + \meps \eps^2}
\Big)
\;\tilde{\mathcal B}_\eps[\upsilon]
= o\big(\eta(\eps)\big)
\end{equation}
so \eqref{estA55} yields
\begin{equation}\label{estA555}
\frac{d}{dt} \tilde{\mathcal A}_\eps[\upsilon]
\; + \;
c \; \eta(\eps)\;
\tilde{\mathcal A}_\eps[\upsilon]
\; \leq \;
C \; \big({\depsq} \eps^{-3} \, + \, {\mepsq} \eps
\; + \;
\eps^{-4}\big) \; \alpha^2(r).
\end{equation}
Integrating \eqref{estA555} we get
\begin{eqnarray}
\tilde{\mathcal A}_\eps[\upsilon(t)]
& \leq &
\tilde{\mathcal A}_\eps[\upsilon(0)] \, e^{ - c \eta t}
\; + \;
C\,
\gamma(\eps) \, \alpha^2(r)
\big(
1 - e^{ - c \eta t}
\big)
\label{eq:a est6 mathcalA}
\label{est tildeA6}
\\[0.5em]
& \leq &
\max
\big\{
\tilde{\mathcal A}_\eps[\upsilon(0)] \; , \;
C\,
\gamma(\eps) \, \alpha^2(r)
\big\}
\nonumber
\end{eqnarray}
with the $ \gamma(\eps) $ given in \eqref{eq:def gammaeps0}. In view of  \eqref{est tildeA6} and the definition \eqref{eq: def ourslow ch0}  of the slow channel $ \Gamma_\rho $ we deduce that the solution $ u $ evolves exponentially towards $ \Gamma_\rho.$

Applying \eqref{est tildeA6} into \eqref{ineq:hdotAantiB} we get
\begin{multline}\label{ineqq:hdotAantiB22}
| \dot{h}_i|
\leq
C (\deps + \meps \eps^2) \eps^{-1} \alpha(r)
\;+\;
C
\big(\eps^{-2}
\deps
\,+\, \meps
\big)
\eps^{-1} \neps^{-1} \,
\big(
\tilde{\mathcal A}_\eps[\upsilon(0)]
\,+\,
\gamma(\eps) \, \alpha^2(r)
\big)
\\
\;+\;
C\,
\big[\deps \eps^{-3}  \; + \; \big(\alpha(r)+\beta(r)\big) \, \eps^{1/2} \meps\big]
\;
\neps^{-1/2}
\;
\big(
\tilde{\mathcal A}^{1/2}_\eps[\upsilon(0)]
\,+\,
\gamma^{1/2}(\eps) \, \alpha(r)
\big)
\end{multline}
and since in the slow channel \eqref{eq: def ourslow ch0} we have
$\, \mathcal A_\eps[\tilde{\upsilon}(0)] \leq
c \, \gamma(\eps) \,  \alpha^2(r),\,  $
\eqref{ineqq:hdotAantiB22} becomes
\begin{multline}\label{ineq:hdotAantiB222}
| \dot{h}_i|
\leq
C (\deps + \meps \eps^2) \eps^{-1} \alpha(r)
\;+\;
C
\big(
\deps
\,+\, \meps \eps^2
\big)
\eps^{-3} \neps^{-1}
\,
\gamma(\eps) \, \alpha^2(r)
\\
\;+\;
C\,
\big[\deps \eps^{-3}  \; + \; \big(\alpha(r)+\beta(r)\big) \, \eps^{1/2} \meps\big]
\;
\neps^{-1/2}
\,
\gamma^{1/2}(\eps) \, \alpha(r)
\end{multline}
which implies that
\begin{equation*}
| \dot{h}_i|
\leq
C \;
\max\big\{
\, \big(
{\deps}
\,+ \,
\meps \eps^2
\big)
\, \eps^{-1}
\, , \;
\big(
\deps
\,+\, \meps \eps^2
\big)
\eps^{-3} \neps^{-1}
\,
\gamma(\eps) \, \alpha(r)
\, , \;
\deps \,\neps^{-1/2} \,\eps^{-3} \,\gamma^{1/2}(\eps)
\big\} \; \alpha(r)
\end{equation*}
where $ \alpha(r) $ is exponentially small (see the definition \eqref{eq: alphabetapm}, \eqref{eq: alphabeta} and the estimate
\eqref{ineq:alpha estimate}), so provided that
$$
\big(
{\deps}
\,+ \,
\meps \eps^2
\big)
\, \eps^{-1}
\ll
\alpha^{-1} \; ,
\quad
\big(
\deps
\,+\, \meps \eps^2
\big)
\eps^{-3} \neps^{-1}
\,
\gamma(\eps)
\ll
\alpha^{-2},
\quad
\deps \,\neps^{-1/2} \,\eps^{-3} \,\gamma^{1/2}(\eps)
\ll
\alpha^{-1},
$$
by \eqref{ineq:hdotAantiB222} we have
$$ |\dot{h}_i| = \mathcal O(e^{-c/r}) $$
and the solution $ \tilde{u} $ stay in the channel for an exponentially long time.
\end{proof}

\begin{remark}
Let for example
$$\delta(\eps)=\mathcal{O}(1) \quad \mbox{small
enough},$$ and
$$\mu(\eps)=\eps^{-3},$$ then $$n(\eps)=\mathcal{O}(\eps^{-3}),$$
or for example
$$\delta(\eps)=\mathcal{O}(\eps^3) \quad \mbox{small
enough},$$ and
$$\mu(\eps)=\mathcal{O}(1),$$ then $$n(\eps)=\mathcal{O}(1).$$

In both cases, the conditions \eqref{eq222} and \eqref{eqasumpmd2} are
satisfied.
\end{remark}
\section{Mass conserving layer dynamics}\label{section:MSP}
We fix a mass $ M \in (-1,1),$ and consider the mixed problem
\begin{equation}\label{eq:deltaCH-AC2}
u_t = - \deps \, \big(\varepsilon^2 u_{xx} - f(u)\big)_{xx} \; +
\;   \meps \, \big(\varepsilon^2 u_{xx}  - f(u)\big) , \tag{ACH}
\end{equation}
for $ 0 < x < 1, \; t > 0 ,\, $ subject to the boundary conditions
\begin{flalign}
&  &
u_{x}(0, \, t) =  u_{x}(1 , \, t) & \; = 0,
& \tag{BC1} \label{eq:BC1p}
\\
& &
u_{xxx}(0, \, t) & \; = 0,
& \tag{BC2} \label{eq:BC2p}
\end{flalign}
together with the constraint of mass conservation
\begin{equation}\label{eq:MC}
\int_0^1 u(x, t) \, dx = \int_0^1 u(x, 0) \, dx =: M, \qquad t>0,
\tag{MC}
\end{equation}
in place of the Neumann b.c. for $u_{xx}$ at $x=1$; here, we
replaced the fourth b.c. $u_{xxx}(1,t)=0$, used in the previous
Section, by \eqref{eq:MC}.

As we shall see, \eqref{eq:MC}, when the integrated version is
stated, yields the integrated Cahn-Hilliard/ Allen-Cahn equation
with the same b.c. as these of the integrated Cahn-Hilliard
proposed and analyzed by Bates and Xun in \cite{BatesXunI}.

We remind that $f(u):=u^3-u$.

Following \cite{BatesXunI,CarrPego}, we define the first
approximate manifold by
\begin{equation}
\mathcal M = \{u^h: \, h \in \Omega_\rho\},
\end{equation}
with $ u^h $ given in \eqref{eq: def of uh}.

Let $$ \mathrm{M}(h):=\int_0^1u^h(x)\,dx, $$ for $ h \in
\Omega_\rho, $ then, by Lemma 2.1 of \cite{BatesXunI}, $
\mathrm{M}(h) $ is a smooth function of $ h $, and
\begin{equation}\label{l21bx}
\frac{\partial \mathrm{M}}{\partial h_j} = 2 (-1)^{j+1} +
\mathcal O(\varepsilon^{-1} \beta(r)).
\end{equation}

We define the second approximate manifold $ \mathcal M_1 $,
as the constant mass sub-manifold of $ \mathcal M, $
\begin{equation}
\mathcal M_1 = \big\{u^h \in \mathcal M: \,
\int_0^1u^h(x)\,dx=M \big\},
\end{equation}
which will be the proper approximate manifold for the
mass-conserving problem \eqref{eq:deltaCH-AC2}; see also in
\cite{BatesXunI}.

It is clear that $\mathcal M_1$ is smooth, while by
\eqref{l21bx}
 and the Implicit Function Theorem, we see
that $ h_{\scriptscriptstyle N} $ is a smooth function of $ h_1,
\cdots, h_{\scriptscriptstyle N-1} $ if $ u^h \in\mathcal
M_1. $ Thus, $ \mathcal M_1 $ can be parameterized by $
(h_1, h_2, \cdots, h_{\scriptscriptstyle N-1}). $

We set
\begin{equation}
\xi := (\xi_1, \xi_2, \cdots, \xi_{\scriptscriptstyle N-1})\equiv
(h_1, h_2, \cdots, h_{\scriptscriptstyle N-1}),
\end{equation}
and for $ u^h \in \mathcal M_1, $ we will denote $ u^h $ by
$ u^\xi  $, and define
\begin{equation}
u_j^\xi := \frac{\partial u^\xi}{\partial \xi_j} =
 u_j^h + \frac{\partial u^h}{\partial h_{\scriptscriptstyle N}}  \frac{\partial h_{\scriptscriptstyle N}}{\partial h_j},
\qquad\qquad j=1,2,\cdots, N-1,
\end{equation}
where $ u_j^h $ still stands as a notation for $ \partial
u^h/\partial h_j.$

\subsection{The coordinate system}
Motivated by \cite{ABF} and \cite{BatesXunI}, instead of working
with the original problem
\eqref{eq:deltaCH-AC}-\eqref{eq:BC1p}-\eqref{eq:BC2p}-\eqref{eq:MC}
we will work with the integrated problem.

More precisely, we integrate \eqref{eq:deltaCH-AC}, use the
conditions at $ x=0 $ given in \eqref{eq:BC1p}-\eqref{eq:BC2p}
and set
\begin{equation}\label{eq:tilde u}
\tilde{u}(x, t) :=\int_0^x u(y,t) \, dy ,
\end{equation}
to get the integrated CH/AC equation,
\begin{equation}\label{eq:integrated IACH}
\tilde{u}_t \;=\; - \deps \big(\varepsilon^2 \tilde{u}_{xxx} -
W'(\tilde{u}_x)\big)_{x} \; + \; \meps \, \Big(\varepsilon^2
\tilde{u}_{xx}  - \int_0^x W'\big(\tilde{u}_x(y,t)\big) \, dy
\Big), \tag{IACH}
\end{equation}
with the boundary conditions, following directly from
\eqref{eq:tilde u},  \eqref{eq:MC}, \eqref{eq:BC1p} respectively,
\begin{flalign}
& & \tilde{u}(0 , \, t) & \;\, = 0,
 \tag{IBC0} \label{eq:IBC0}
& &
\\
& & \tilde{u}(1 , \, t) & \;\, = M, &   \tag{IMC} \label{eq:IMC}
\\
& & \tilde{u}_{xx}(0, \, t) =  \tilde{u}_{xx}(1 , \, t) & \; = 0.
& \tag{IBC1} \label{eq:IBC1}
\end{flalign}
We may apply standard arguments for establishing the well-posedness of this problem, resulting from the one of the original problem \eqref{eq:deltaCH-AC2}; we outline the basic points in \S\ref{sectionWP} of Appendix.

Here, and for the rest of this section, we have adopted the
notation $W'(u):=f(u)$ and $W(u)\equiv F(u)$ introduced in
\cite{BatesXunI}, since for the case $\delta(\eps):=1$,
$\mu(\eps):=0$, the equation \eqref{eq:integrated IACH} with the
above b.c. coincides exactly with the problem analyzed therein
(see pg. 431).

Equivalently, \eqref{eq:integrated IACH} is written as
$$
\tilde{u}_t = - \deps \, \big(\varepsilon^2 \tilde{u}_{xx} -
\euscr W(\tilde{u}_x) \big)_{xx} \; + \; \meps \, \big(
\varepsilon^2 \tilde{u}_{xx}  -  \euscr W(\tilde{u}_x)\big),$$
for $ \, 0 < x < 1, \; t > 0 , \; $ where $ \euscr W $ is
obviously given by
\begin{equation}\label{eq:definition of mathcalW}
\euscr W(u)(x,t) :=  \int_0^x W'\big(u(y,t)\big) \, dy,
\end{equation}
that is $ (\euscr W(u))_x = W'(u) $ \, with \, $ \euscr
W(u)(0,t)=0. $

Let us also denote by $ \mathrm{A}_\varepsilon $ the spatial
differential operator at the right side of \eqref{eq:integrated
IACH}, that is
\begin{equation}\label{eq:def of Aeps}
 \mathrm{A}_\varepsilon(\tilde{u}) :=
 \deps
\mathrm{A}_{1,\varepsilon}(\tilde{u}) \, + \, \meps
\mathrm{A}_{2,\varepsilon}(\tilde{u}),
\end{equation}
where $ \mathrm{A}_{1,\varepsilon}(\tilde{u}),
\mathrm{A}_{2,\varepsilon}(\tilde{u}) $ stand for the integrated
CH operator and the integrated AC operator respectively,
\begin{equation}\label{eq:def of A12epsilon}
\mathrm{A}_{1,\varepsilon}(\tilde{u}) := -\big(\varepsilon^2
\tilde{u}_{xx}  -   \euscr W(\tilde{u}_x) \big)_{xx} \, = \,
-\big(\mathrm{A}_{_{2,\varepsilon}}(\tilde{u})\big)_{xx}
\qquad\mbox{and}\qquad \mathrm{A}_{_{2,\varepsilon}}(\tilde{u}) :=
\varepsilon^2 \tilde{u}_{xx}  -  \euscr W(\tilde{u}_x).
\end{equation}

To study the dynamics of \eqref{eq:integrated IACH}
in a neighborhood of $ \mathcal M, $ we introduce a coordinate
system relative to $ \mathcal M, $
$$ \tilde{u} \mapsto (\xi, \tilde{\upsilon}), $$
as in \cite{BatesXunI}, in the sense that for a solution $
\tilde{u} $ close to $ \mathcal M $ there exist unique components
$ \tilde{u}^{\xi}, \tilde{\upsilon} $ such that
\begin{equation}\label{eq:coord decomp for int}
\tilde{u}(x, t) = \tilde{u}^{\xi(t)}(x) + \tilde{\upsilon}(x, t).
\end{equation}
More specifically, the approximate solution $\tilde{u}^{\xi}$ is
in $\mathcal M $, and
\begin{equation}\label{eq:ups BCs}
 \tilde{\upsilon}=\tilde{\upsilon}_{xx}=0 \qquad\mbox{at}\qquad x=0,1,
\end{equation}
with
 \begin{equation}\label{orthogonality condition for int}
\langle\tilde{\upsilon}, E_j\rangle := \int_0^1 \tilde{\upsilon}
\, E_j \, dx = 0,  \qquad j=1,\ldots, N-1,
\end{equation}
where $ E_j $ are approximate tangent vectors to $ \mathcal M \,
$, defined as in \cite{BatesXunI}, by
\begin{equation}\label{eq:def Ej}
E_j(x) = \bar{\mathrm{w}}_j(x) - Q_j(x), \qquad\qquad
j=1,2,\cdots,N-1,
\end{equation}
with
\begin{equation}\label{eq:def barwj}
\bar{\mathrm{w}}_j(x) := \tilde{u}^h_j(x) + \tilde{u}^h_{j+1}(x),
\end{equation}
and
\begin{equation}\label{eq:def Qj}
Q_j(x) :=
 (-\tfrac{1}{6}x^3+\tfrac{1}{2} x^2 - \tfrac{1}{3} x)\bar{\mathrm{w}}_{jxx}(0)
+ \tfrac{1}{6}(x^3-x)\bar{\mathrm{w}}_{jxxx}(1) + x
\bar{\mathrm{w}}_j(1) ,
\end{equation}
so that
\begin{equation}\label{eq:Eps BCs}
 E_j = (E_{j})_{xx} = 0 \qquad\mbox{at}\qquad x=0,1.
\end{equation}

\subsection{Equations of motion}\label{eqmmc}
We proceed next to obtain the odes system describing the motion of
$ (\xi, \tilde{\upsilon}). $ To this end, we consider the
linearized $ \mathrm{A}_\varepsilon $ at $ \tilde{u}^h, $
\begin{equation}\label{eq:the linearized operator}
L_\varepsilon^h(\tilde{\upsilon}) :=
 \deps L_{1,\varepsilon}^h(\tilde{\upsilon})
\;+\; {\meps} L_{2,\varepsilon}^h(\tilde{\upsilon}),
\end{equation}
with the linearized CH part and the linearized AC part
\begin{equation}\label{eqs:def of lin ACCH oper}
L_{1,\varepsilon}^h(\tilde{\upsilon}) := - \big( \varepsilon^2
\tilde{\upsilon}_{xx} -   L_{\scriptscriptstyle
W}^h(\tilde{\upsilon}_x) \big)_{xx} \; = \; -  \big(
L_{2,\varepsilon}^h(\tilde{\upsilon}) \big)_{xx}
\qquad\mbox{and}\qquad L_{2,\varepsilon}^h(\tilde{\upsilon}) :=
\varepsilon^2 \tilde{\upsilon}_{xx}  -  L_{\scriptscriptstyle
W}^h(\tilde{\upsilon}_x),
\end{equation}
and  $ L_{\scriptscriptstyle W}^h $ let the linearized $ \euscr W
$ at $ u^h, $
\begin{equation}\label{eq:definition of mathcalW Gateaux}
L_{\scriptscriptstyle W}^h(\tilde{\upsilon}_x)(x,t) \; := \;
\int_0^xW''\big(u^{h(t)}(y)\big) \; \tilde{\upsilon}_x(y,t) \; dy
\end{equation}
that is $ \euscr (L_{\scriptscriptstyle
W}^h(\tilde{\upsilon}_x))_x = W''(u^h) \tilde{\upsilon}_x $, with
$ L_{\scriptscriptstyle W}^h(\tilde{\upsilon}_x)(0,t)=0.$

We differentiate \eqref{orthogonality condition for int}, with
respect to $ t, $  to get
\begin{equation}\label{eq: diff orthog}
\big\langle  \partial_t \tilde{\upsilon} \; , \; E_j
\big\rangle \,+\, \big\langle \tilde{\upsilon} \; , \;
\partial_t E_j \big\rangle =  0 , \qquad j=1,\ldots, N-1,
\end{equation}
with
$$ \partial_t \tilde{\upsilon} = \partial_t (\tilde{u} - \tilde{u}^\xi)
\stackrel{\scriptscriptstyle\eqref{eq:integrated IACH}}{=\joinrel=\joinrel=}
\mathrm{A}_\varepsilon(\tilde{u}) - \partial_t \tilde{u}^\xi = \mathrm{A}_\varepsilon(\tilde{u}) -
\sum_k \tilde{u}^\xi_k  \, \dot{\xi}_k , $$
and $ \,  \partial_t E_j = \sum_k E_{j,k}  \, \dot{\xi}_k, \,$
hence \eqref{eq: diff orthog} becomes
\begin{equation}\label{eq:coord ode 1}
\sum\limits_{k=1}^{N-1} a_{jk} \, \dot{\xi}_k
 \; = \;
\big\langle A_{\varepsilon}(\tilde{u}^\xi +  \tilde{\upsilon}) \; , \; E_j\; \big\rangle ,\qquad j=1,2,\ldots, N-1,
\end{equation}
where
\begin{equation}\label{eq: coeff matrix}
a_{jk} := \big\langle u^\xi_k \; , \; E_j \big\rangle \; - \;
\big\langle \tilde{\upsilon}  \; , \; E_{j,k}\, \big\rangle
,\qquad j,k=1,2,\ldots, N-1,
\end{equation}
and the subscripts $ k $ indicate the differentiation with respect
to $ \xi_k,$
$$ u^\xi_k := \frac{\partial u^\xi}{\partial \xi_k}
\qquad \mbox{and} \qquad E_{j,k} := \frac{\partial E_j}{\partial
\xi_k}.$$

We write \eqref{eq:coord ode 1} in more useful form by expanding
the term
\begin{eqnarray}\label{eq:expand Ae}
\mathrm{A}_\varepsilon(\tilde{u}^\xi + \tilde{\upsilon})
 =
\mathrm{A}_\varepsilon(\tilde{u}^\xi) \,+\,
L_\varepsilon^h(\tilde{\upsilon}) \,+\, \deps \,  \big(f^\xi \,
\tilde{\upsilon}_x^2\big)_{x} \,+ \, \meps \int_0^x f^\xi \,
\tilde{\upsilon}_x^2\,dy,
\end{eqnarray}
where $ L_\varepsilon^h(\tilde{\upsilon}) $ is given in
\eqref{eq:the linearized operator}, and
\begin{equation}\label{eq:definition of f2}
f^\xi(x) := \int_0^1 (1-\tau) \; W'''(\tilde{u}^\xi_x + \tau \tilde{\upsilon}_x) \; d\tau ,
\end{equation}
to get
$$
\sum\limits_{k=1}^{N-1} a_{jk} \, \dot{\xi}_k
 \; = \;
 \big\langle \mathrm{A}_\varepsilon(\tilde{u}^\xi)  \; , \; E_j\; \big\rangle
\; + \;
 \big\langle L_\varepsilon^h(\tilde{\upsilon})   \; , \; E_j\; \big\rangle
\; + \; \deps \,  \big\langle \big( f^\xi \,
\tilde{\upsilon}_x^2\big)_{x}  \; , \; E_j\; \big\rangle \; +\;
\meps \, \big\langle  \int_0^x f^\xi \, \tilde{\upsilon}_x^2\,dy
\; , \; E_j\; \big\rangle.
$$

Discriminating between the (integrated) CH and AC parts (see
\eqref{eq:def of Aeps}, \eqref{eq:the linearized operator}), we
have
\begin{equation}\label{eq:2coord ode 1}
\begin{split}
\sum\limits_{k=1}^{N-1} a_{jk} \, \dot{\xi}_k \, =& \, \deps \,
\big\langle \mathrm{A}_{1,\varepsilon}(\tilde{u}^\xi)
 +
L_{1,\varepsilon}^h(\tilde{\upsilon})
 +
\big( f^\xi \, \tilde{\upsilon}_x^2\big)_{x} \; , \; E_j
\big\rangle\\
& + \, \meps \, \big\langle
\mathrm{A}_{2,\varepsilon}(\tilde{u}^\xi) +
L_{2,\varepsilon}^h(\tilde{\upsilon})
 +
\int_0^x f^\xi \, \tilde{\upsilon}_x^2\,dy \; , \;E_j\big\rangle.
\end{split}
\end{equation}

Moreover, we apply \eqref{eq:coord decomp for int} to \eqref{eq:integrated IACH}, to get 
\begin{equation}\label{eq:coord ode 2}
\tilde{\upsilon}_t = \mathrm{A}_\varepsilon(\tilde{u}^\xi +
\tilde{\upsilon})
 -
 \sum\limits_{j=1}^{N-1} \tilde{u}_j^\xi \, \dot{\xi}_j.
\end{equation}
As above, we expand in \eqref{eq:coord ode 2} the term $
\mathrm{A}_\varepsilon(\tilde{u}^\xi + \tilde{\upsilon}), $
according to \eqref{eq:expand Ae}, to get
$$
\tilde{\upsilon}_t \, = \, \mathrm{A}_\varepsilon(\tilde{u}^\xi)
\,+\, L_\varepsilon^h(\tilde{\upsilon}) \,+\, \deps \,  \big(
f^\xi \, \tilde{\upsilon}_x^2\big)_{x} \, + \, \meps \, \int_0^x
f^\xi \, \tilde{\upsilon}_x^2\,dy \, - \,
 \sum\limits_{j=1}^{N-1} \tilde{u}_j^\xi \, \dot{\xi}_j,
$$
while separating the (integrated) CH and AC induced parts, we
arrive at
\begin{equation}\label{eq:2coord ode 2}
\tilde{\upsilon}_t \, = \, \deps \, \bigg[
\mathrm{A}_{1,\varepsilon}(\tilde{u}^\xi) \, + \,
L_{1,\varepsilon}^h(\tilde{\upsilon}) \, + \, \big( f^\xi \,
\tilde{\upsilon}_x^2\big)_{x} \bigg] \, + \, \meps \, \bigg[
\mathrm{A}_{2,\varepsilon}(\tilde{u}^\xi) \, + \,
L_{2,\varepsilon}^h(\tilde{\upsilon}) \, + \, \int_0^x f^\xi \,
\tilde{\upsilon}_x^2\,dy \bigg] \, - \, \sum\limits_{j=1}^{N-1}
\tilde{u}_j^\xi\,\dot{\xi}_j .
\end{equation}

Equations \eqref{eq:2coord ode 1}, \eqref{eq:2coord ode 2} will be
mainly used in the sequel, and for the rest of the section.

\subsection{Flow near layered equilibria}For $ \tilde{\upsilon} \in C^2[0,1] $ with $ \tilde{\upsilon}(0)=\tilde{\upsilon}(1)=0, $
 we introduce the form
\begin{equation}\label{eq: def of B form}
B_\varepsilon[\tilde{\upsilon}] := \int_0^1 \big[ \varepsilon^2
\tilde{\upsilon}_{xx}^2 + \tilde{\upsilon}_x^2 \big] \, dx.
\end{equation}
We will study the orbit
$ \tilde{u}(x, t) = \tilde{u}^{\xi(t)}(x) +
\tilde{\upsilon}(x, t) $ of \eqref{eq:integrated IACH} as long
as (cf. \cite[$(80)'$]{BatesXunI} at pg. 448, for an analogous
argument)
\begin{equation}\label{eq: assumption B}
\mepsq \, \depsqinv \, \varepsilon^{-3} \,+\, \mepsq \, \depsinv
\,\varepsilon^{-5} \; + \; \big( \deps \,\varepsilon^{-7} \,+\,
\varepsilon^{-6} \,+\, \mepsq \, \depsqinv \, \varepsilon^{-2}
\big) \, B_\varepsilon[\tilde{\upsilon}] = o(1),
\end{equation}
(condition \eqref{eq: assumption B} arises in \eqref{eq: est4
mathcalA}-\eqref{eq: est5 mathcalA} further below), or
sufficiently for 
\begin{equation}\label{eq:a assumption B}
 \mepsq \, \depsqinv  \, \big( 1 \,+\, \deps \,\varepsilon^{-5}
\big) \,=\,o(\varepsilon^3) ,
\end{equation}
and as long as
\begin{equation}\label{eq:b assumption
B} \big( \deps \,+\, \varepsilon \big) \,
B_\varepsilon[\tilde{\upsilon}] = o(\varepsilon^7).
\end{equation}

By Lemma 4.1 in \cite{BatesXunI}, if $ \tilde{\upsilon} \in
C^2[0,1] $ with $ \tilde{\upsilon}(0)=\tilde{\upsilon}(1)=0, $
then the following estimates hold true
\begin{eqnarray}
\| \tilde{\upsilon} \|^2_{L^\infty} & \leq &
B_\varepsilon[\tilde{\upsilon}] \label{ineq: tupsinfty leq B},
\\[0.5em]
\| \tilde{\upsilon}_x \|^2_{L^\infty} & \leq &
\frac{1+\varepsilon}{\varepsilon} \;
B_\varepsilon[\tilde{\upsilon}]. \label{ineq: tupsxinfty leq B}
\end{eqnarray}

We prove now the next Main Theorem estimating the dynamics of the
layers, in the current mass conservative case.
\begin{theorem}\label{prop xi velocity}There exist $ \rho_2 > 0, $ and constant
$ C>0 $, such that, as long as $ h \in \Omega_{\rho} $ with $ \rho
< \rho_2 $ and the orbit $ \, \tilde{u}(x, t) =
\tilde{u}^{h(t)}(x) + \tilde{\upsilon}(x, t) \, $ of
\eqref{eq:integrated IACH} remains close to $ \mathcal M $ so
that \eqref{eq: assumption B} holds, the next bound is valid
\begin{equation}\label{ineq: estimate xidot}
\begin{split}
| \dot{\xi}_i| \leq &C {\deps} \Big( \varepsilon^{-2} \, \alpha(r)
+
 \varepsilon^{-5} \, \beta(r) \, B_\varepsilon^{1/2}[\tilde{\upsilon}]
+ \varepsilon^{-2} \, B_\varepsilon[\tilde{\upsilon}] \Big)
\\
& +  C\meps \Big( \alpha(r) \; + \; \varepsilon^{-1}
B^{1/2}_\varepsilon[\tilde{\upsilon}]+{
\eps^{-2}B_\varepsilon[\tilde{\upsilon}]} \Big).
\end{split}
\end{equation}
\end{theorem}
\begin{proof}
The first summand in the RHS of \eqref{eq:2coord ode 1} is
estimated in Bates-Xun \cite[(78)-(80)]{BatesXunI}. In
particular, it holds that
\begin{equation}\label{eq:coord ode 2 1st RHS term estimate}
\big\langle \mathrm{A}_{1,\varepsilon}(\tilde{u}^\xi)
+
L_{1,\varepsilon}^h(\tilde{\upsilon})
 +
\big( f^\xi \, \tilde{\upsilon}_x^2\big)_{x} \; , \; E_j
\big\rangle \leq C \, \bigl( \varepsilon^{-1} \, \alpha(r) +
\varepsilon^{-4} \, \beta(r) \, B_\varepsilon^{1/2}[\tilde{\upsilon}]
+ \varepsilon^{-1} \, B_\varepsilon[\tilde{\upsilon}] \bigr).
\end{equation}

Let us estimate the AC originated part
\begin{equation}\label{eq:coord ode 2 2nd RHS term estimate}
\underbrace{ \hspace{-0.4em} \big\langle
\mathrm{A}_{2,\varepsilon}(\tilde{u}^\xi) \; , \; E_j \big\rangle
\hspace{-0.4em} }_{\scriptscriptstyle T_1} \quad + \quad
\underbrace{ \hspace{-0.4em} \big\langle
L_{2,\varepsilon}^h(\tilde{\upsilon}) \; , \; E_j \, \big\rangle
\hspace{-0.4em} }_{\scriptscriptstyle T_2} \quad + \quad
\underbrace{ \hspace{-0.4em} \big\langle \int_0^x f^\xi \,
\tilde{\upsilon}_x^2\,dy \; , \; E_j \, \big\rangle
\hspace{-0.4em} }_{\scriptscriptstyle T_3} \;\;.
\end{equation}

We begin with the term $ T_1. $ First notice that $ E_j = \mathcal O(1) $ (cf.
\cite[(55)]{BatesXunI}), and therefore
\begin{equation}\label{eq:2 estimate T1}
\big| \big\langle
\mathrm{A}_{2,\varepsilon}(\tilde{u}^\xi) \; , \; E_j \big\rangle
\big|  \leq C \, \int_0^1
\mathrm{A}_{2,\varepsilon}(\tilde{u}^\xi) \,dy,
\end{equation}
where we recall that
\begin{eqnarray}\label{eq:3 estimate T1}
\mathrm{A}_{2,\varepsilon}(\tilde{u}^\xi) &:=& \varepsilon^2
\tilde{u}^\xi_{xx}  -  \euscr W(\tilde{u}_x^\xi) = \varepsilon^2
u^\xi_{x} -  \euscr W(u^\xi) \phantom{\int_0^x} \nonumber
\\
& = & \int_0^x \mathscr L^b(u^h) \, dy,
\end{eqnarray}
and that $ \euscr W $ is defined as
$$\euscr W(u^h)(x)
 :=
 \int_0^x W'\big(u^h(y)\big) \, dy,$$
while $ \mathscr L^b $ denotes the bistable operator given in
\eqref{eq:bistable operator}.

Combining \eqref{eq: uh sat BS}, \eqref{ineq: Lbuh leq alpha},
\eqref{eq:2 estimate T1}, \eqref{eq:3 estimate T1},
 we get
\begin{equation}\label{eq:6 estimate T1}
 \bm{|} T_1\bm{|} \, \leq \, C \, \varepsilon \, \alpha(r).
\end{equation}

Considering the term $ T_2 $ we have, (for $
L_{2,\varepsilon}(\tilde{\upsilon}), $ see \eqref{eq:the
linearized operator})
\begin{subequations}\label{eq:1 estimate T2}
\begin{eqnarray}
T_2 & := & \big\langle L_{2,\varepsilon}(\tilde{\upsilon}) \; ,
\; E_j \big\rangle \phantom{\int_0^x} \nonumber
\\
& = & \big\langle \varepsilon^2 \tilde{\upsilon}_{xx}  -
L_{\scriptscriptstyle W}^h(\tilde{\upsilon}_x) \; , \; E_j
\big\rangle \phantom{\int_0^x} \nonumber
\\
& = & \varepsilon^2 \big\langle  \tilde{\upsilon}  \; , \;
(E_j)_{xx} \big\rangle \; - \; \big\langle  L_{\scriptscriptstyle
W}^h(\tilde{\upsilon}_x) \; , \; E_j \big\rangle
\phantom{\int_0^x} \label{eq:1a estimate T2}
\\
& = & \varepsilon^2 \big\langle  \tilde{\upsilon}  \; , \;
\frac{\partial }{\partial h_j}\tilde{u}^h_{xx} \big\rangle \; + \;
\varepsilon^2 \big\langle  \tilde{\upsilon}  \; , \;
\frac{\partial }{\partial h_{j+1}}\tilde{u}^h_{xx} \big\rangle \;
- \; \varepsilon^2 \big\langle  \tilde{\upsilon}  \; , \;
(Q_j)_{xx} \big\rangle \; - \; \big\langle  L_{\scriptscriptstyle
W}^h(\tilde{\upsilon}_x) \; , \; E_j \big\rangle \label{eq:1b
estimate T2}
\\
&=& \underbrace{\varepsilon^2 \big\langle  \tilde{\upsilon}  \; ,
\; \frac{\partial }{\partial h_j}u^h_x \big\rangle }_{T_{2,1}} \;
+ \; \underbrace{\varepsilon^2 \big\langle  \tilde{\upsilon}  \;
, \; \frac{\partial }{\partial h_{j+1}}u^h_x \big\rangle
}_{T_{2,2}} \; - \; \underbrace{\varepsilon^2 \big\langle
\tilde{\upsilon}  \; , \; (Q_j)_{xx} \big\rangle }_{T_{2,3}} \; -
\; \underbrace{ \big\langle  L_{\scriptscriptstyle
W}^h(\tilde{\upsilon}_x) \; , \; E_j \big\rangle }_{T_{2,4}}.
\label{eq:1c estimate T2}
\end{eqnarray}
\end{subequations}
In \eqref{eq:1a estimate T2} we used the Dirichlet boundary
conditions for $ E_j, \tilde{\upsilon} $ given in \eqref{eq:ups
BCs} and \eqref{eq:Eps BCs} respectively. In \eqref{eq:1b
estimate T2} we took into account that $ u^h $ is a smooth
function, so, we interchanged $ \partial_{h_j} $ with $
\partial_{xx} $ after applying the definition of $ E_j $ by
\eqref{eq:def Ej}-\eqref{eq:def Qj}, and then, in \eqref{eq:1c
estimate T2} we substituted
\begin{equation}\label{eq:2 estimate T2}
\tilde{u}^h_x=u^h.
\end{equation}

Let us proceed with the term $ T_{2,1}. $ In order to apply $
\partial_{h_j} $ into $ u_x^h $ given in \eqref{eq: derivative of
uh}, we notice first that
$$
\chi^j = \chi\big(\frac{x-h_j}{\varepsilon}\big) , \qquad
m_j=\frac{h_{j-1} + h_j}{2}.
$$
Moreover, considering
\begin{eqnarray*}
\phi^j(x) &:=& \phi\big(x - m_j, \, h_j - h_{j-1}, \, (-1)^j\big)
\\
&=& \phi\big(x - \tfrac{h_{j-1} + h_j}{2}, \, h_j - h_{j-1}, \,
(-1)^j\big) , \qquad\mbox{for} \quad x \in [h_{j-1}, \, h_j] ,
\end{eqnarray*}
we use \eqref{eq:Lemma 7.8 CP} to get
\begin{eqnarray*}\label{eq:5 estimate T2}
 \frac{\partial }{\partial h_j} \phi^j
& = & \phi^j_x \, \frac{\partial }{\partial
h_j}\left(\tfrac{h_{j-1} + h_j}{2}\right) \, + \, \phi^j_\ell
\,\frac{\partial }{\partial h_j} \big(h_j - h_{j-1}\big)
\\
& = & - \, \frac{1}{2} \phi^j_x \, - \, \frac{1}{2}
\operatorname{sgn}(x-m_j)\, \phi_x^j \,+\,  \mathrm{w}^j
\\
& = & - \phi^j_x \,+\,  \mathrm{w}^j \qquad\mbox{in} \quad
I_j:=[m_j,\, m_{j+1}] ,
\end{eqnarray*}
and similarly
\begin{equation}\label{eq:6 estimate T2}
 \frac{\partial }{\partial h_j} \phi^{j+1}
 \,=\,
 -  \phi^{j+1}_x
\,-\,  \mathrm{w}^{j+1} \qquad\mbox{in} \; I_j ,
\end{equation}
with
$$
\mathrm{w}^j(x, h_{j-1}, h_j) := \mathrm{w}\big(x-m_j,
h_j-h_{j-1}, (-1)^j\big).
$$
Therefore, we obtain
\begin{equation}\label{eq:7 estimate T2}
 \frac{\partial }{\partial h_j} \big(\phi^{j+1} -  \phi^j\big)
\, = \,
 \phi^j_x
\, - \, \phi^{j+1}_x \, - \,  \mathrm{w}^j \,- \,
\mathrm{w}^{j+1},
\end{equation}
and
\begin{subequations}\label{eq:8 estimate T2}
\begin{eqnarray}
 \frac{\partial }{\partial h_j} \big(\phi_x^{j+1} -  \phi_x^j\big)
& =\joinrel= & \frac{\partial}{\partial x} \frac{\partial
}{\partial h_j} \big(\phi^{j+1} -  \phi^j\big) \label{eq:8a
estimate T2}
\\
& \stackrel{\scriptscriptstyle \eqref{eq:7 estimate
T2}}{=\joinrel=} & \frac{\partial}{\partial x} \big( \phi^j_x \,
- \, \phi^{j+1}_x \, - \,  \mathrm{w}^j \,- \,  \mathrm{w}^{j+1}
\big) \label{eq:8b estimate T2}
\\
&=\joinrel= & \phi^j_{xx} \, - \, \phi^{j+1}_{xx} \, - \,
\mathrm{w}^j_x \,- \,  \mathrm{w}^{j+1}_x. \label{eq:8c estimate
T2}
\end{eqnarray}
\end{subequations}

We now apply apply $ \partial_{h_j} $ to $ u^h_x $ given in
\eqref{eq: derivative of uh}, then we use \eqref{eq:5 estimate
T2}-\eqref{eq:8c estimate T2}, and noticing that
 $$ \chi^j_x = - \chi_{h_j}^j, $$ we get
\begin{equation}\label{eq:9 estimate T2}
 \frac{\partial }{\partial h_j} u^h_x
\, = \,
\begin{cases}
 - \phi^j_{xx}
\,+\,  \mathrm{w}^j_x,   \hspace{7.5cm}  \mbox{for} \quad  m_j
\leq x \leq h_j-\varepsilon , \phantom{\Bigg|_0^1}
\\
-\chi_{xx}^j \, \big(\phi^{j+1}-\phi^j\big) \,+\, \chi^j_x\big(
 \phi^j_x
\, - \, \phi^{j+1}_x \, - \,  \mathrm{w}^j \,- \,
\mathrm{w}^{j+1} \big) - \chi^j_x \,
\big(\phi_x^{j+1}-\phi_x^j\big)
\\
\,+\, \chi^j\big(\phi^j_{xx} \, - \, \phi^{j+1}_{xx} \, - \,
\mathrm{w}^j_x \,- \,  \mathrm{w}^{j+1}_x\big) \,-\, \phi^j_{xx}
\,+\, \mathrm{w}^j_x , \qquad\qquad  \mbox{for} \quad |x-h_j| <
\varepsilon , \phantom{\Bigg|_0^1}
\\
-  \phi^{j+1}_{xx} \,-\,  \mathrm{w}^{j+1}_x , \hspace{7cm}
\mbox{for} \quad h_j +\varepsilon \leq x \leq m_{j+1}.
\end{cases}
\end{equation}
By \eqref{eq: Lemma 7.10 CP}, \eqref{eq: phi xx leq ceps},
\eqref{eq: lemma 8.2 CP}, \eqref{ineq: estimate diff phij},
\eqref{eq:9 estimate T2}, we derive
\begin{equation}\label{eq: estimate T2.1}
|T_{2,1}| < \varepsilon^2  \|\tilde{\upsilon}\|_{L^1}\,
\|\partial_{h_j}u^h_x\|_{L^\infty} \leq C \,
B^{1/2}_\varepsilon[\tilde{\upsilon}].
\end{equation}
We may see that a similar estimate holds true for the term $
|T_{2,2}|. $

The term $ |T_{2,3}| $ turns out to be dominated by the
other terms (cf. \cite[(54)]{BatesXunI}), and
\begin{equation}\label{eq: estimate T2.4}
|T_{2,4}| = |\big\langle  L_{\scriptscriptstyle
W}^h(\tilde{\upsilon}_x) \; , \; E_j \big\rangle | \leq
\|L_{\scriptscriptstyle W}^h(\tilde{\upsilon}_x)\|_{L^1} \;
\| E_j\|_{L^\infty} \leq C \|\tilde{\upsilon}_x\|_{L^1}
\leq C \, B^{1/2}_\varepsilon[\tilde{\upsilon}].
\end{equation}

Moreover, we have
\begin{equation}\label{dddd7}
|T_3|\leq C\eps^{-1}B_\varepsilon[\tilde{\upsilon}].
\end{equation}

 Combining \eqref{eq:2coord ode 1}, \eqref{eq:coord ode 2
1st RHS term estimate}, \eqref{eq:6 estimate T1}, \eqref{eq:
estimate T2.1}, \eqref{eq: estimate T2.4}, and
\eqref{dddd7}, and taking into account the fact that the matrix $
\varepsilon (a_{ij})^{-1} $ is uniformly bounded as $ \varepsilon
\to 0, $ as shown in \cite[p. 448]{BatesXunI}, we derive
\eqref{ineq: estimate xidot}.
\end{proof}

\subsection{The slow channel}\label{slowmc}

For $ \tilde{\upsilon} \in C^2([0,1]) $ with $ \tilde{\upsilon} =
\tilde{\upsilon}_{xx} = 0 $ at $ x=0, 1, $ we define the form
\begin{subequations}\label{eq: def of mathcalA tups}
\begin{eqnarray}
\mathcal A_{\varepsilon}[\tilde{\upsilon}] & := & - \big\langle
L_{1,\varepsilon}^h(\tilde{\upsilon}) \,\bm{,}\,
\tilde{\upsilon}\big\rangle \label{eq:a def of mathcalA tups}
\\[0.5em]
&  = & \int_0^1 \big[ \varepsilon^2 \tilde{\upsilon}_{xx}^2 +
W''(u^h) \tilde{\upsilon}^2_x \big] \, dx \label{eq:b def of
mathcalA tups}
\end{eqnarray}
\end{subequations}
where we performed integration by parts and recall that $
L_{1,\varepsilon}^h $ stands for the linearized Cahn-Hilliard
operator $ \mathrm{A}_{1,\varepsilon} $ at $ u^h $ (see
\eqref{eq:def of A12epsilon}, \eqref{eqs:def of lin ACCH oper}),
and is given by
$$L_{1,\varepsilon}^h(\tilde{\upsilon})
= - \varepsilon^2 \tilde{\upsilon}_{xxxx} \, + \, \big(W''(u^h)
\, \tilde{\upsilon}_x \big)_x,
$$
associated with the BVP for the integrated Cahn-Hilliard equation
\begin{equation}\label{cases: CH BVP}
\begin{dcases}
\tilde{u}_t = - \varepsilon^2 \tilde{u}_{xxxx}  +
\big(W'(\tilde{u}_x)\big)_x, & \qquad 0 < x < 1,
\\
\tilde{u}(0, \, t) = 0, \qquad  \tilde{u}(1 , \, t)  \; = M, &
\\
\tilde{u}_{xx}(0, \, t) =  \tilde{u}_{xx}(1 , \, t)  \; = 0. &
\end{dcases}\tag{\mbox{ICH}}
\end{equation}

The definition of $ \mathcal A_{\varepsilon} $ is motivated by
Lemma 4.2 of \cite{BatesXunI}. There, this Lemma, combined with
the estimates on the growth of $ | \xi_i|  $ in terms of $
B_{\varepsilon} $ (see \cite[(84)]{BatesXunI}), together with the
estimates on the growth of $ \mathcal A_{\varepsilon} $ (see
\cite[(96)-(98)]{BatesXunI}) obtained by the equations of motion,
led to the characterization of the ``slow channel'':
\begin{equation}\label{eq: def slow ch}
\Gamma := \Big\{ \tilde{u}(x): \; \tilde{u} = \tilde{u}^\xi +
\tilde{\upsilon}, \; \; \mathcal A_\varepsilon[\tilde{\upsilon}]
\leq c \varepsilon^{-5} \alpha^2(r) \Big\},
\end{equation}
for the solutions of the integrated Cahn-Hilliard near
$N$-layered equillibria. This, stands as a special case of our
problem, for $\delta:=1$, $\mu(\eps):=0$.

In particular, according to \cite[Lemma 4.2]{BatesXunI}, there
is a $ \rho_0 > 0 $ such that if $ 0 < \rho < \rho_0 $ and $ h
\in \Omega_\rho, $ then for any $ \tilde{\upsilon} \in C^2 $ with
$ \tilde{\upsilon}=0 $ at $ x=0, 1 $ and $
\langle\tilde{\upsilon}, E_j\rangle = 0, \, j=1,\cdots, N-1, $
there exists a constant $ C $ independent of $ \eps $ and $
\tilde{\upsilon} $ such that
\begin{equation}\label{ineq: tA geq B}
\varepsilon^2 B_\varepsilon[\tilde{\upsilon}] \leq C \mathcal
A_\varepsilon[\tilde{\upsilon}].
\end{equation}

Let us point out that the forms $ \mathcal
A_\varepsilon[\tilde{\upsilon}], B_\varepsilon[\tilde{\upsilon}]
$ as defined here in \S\ref{section:MSP} are the forms associated
with the 4\textsuperscript{th} order Cahn-Hilliard operator and
they are defined by Bates-Xun \cite[(76)]{BatesXunI}.

Considering our problem, we define the slow channel for
\eqref{eq:integrated IACH} by (cf. \eqref{eq: def slow ch})
\begin{equation}\label{eq: def ourslow ch}
\Gamma_\rho := \Big\{ \tilde{u}(x): \; \tilde{u} = \tilde{u}^\xi
+ \tilde{\upsilon}, \; \; \mathcal
A_\varepsilon[\tilde{\upsilon}] \leq c \, \gamma(\varepsilon) \,
\alpha^2(r) \Big\},
\end{equation}
with
\begin{equation}\label{eq:def gammaeps}
\gamma(\varepsilon) := \mepsq \depsinv \varepsilon^{-1} \,+\,
\deps \varepsilon^{-5} \,+\, \, \varepsilon^{-2} \; + \; \mepsq
\, {\depsqinv} \, \varepsilon.
\end{equation}
It is clear that $ \gamma(\varepsilon) \gg 1, $ and in view of
\eqref{eq:a assumption B},
\begin{equation}\label{eq: order geps}
\gamma(\varepsilon) \;=\; \mathcal O\big(\deps \varepsilon^{-5}
\,+\, \, \varepsilon^{-2}\big).
\end{equation}

The next Main Theorem establishes attractiveness, and the slow
evolution of states within the channel \eqref{eq: def ourslow
ch}; cf. \cite[Theorem B]{BatesXunI} for an analogous result in
the Cahn-Hilliard case.
\begin{theorem}\label{mtat}
Let $ \, \tilde{u}(x, t) = \tilde{u}^{\xi(t)}(x) +
\tilde{\upsilon}(x, t) \, $ be an orbit of \eqref{eq:integrated
IACH} starting outside but near the slow channel $ \Gamma_\rho $
in the sense that $ \tilde{\upsilon}(\cdot, 0) $ satisfies
condition \eqref{eq: assumption B}. Then $
B_\varepsilon[\tilde{\upsilon}] $ will decrease exponentially
until $ \tilde{u} $ enters the channel and will remain in the
channel
following the approximate manifold $ \mathcal M $ with speed $
\mathcal O(e^{-c/r}), $ thus staying in the channel for an
exponentially long time. It can leave $ \Gamma_\rho $ only
through the ends of the channel i.e at a time that $
(h_j-h_{j-1}) $ is reduced to $ \frac{\varepsilon}{\rho} $ for
some $ j.$
\end{theorem}
\begin{proof}
Applying \eqref{ineq: tA geq B} into \eqref{ineq: estimate xidot}
we immediately get

\begin{equation}\label{ineq: estimate xidot meAantiB}
\begin{split}
| \dot{\xi}_i| \leq &C\delta(\eps) \Big( \varepsilon^{-2}
\alpha(r) \, + \, \varepsilon^{-6} \beta(r) \mathcal
A^{1/2}_{\varepsilon}[\tilde{\upsilon}] \, + \, \varepsilon^{-4}
\mathcal A_\varepsilon[\tilde{\upsilon}] \Big)\\
&+ C \meps \Big( \alpha(r) \, + \, \varepsilon^{-2} \mathcal
A^{1/2}_\varepsilon[\tilde{\upsilon}] + \eps^{-4}\mathcal
A_\varepsilon[\tilde{\upsilon}]\Big).
\end{split}
\end{equation}

In view of \eqref{ineq: estimate xidot meAantiB}, our aim is to
establish estimates on the growth of $ \mathcal
A_\varepsilon[\tilde{\upsilon}(\cdot,t)]. $

Let us set
\begin{eqnarray}
\mathcal I_\varepsilon[\tilde{\upsilon}] & := & \frac{1}{2}
\frac{d}{dt} \mathcal A_\varepsilon[\tilde{\upsilon}] \nonumber
\\[0.5em]
& = & \frac{1}{2} \frac{d}{dt} \big\langle - L_{1,
\varepsilon}^h(\tilde{\upsilon}), \, \tilde{\upsilon}
\big\rangle
\nonumber
\\[0.5em]
& = & \big\langle -  \frac{1}{2} \frac{\partial}{\partial
t}  L_{1, \varepsilon}^h(\tilde{\upsilon}), \; \tilde{\upsilon}
\big\rangle \;-\; \frac{1}{2} \big\langle L_{1,
\varepsilon}^h(\tilde{\upsilon}), \; \tilde{\upsilon}_t
\big\rangle.\label{eq:9b abbrev I}
\end{eqnarray}
In order to write $ \mathcal I_{\varepsilon}[\tilde{\upsilon}] $
in a more convenient form, we first observe that
\begin{equation}\label{eq:9 pointwise estim 1 slow chan}
\frac{\partial}{\partial t}  L_{1,
\varepsilon}^h(\tilde{\upsilon}) =
L_{1,\varepsilon}^h(\tilde{\upsilon}_t) \; + \;
\Big(\big(W''\big(u^\xi\big)\big)_t \; \tilde{\upsilon}_x\Big)_x.
\end{equation}
Moreover, using integrations by parts (i.e. symmetry of the
integrated linearized CH operator $ L_{1,\varepsilon}^h$), we
obtain
\begin{equation}\label{eq:9 int by parts slow channel 1}
\big\langle L_{1, \varepsilon}^h(\tilde{\upsilon}_t), \;
\tilde{\upsilon} \big\rangle \; = \; \big\langle
\tilde{\upsilon}_t , \;  L_{1,\varepsilon}^h(\tilde{\upsilon})
\big\rangle,
\end{equation}
where the boundary terms vanish due to the zero boundary values
of $ \tilde{\upsilon}, \tilde{\upsilon}_{xx}.$

Therefore, by \eqref{eq:9b abbrev I}, \eqref{eq:9 pointwise estim
1 slow chan}, \eqref{eq:9 int by parts slow channel 1} we get
\begin{equation}\label{eq:1 estimate mathcalI}
\frac{1}{2} \frac{d}{dt} \mathcal A_\varepsilon[\tilde{\upsilon}]
\; = \; - \big\langle L_{1,
\varepsilon}^h(\tilde{\upsilon}), \, \tilde{\upsilon}_t
\big\rangle \; - \; \frac{1}{2} \Big\langle
\Big(\big(W''\big(u^\xi\big)\big)_t \; \tilde{\upsilon}_x\Big)_x
\; \bm{,} \; \tilde{\upsilon} \Big\rangle .
\end{equation}
Regarding the second term in \eqref{eq:1 estimate mathcalI}, we
integrate by parts and use Lemma \ref{lemma 4.1 BX}, to derive as
in \cite[(93)]{BatesXunI}
\begin{eqnarray}
\left| \Big\langle \big(\big(W''\big(u^\xi\big)\big)_t \;
\tilde{\upsilon}_x\big)_x , \; \tilde{\upsilon} \Big\rangle
\right| & = & \left| \Big\langle
\big(W''\big(u^\xi\big)\big)_t \; \tilde{\upsilon}_x , \;
\tilde{\upsilon}_x \Big\rangle \right| \nonumber
\\
& \leq & \|\tilde{\upsilon}_{x}\|^2_{L^\infty} \,
\|W'''(u^\xi)\|_{_{L^\infty}} \,  \sum_{j=1}^{N-1}
\|u_j^\xi\|_{_{L^1}}  \, |\dot{\xi}_j| \nonumber
\\
& \leq & C \, \varepsilon^{-1} \,
B_\varepsilon[\tilde{\upsilon}] \;  \sum_{j=1}^{N-1}  \,
|\dot{\xi}_j| \nonumber
\\
& \leq & C \, \varepsilon^{-1} \big(
B^2_\varepsilon[\tilde{\upsilon}] \,+\, \max_j |\dot{\xi}_j|^2\big) \label{eq: est nonlin}
\end{eqnarray}
where we used Propositions \ref{lemma7.7 CP}, \ref{lemma7.8 CP}, \ref{prop7.9 CP}, for the boundedness
of the $L^1$-norm of $ u^\xi_j $.

We next have to estimate the first term in \eqref{eq:1 estimate
mathcalI}.

Recall the equation of motion \eqref{eq:2coord ode 2}
\begin{equation*}
\tilde{\upsilon}_t \, = \, \deps \, \bigg[
\mathrm{A}_{1,\varepsilon}(\tilde{u}^\xi) \, + \,
L_{1,\varepsilon}^h(\tilde{\upsilon}) \, + \, \big( f^\xi \,
\tilde{\upsilon}_x^2\big)_x \bigg] \, + \, \meps \, \bigg[
\mathrm{A}_{2,\varepsilon}(\tilde{u}^\xi) \, + \,
L_{2,\varepsilon}^h(\tilde{\upsilon}) \, + \, \int_0^x f^\xi \,
\tilde{\upsilon}_x^2\,dy \bigg] \, - \, \sum\limits_{j=1}^{N-1}
\tilde{u}_j^\xi\,\dot{\xi}_j,
\end{equation*}
with
\begin{equation*}
 f^\xi(x) := \int_0^1 (1-\tau) \; W'''(\tilde{u}^\xi_x + \tau \tilde{\upsilon}_x) \; d\tau ,
\end{equation*}
given in \eqref{eq:definition of f2}.

We may write the first term in \eqref{eq:1 estimate mathcalI} as
follows, \allowdisplaybreaks
\begin{flalign}
-\big\langle L_{1,\varepsilon}^h(\tilde{\upsilon}) \, ,\,
\tilde{\upsilon}_t\big\rangle & = -\deps
\|L_{1,\varepsilon}^h(\tilde{\upsilon})\|^2 \; - \;
\overbrace{ \Big\langle L_{1,\varepsilon}^h(\tilde{\upsilon}) \,
\bm{,} \, \deps \big[ \mathrm{A}_{1,\varepsilon}(\tilde{u}^\xi)
\, + \, \big( f^\xi \, \tilde{\upsilon}_x^2\big)_x \big] \, - \,
\sum\limits_{j=1}^{N-1} \tilde{u}_j^\xi\,\dot{\xi}_j \Big\rangle
}^{\hfill\scriptscriptstyle =:
I_{0,\varepsilon}[\tilde{\upsilon}], \; \mbox{\tiny can be
estimated in terms of $
B_\varepsilon[\tilde{\upsilon}],\|L_{1,\varepsilon}^h(\tilde{\upsilon})\|
$ by \cite[(88)-(90)]{BatesXunI}}} \nonumber
\\
- \, \meps \, & \underbrace{ \Big\langle
L_{1,\varepsilon}^h(\tilde{\upsilon}) \, \bm{,} \,
\mathrm{A}_{2,\varepsilon}(\tilde{u}^\xi) \Big\rangle
}_{I_{1,\varepsilon}[\tilde{\upsilon}]} \, - \, \meps \,
\underbrace{ \Big\langle L_{1,\varepsilon}^h(\tilde{\upsilon}) \,
\bm{,} \, L_{2,\varepsilon}^h(\tilde{\upsilon}) \Big\rangle
}_{I_{2,\varepsilon}[\tilde{\upsilon}]} \, - \, \meps \,
\underbrace{ \Big\langle L_{1,\varepsilon}^h(\tilde{\upsilon}) \,
\bm{,} \, \int_0^x f^\xi \, \tilde{\upsilon}_x^2\,dy \Big\rangle
}_{I_{3,\varepsilon}[\tilde{\upsilon}]}
\nonumber
\\[0.5em]
& = - \deps \, \|L_{1,\varepsilon}^h(\tilde{\upsilon})\|^2 \,
- \, I_{0,\varepsilon}[\tilde{\upsilon}] \,- \, \meps \,
I_{1,\varepsilon}[\tilde{\upsilon}] \, - \, \meps \,
I_{2,\varepsilon}[\tilde{\upsilon}] \, - \, \meps \,
I_{3,\varepsilon}[\tilde{\upsilon}].
\label{eq: est Iepstups}
\end{flalign}

Arguing as in \cite[(88)-(90)]{BatesXunI} and applying
\cite[(101)]{BatesXunI}, the term $
I_{0,\varepsilon}[\tilde{\upsilon}] $ is estimated by
\begin{equation}\label{eq: I0epstups}
\Big| I_{0,\varepsilon}[\tilde{\upsilon}] \Big| \leq
\frac{{\deps}}{4} \|L_{1,\varepsilon}^h(\tilde{\upsilon})\|^2
\; + \; C \, \Big( \varepsilon^{-1} \max_j |\dot{\xi}_j|^2 \, + \,
{\deps} \big( \varepsilon^{-2}\alpha^2(r)\, + \,
\varepsilon^{-4}B^2_\varepsilon[\tilde{\upsilon}] \big) \Big).
\end{equation}

Let us estimate the terms $ I_{1,\varepsilon}[\tilde{\upsilon}],
I_{2,\varepsilon}[\tilde{\upsilon}],
I_{3,\varepsilon}[\tilde{\upsilon}] $ which are the ones coming
from the AC part. Regarding the term $
I_{1,\varepsilon}[\tilde{\upsilon}] $ in \eqref{eq: est
Iepstups}, we have the estimate
\begin{eqnarray}\label{eq: I1epstups}
I_{1,\varepsilon}[\tilde{\upsilon}] &=& \Big\langle
L_{1,\varepsilon}^h(\tilde{\upsilon}) \, \bm{,} \,
\mathrm{A}_{2,\varepsilon}(\tilde{u}^\xi) \Big\rangle \nonumber
\\[0.5em]
& \leq & \|L_{1,\varepsilon}^h(\tilde{\upsilon})\| \,
\|\mathrm{A}_{2,\varepsilon}(\tilde{u}^\xi) \| \nonumber
\\[0.5em]
& \leq & \frac{\deps}{12 \, \meps}
\|L_{1,\varepsilon}^h(\tilde{\upsilon})\|^2 \, + \, 3 \,
{\depsinv} \, \meps \,
\|\mathrm{A}_{2,\varepsilon}(\tilde{u}^\xi) \|^2 \nonumber
\\[0.5em]
& \leq & \frac{\deps}{12 \, \meps}
\|L_{1,\varepsilon}^h(\tilde{\upsilon})\|^2 \, + \, C \,
{\depsinv} \, \meps \, \varepsilon \, \alpha^2(r).
\end{eqnarray}
In the last inequality we used \eqref{eq: uh sat BS},
\eqref{ineq: Lbuh leq alpha}, and \eqref{eq:3 estimate T1}.

As for the term $ I_{2,\varepsilon}[\tilde{\upsilon}] $ in
\eqref{eq: est Iepstups}, we first easily get
\begin{eqnarray}\label{eq: I2epstups}
I_{2,\varepsilon}[\tilde{\upsilon}] &=& \Big\langle
L_{1,\varepsilon}^h(\tilde{\upsilon}) \, \bm{,} \,
L_{2,\varepsilon}^h(\tilde{\upsilon}) \Big\rangle \nonumber
\\[0.5em]
& \leq & \|L_{1,\varepsilon}^h(\tilde{\upsilon})\| \,
\|L_{2,\varepsilon}^h(\tilde{\upsilon}) \| \nonumber
\\[0.5em]
& \leq & \frac{\deps}{12 \, \meps}
\|L_{1,\varepsilon}^h(\tilde{\upsilon})\|^2 \, + \, 3 \,
{\depsinv} \, \meps \, \|L_{2,\varepsilon}^h(\tilde{\upsilon})
\|^2,
\end{eqnarray}
and then, as $ u^h =\mathcal O(1), $ we have
 $ |W''(u^h)| = \mathcal O(1) $.

 By the definitions \eqref{eqs:def of lin ACCH oper}, \eqref{eq:definition of mathcalW Gateaux}
 of the linearized AC operator $ L_{2,\varepsilon}^h $ together with \eqref{ineq: tupsxinfty leq B}, we get
\begin{eqnarray}
[L_{2,\varepsilon}^h(\tilde{\upsilon}) ]^2 & := & \Big[
\varepsilon^2 \tilde{\upsilon}_{xx}  -
\int_0^xW''\big(u^{h(t)}(y)\big) \; \tilde{\upsilon}_x(y,t) \; dy
\Big]^2 \nonumber
\\[0.5em]
& \leq & 2 \varepsilon^4 \tilde{\upsilon}^2_{xx} \, + \, C
\|\tilde{\upsilon}_x\|_{L^\infty}^2 \nonumber
\\[0.5em]
& \leq & 2 \varepsilon^4 \tilde{\upsilon}^2_{xx} \, + \, C \,
\varepsilon^{-1} \, B_\varepsilon[\tilde{\upsilon}] {\red\leq C \,
\varepsilon^{-1} \, B_\varepsilon[\tilde{\upsilon}]}.
\end{eqnarray}
So, \eqref{eq: I2epstups} yields
\begin{equation}\label{eq: I2epstups2}
I_{2,\varepsilon}[\tilde{\upsilon}] \leq \frac{\deps}{12 \, \meps}
\|L_{1,\varepsilon}^h(\tilde{\upsilon})\|^2 \, + \, C \,
{\depsinv} \, \meps \, \varepsilon^{-1} \,
B_\varepsilon[\tilde{\upsilon}].
\end{equation}

For the term $ I_{3,\varepsilon}[\tilde{\upsilon}] $ in \eqref{eq:
est Iepstups}, we have
\begin{eqnarray}\label{eq: I3epstups}
I_{3,\varepsilon}[\tilde{\upsilon}] &=& \Big\langle
L_{1,\varepsilon}^h(\tilde{\upsilon}) \, \bm{,} \, \int_0^x f^\xi
\, \tilde{\upsilon}_x^2\,dy \Big\rangle \nonumber
\\[0.5em]
& \leq & \|L_{1,\varepsilon}^h(\tilde{\upsilon})\| \;
\big|\!\big| \int_0^x f^\xi \, \tilde{\upsilon}_x^2\,dy
\big|\!\big| \nonumber
\\[0.5em]
& \leq & \frac{\deps}{12\, \meps}
\|L_{1,\varepsilon}^h(\tilde{\upsilon})\|^2 \; + \; 3 \,
\depsinv \, \meps \, \big|\!\big| \int_0^x f^\xi \,
\tilde{\upsilon}_x^2\,dy \big|\!\big|^2,
\end{eqnarray}
where we recall \eqref{eq:definition of f2}
\begin{equation*}
 f^\xi(x) := \int_0^1 (1-\tau) \; W'''(\tilde{u}^\xi_x + \tau \tilde{\upsilon}_x) \; d\tau.
\end{equation*}
By \eqref{eq: assumption B}, \eqref{ineq: tupsxinfty leq B} and
the fact that $ u^h = \mathcal O(1), $ we have (cf. \eqref{eq:121
estimate T3})
\begin{equation*}
\tilde{u}^\xi_x + \tau \tilde{\upsilon}_x = u^h + \tau \upsilon =
\mathcal O(1),
\end{equation*}
and thus the integrand $ |W'''(\tilde{u}^\xi_x + \tau
\tilde{\upsilon}_x)| $ in the definition \eqref{eq:definition of
f2} of $ f^\xi $ is uniformly bounded.

So,
\begin{equation}\label{eq:2 I3epstups}
\big|\!\big| \int_0^x f^\xi \, \tilde{\upsilon}_x^2\,dy
\big|\!\big| \leq C \, B_\varepsilon[\tilde{\upsilon}],
\end{equation}
and therefore, \eqref{eq: I3epstups} yields
\begin{equation}\label{eq:3 I3epstups}
I_{3,\varepsilon}[\tilde{\upsilon}] \leq \frac{\deps}{12\, \meps}
\|L_{1,\varepsilon}^h(\tilde{\upsilon})\|^2 \; + \; C \,
{\depsinv} \, \meps \, B^2_\varepsilon[\tilde{\upsilon}].
\end{equation}

\vspace{0.8cm}

Gathering  \eqref{eq:1 estimate mathcalI}-
\eqref{eq: I1epstups}, \eqref{eq: I2epstups2}, \eqref{eq:3
I3epstups},
we get
\begin{multline}\label{eq: est1 mathcalA}
\frac{d}{dt} \mathcal A_\varepsilon[\tilde{\upsilon}] \; + \;
\frac{\deps}{2} \|L_{1,\varepsilon}^h(\tilde{\upsilon})\|^2
\; \leq \; C \, \Big[ \varepsilon^{-1} \max_j |\dot{\xi}_j|^2 \;
+ \; \Big( {\deps} \, \varepsilon^{-2} \; + \; \mepsq \,
{\depsinv} \, \varepsilon \Big) \alpha^2(r)
\\[0.2em]
+ \; \Big( \big( \varepsilon^{-1} +\, {\deps} \, \varepsilon^{-4}
\; + \; \mepsq \, {\depsinv} \big) \,
B_\varepsilon[\tilde{\upsilon}] \; + \; \mepsq \, {\depsinv} \,
\varepsilon^{-1} \Big) B_\varepsilon[\tilde{\upsilon}] \Big].
\end{multline}
In the above, we apply the estimate \eqref{ineq: estimate xidot}
for $ \, \max_j|\dot{\xi}_j| \, $ into the first term in the RHS
of \eqref{eq: est1 mathcalA} to get
\begin{equation}\label{eq: est2 mathcalA}
\begin{split}
\frac{d}{dt} \mathcal A_\varepsilon[\tilde{\upsilon}] \; &+ \;
\frac{\deps}{2} \|L_{1,\varepsilon}^h(\tilde{\upsilon})\|^2
\; \leq \; C \, \Big[ \Big( \mepsq \varepsilon^{-1} \,+\, \depsq
\varepsilon^{-5} \,+\, {\deps} \, \varepsilon^{-2} \; + \; \mepsq
\, {\depsinv} \, \varepsilon \Big) \, \alpha^2(r)
\\&
+ \; \Big( \depsq \, \varepsilon^{-11} \, \beta^2(r) \; + \;
\mepsq \, {\depsinv} \, \varepsilon^{-1} \; + \; \mepsq \,
\varepsilon^{-3} \; \\
&+ \; \big( {\depsq} \, \varepsilon^{-5} \; {\red{+
\;\mu(\eps)^2\varepsilon^{-5}+\eps^{-1}}}+ \; {\deps} \,
\varepsilon^{-4} \; + \; \mepsq \, {\depsinv}\big) \,
B_\varepsilon[\tilde{\upsilon}] \Big)
B_\varepsilon[\tilde{\upsilon}] \Big].
\end{split}
\end{equation}

Let us now note that by \cite[Lemma 3.2]{BatesXunI} we have the
spectral estimate
\begin{equation}\label{ineq: minmax princ mathcalA}
0 < \Lambda \leq \lambda_{\scriptscriptstyle N} \leq
\frac{\mathcal
A_\varepsilon[\tilde{\upsilon}]}{\|\tilde{\upsilon}\|^2},
\end{equation}
where $ \Lambda $ is a constant independent of $ \varepsilon  $
and $ \xi, $ and $ \lambda_{\scriptscriptstyle N} $ denotes the
$N$\textsuperscript{th} eigenvalue of $ L_{1,\varepsilon}^h, $
\begin{equation}\label{cases: L1 EigProb}
\begin{dcases}
L_{1,\varepsilon}^h(\phi) := - \varepsilon^2 \phi''''  +
\big(W''(u^\xi)\phi' \big)' \;=\; \lambda(\varepsilon, \xi) \,
\phi, & \qquad 0 < x < 1,
\\
\phi(0) = \phi(1) = 0, &
\\
\phi''(0) = \phi''(1) = 0. &
\end{dcases}\tag{\mbox{EVP}}
\end{equation}
From \eqref{ineq: minmax princ mathcalA} we obtain
$$
\bm{\big|} \mathcal A_{\varepsilon}[\tilde{\upsilon}]  \bm{\big|}
:= \; \bm{|} - \big\langle L_{1,\varepsilon}^h(\tilde{\upsilon})
\,\bm{,}\, \tilde{\upsilon}\big\rangle \bm{|} \; \leq \; \|
L_{1,\varepsilon}^h(\tilde{\upsilon})  \| \; \|
\tilde{\upsilon} \| \stackrel{\scriptscriptstyle \eqref{ineq:
minmax princ mathcalA}}{\leq} \frac{1}{\sqrt{\Lambda}} \cdot \|
L_{1,\varepsilon}^h(\tilde{\upsilon})  \| \cdot \bm{\big|}
\mathcal A_\varepsilon[\tilde{\upsilon}] \bm{\big|}^{1/2},
$$
therefore
\begin{equation}\label{eq: A leq L1}
\mathcal A_\varepsilon[\tilde{\upsilon}] \; \leq \;
\frac{1}{\Lambda} \; \| L_{1,\varepsilon}^h(\tilde{\upsilon})
\|^2.
\end{equation}
Combining \eqref{ineq: tA geq B} with \eqref{eq: A leq L1}, we get
\begin{equation}\label{eq: est3 mathcalA}
 B_\varepsilon[\tilde{\upsilon}]
\; \leq \; \, \frac{1}{\Lambda} \, \varepsilon^{-2} \,
\|L_{1,\varepsilon}^h(\tilde{\upsilon})\|^2.
\end{equation}

Applying \eqref{eq: est3 mathcalA} into \eqref{eq: est2 mathcalA}
we obtain
\begin{multline}\label{eq: est4 mathcalA}
\frac{d}{dt} \mathcal A_\varepsilon[\tilde{\upsilon}] \; + \;
\frac{\deps}{2} \|L_{1,\varepsilon}^h(\tilde{\upsilon})\|^2
\; \leq \; C \, \Big[ \Big( \mepsq \varepsilon^{-1} \,+\, \depsq
\varepsilon^{-5} \,+\, {\deps} \, \varepsilon^{-2} \; + \; \mepsq
\, {\depsinv} \, \varepsilon \Big) \, \alpha^2(r)
\\[0.5em]
\; + \; \Big( \depsq \, \varepsilon^{-13} \, \beta^2(r) \; + \;
\mepsq \,  {\depsinv} \, \varepsilon^{-3} \; + \; \mepsq \,
\varepsilon^{-5}
\\[0.5em]
\; + \; \big( {\depsq} \, \varepsilon^{-7} \; {\red{+
\;\mu(\eps)^2\varepsilon^{-7}+\eps^{-3}}}+ \; {\deps} \,
\varepsilon^{-6} \; + \; \mepsq \, \depsinv \, \varepsilon^{-2}
\big) \, B_\varepsilon[\tilde{\upsilon}] \Big) \,
\|L_{1,\varepsilon}^h(\tilde{\upsilon})\|^2 \Big]
\end{multline}
and taking into account  \eqref{ineq:beta estimate}, \eqref{eq:
assumption B} and \eqref{eq: A leq L1}, we arrive at (cf.
\cite[(96)]{BatesXunI} for an analogous argument)
\begin{equation}\label{eq: est5 mathcalA}
\frac{d}{dt} \mathcal A_\varepsilon[\tilde{\upsilon}(t)] \; + \;
\frac{\Lambda\,\deps}{3} \mathcal
A_\varepsilon[\tilde{\upsilon}(t)] \; \leq \; C \, \Big( \mepsq
\varepsilon^{-1} \,+\, \depsq \varepsilon^{-5} \,+\, {\deps} \,
\varepsilon^{-2} \; + \; \mepsq \, {\depsinv} \, \varepsilon
\Big) \, \alpha^2(r)
\end{equation}
with the light abuse of notation $ \mathcal
A_\varepsilon[\tilde{\upsilon}(t)] $ in place of $ \mathcal
A_\varepsilon[\tilde{\upsilon}(\cdot, t)]. $

Integrating \eqref{eq: est5 mathcalA} we get
\begin{eqnarray}
\mathcal A_\varepsilon[\tilde{\upsilon}(t)] & \leq & \mathcal
A_\varepsilon[\tilde{\upsilon}(0)] \, e^{ - C_\delta t} \; + \;
C\, \gamma(\varepsilon) \, \alpha^2(r) \big( 1 - e^{ - C_\delta t}
\big) \label{ineq: tildeAfinal}
\\[0.5em]
& \leq & \max \big\{ \mathcal A_\varepsilon[\tilde{\upsilon}(0)]
\; , \; C\, \gamma(\varepsilon) \, \alpha^2(r) \big\},
\nonumber
\end{eqnarray}
where $ C_\delta:=\tfrac{\Lambda\,\deps}{3},\, C $ is a positive
constant independent of $ \varepsilon, \tilde{\upsilon}, \,$ and
the coefficient $ \gamma(\varepsilon) $ is given in \eqref{eq:def
gammaeps}.
We see by \eqref{ineq: tildeAfinal} that the solution $
\tilde{u} $ evolves exponentially towards the slow channel
\eqref{eq: def ourslow ch}.

In view of \eqref{ineq: tildeAfinal}, the estimate \eqref{ineq:
estimate xidot meAantiB}
yields
\begin{multline}\label{eq: est7 mathcalA}
| \dot{\xi}_i| \leq C \Big[ {\deps} \, \varepsilon^{-6} \beta(r)
\Big( \mathcal A^{1/2}_\varepsilon[\tilde{\upsilon}(0)] \,+\,
\gamma^{1/2}(\varepsilon) \, \alpha(r) \Big) \; + \; \big( {\deps}
\,+ \, \meps \varepsilon^2 \big) \varepsilon^{-2}
\alpha(r){\red+\mu(\eps) \, \varepsilon^{-4}\mathcal
A_\varepsilon[\tilde{\upsilon}(0)]}
\\[0.5em]
 + \,
\deps \, \varepsilon^{-4} \Big( \mathcal
A_\varepsilon[\tilde{\upsilon}(0)] \,+\, \gamma(\varepsilon) \,
\alpha^2(r) \Big) \,+\, \meps \, \varepsilon^{-2} \Big( \mathcal
A^{1/2}_\varepsilon[\tilde{\upsilon}(0)] \,+\,
\gamma^{1/2}(\varepsilon) \, \alpha(r) \Big) \Big], \qquad
\end{multline}
and in the slow channel \eqref{eq: def ourslow ch} we have
\begin{equation}
\mathcal A_\varepsilon[\tilde{\upsilon}(0)] \leq c \,
\gamma(\varepsilon) \,  \alpha^2(r),
\end{equation}
so \eqref{eq: est7 mathcalA} on its turn gives
\begin{equation}\label{eq: est8 mathcalA}
| \dot{\xi}_i| \leq C \; \max\big\{ \, \big( {\deps} \,+ \, \meps
\varepsilon^2 \big) \, \varepsilon^{-2} \, , \; \deps
\varepsilon^{-4} \gamma(\varepsilon) \alpha(r) \, , \; \meps
\varepsilon^{-2} \gamma^{1/2}(\varepsilon) \big\} \; \alpha(r),
\end{equation}
where $ \alpha(r) $ is exponentially small in $\eps$ (see the
detailed definition in Appendix and the estimate \eqref{ineq:alpha
estimate}).

So, provided that \eqref{eq: assumption B} is satisfied, and if
\begin{equation}\label{dddd1}
\big( {\deps} \,+ \, \meps \varepsilon^2 \big) \,
\varepsilon^{-2} \ll \alpha^{-1} \; , \qquad \deps
\varepsilon^{-4} \gamma(\varepsilon) \ll \alpha^{-2}
\qquad\mbox{and}\qquad \meps \,\varepsilon^{-2}
\,\gamma^{1/2}(\varepsilon) \ll \alpha^{-1},
\end{equation}
by \eqref{eq: est8 mathcalA} we have
$$ |\dot{\xi}_i| = \mathcal O(e^{-c/r}),$$ and the solution $ \tilde{u} $ stay in the channel
for an exponentially long time.

{ Note that any $\delta(\eps)$, $\mu(\eps)$ of polynomial or
negative polynomial order in $\eps$  satisfy \eqref{dddd1}.}
\end{proof}

\section{Appendix}
In \S\ref{sec1Ap2}-\S\ref{sec1Ap2} we prove various estimates for the non
mass-conserving manifold approximation used throughout this
paper, and collect together existing results thereof from the
work of Carr and Pego, \cite{CarrPego}. Some of the estimates
have been also proven in \cite{AntonBlomkerKarali} and then used
in their integrated version for the mass-conserving case. Then in \S\ref{sectionWP} we derive certain apriori energy estimates for  establishing the well-posedness of the mass-conserving problem considered in \S\ref{section:MSP}.

\subsection{Estimates for the stationary Dirichlet problem (\ref{eq: Bistabe BVP})}\label{sec1Ap}
As it is clear from the definition \eqref{eq: def of uh}, many of
our subsequent estimates involving $ u^h, $ rest upon certain
properties of the stationary states $ \phi $ of
\eqref{eq:deltaCH-AC}, namely the solutions of the Dirichlet
problem \eqref{eq: Bistabe BVP}. In this section we record these
properties and for their proof we refer to \cite{CarrPego}.

Since $ \phi_{\scriptscriptstyle\varepsilon}(0, \ell, \pm1) $
depends on $ \varepsilon $ and $ \ell $ only through the ratio
$\; \mathtt{r}  = \varepsilon/\ell, $ we may define
\begin{equation}\label{eq: def alpha beta}
\alpha_{\pm}( \mathtt{r}) :=
F\big(\phi_{\scriptscriptstyle\varepsilon}(0,\ell,
\pm1)\big),\qquad \beta_{\pm}(\mathtt{r}) := 1 \mp
\phi_{\scriptscriptstyle\varepsilon}(0,\ell, \pm1).
\end{equation}

In what follows, $ C $ will denote a positive constant not
necessarily the same at each occurrence and we stress that $ C $
is independent of $ \varepsilon, x, h_j$'s, $ j$'s.

\begin{proposition}[{\cite[Proposition 3.4]{CarrPego}}]\label{Proposition 3.4 CarrPego}There exists $ r_0>0 $ such that if $ 0 < r < r_0, $ then
\begin{eqnarray}
\alpha_{\pm}(r) & = & \frac{1}{2} K_{\pm}^2 A_{\pm}^2
\operatorname{exp}\Big(\frac{-A_{\pm}}{r}\Big) \left[ 1+\mathcal
O\left(r^{-1}\operatorname{exp}\bigl(\frac{-A_{\pm}}{2r}\bigr)\right)
\right] , \label{ineq:alpha estimate}
\\
\beta_{\pm}(r) & = & K_{\pm}
\operatorname{exp}\Big(\frac{-A_{\pm}}{2r}\Big) \left[ 1+\mathcal
O\left(r^{-1}\operatorname{exp}\bigl(\frac{-A_{\pm}}{2r}\bigr)\right)
\right] ,\label{ineq:beta estimate}
\end{eqnarray}
where
\begin{eqnarray}
A_{\pm} & := & f'(\pm 1) > 0,
\\
K_{\pm} & := & 2 \, \operatorname{exp} \left[ \int_0^1\Big(\frac{
A }{ \sqrt{2 F(\pm t)} } - \frac{1}{1-t}\Big) \, dt \right],
\end{eqnarray}
with $ A:=\min\{A_+, A_-\}, $ and the asymptotic formulas
\eqref{ineq:alpha estimate}, \eqref{ineq:beta estimate} also hold
when they are differentiated a finite number of times, e.g.
\begin{equation}\label{eq: 3 Prop 3.4 CP}
\alpha_+'(r)= - A_+ \, r^{-2} \, \alpha_+(r) \, \left[ 1+\mathcal
O\left(r^{-1}\operatorname{exp}\bigl(\frac{-A_+}{2r}\bigr)\right)
\right].
\end{equation}
\end{proposition}

\begin{proposition}[{\cite[Lemma 7.4]{CarrPego}}]\label{lemma7.4 CP}Let $ 0 < r < r_0. $ Then there exist constants
$ C_1 \in (0,1) $ and $ C_2 > 0 $ such that, for $ |x\pm \frac{\ell}{2}|\leq \varepsilon,$
\begin{subequations}\label{ineq: lemma7.4 CP}
\begin{eqnarray}
|\phi(x, \ell, \pm 1)| & \leq & C_1 , \label{ineq:a lemma7.4 CP}
\\
F(\phi(x, \ell, \pm1)) & \geq & C_2. \label{ineq:b lemma7.4 CP}
\end{eqnarray}
\end{subequations}
\end{proposition}

\begin{proposition}[{\cite[Lemma 7.5]{CarrPego}}]\label{lemma7.5 CP}For $ |x|\leq 2\varepsilon, $ we have
\begin{subequations}\label{ineq: lemma7.5 CP}
\begin{eqnarray}
|\phi(x, \ell, \pm 1) - (-1)^j| & \leq & C \beta(r) ,
\label{ineq:a lemma7.5 CP}
\\
|\phi_x(x, \ell, \pm1)| & \leq & C \varepsilon^{-1} \beta(r).
\label{ineq:b lemma7.5 CP}
\end{eqnarray}
\end{subequations}
\end{proposition}

\begin{proposition}[{\cite[Lemma 7.7]{CarrPego}}]\label{lemma7.7 CP}We have
\begin{equation}\label{eq: 1 Lemma 7.7 CP}
\int_{-\ell/2}^{\ell/2} |\phi_x| \, dx \leq 2, \qquad
\int_{-\ell/2}^{\ell/2} |\phi_{xx}| \, dx \leq C\varepsilon^{-1},
\qquad \int_{-\ell/2}^{\ell/2} |\phi_{xx}|^2 \, dx \leq
C\varepsilon^{-3}.
\end{equation}
\end{proposition}

Beside the above estimates for $ \phi $ and its derivatives with
respect to $ x,$ we will also need estimates on the derivatives $
\phi_\ell(x, \ell, \pm1):=\tfrac{\partial}{\partial\ell}\phi(x,
\ell, \pm1).$

\begin{proposition}[{\cite[Lemma 7.8]{CarrPego}}]\label{lemma7.8 CP}For $ x \in [-\ell, \ell], $
\begin{equation}\label{eq:Lemma 7.8 CP}
\phi_\ell\big(x, \ell, \pm1\big) \,= \, - \frac{1}{2}
\operatorname{sgn}(x)\,\phi_x\big(x, \ell, \pm1\big) \,+\,
\mathrm{w}\big(x, \ell, \pm1\big),
\end{equation}
where, for $ x \neq 0,$
\begin{equation}\label{eq: 2 Lemma 7.8 CP}
\mathrm{w}\big(x, \ell, \pm1\big) \;=\; \varepsilon^{-1} \,
\ell^{-2} \, \alpha'_\pm(r) \,\phi_x(|x|, \ell, \pm1)\,
\int_{\ell/2}^{|x|} \phi_x\big(s, \ell, \pm1\big)^{-2} \, ds,
\end{equation}
and
\begin{equation}\label{eq: 3 Lemma 7.8 CP}
 \mathrm{w}\big(0, \ell, \pm1\big) \;=\;
 \frac{-\varepsilon^{-1} \, \ell^{-2} \, \alpha'_\pm(r)}{\phi_{xx}(0, \ell, \pm1)}.
\end{equation}
\end{proposition}

\begin{proposition}[{\cite[Lemma 7.9]{CarrPego}}]\label{prop7.9 CP}Let $ \mathrm{w} $ be defined
in Proposition \ref{lemma7.8 CP}. There exists $ r_0>0 $ such that if $ 0 < r < r_0, $ then
\begin{eqnarray}
 |\mathrm{w}\big(x, \ell, \pm1\big)| & \leq & C \varepsilon^{-1} \beta_\pm(r), \qquad \mbox{for}\quad
 x \in \big[-\frac{\ell}{2} - \varepsilon, \, \frac{\ell}{2} + \varepsilon\big],
\label{eq: 1 Lemma 7.9 CP}
\\
\ |\mathrm{w}\big(x, \ell, \pm1\big)| & \leq &  C
\varepsilon^{-1} \alpha_\pm(r), \qquad \mbox{for} \quad \big|x \pm
\frac{\ell}{2}\big|  < \varepsilon. \label{eq: 2 Lemma 7.9 CP}
\end{eqnarray}
\end{proposition}

\begin{lemma}[{\cite[Lemma 7.10]{CarrPego}}]\label{lemma 7.10 CP} For $ x \in [-\tfrac{\ell}{2} - \varepsilon,
\, \tfrac{\ell}{2} + \varepsilon], \, x\neq 0,$
\begin{equation}\label{eq: Lemma 7.10 CP}
\bm{\big|} \mathrm{w}_x\big(x, \ell, \pm 1\big) \bm{\big|} \; = \;
\left| \phi_{\ell x}\big(x, \ell, \pm 1\big) + \frac{1}{2}
\operatorname{sgn}(x)\,  \phi_{xx}\big(x, \ell, \pm 1\big) \right|
\; \leq \; C \,\varepsilon^{-2}\, r^{-1}\, \beta_{\pm}(r).
\end{equation}
\end{lemma}

One may show that $ \mathrm{w} $ is $ C^2 $ on $ [0, \ell] $ and
satisfies (see \cite[(7.19)]{CarrPego})
\begin{equation}
\varepsilon^2 \mathrm{w}_{xx} = f'(\phi) \mathrm{w},
\end{equation}
which together with \eqref{eq: 1 Lemma 7.9 CP} yields
\begin{equation}\label{eq:2 prop 2.10 BX}
\bm{\big|} \mathrm{w}_{xx}\big(x, \ell, \pm 1\big) \bm{\big|}
 \leq  C \,\varepsilon^{-3}\, \beta_{\pm}(r).
\end{equation}

\subsection{Estimates on the states \itshape u\textsuperscript{h}}\label{sec1Ap2}
For $ j =1 ,2, \ldots, N+1, $ and the $ \ell_j $ that are given
in \eqref{eq: def ellj}, we set
\begin{equation}
r_j  :=  \frac{\varepsilon}{\ell_j},
\end{equation}
and
\begin{equation}\label{eq: alphabetapm}
\alpha^j  :=
\begin{cases}
\alpha_+(r_j), & \mbox{for $ j $ even},
\\
\alpha_-(r_j), & \mbox{for $ j $ odd},
\end{cases}
\qquad\mbox{and}\qquad \beta^j  :=
\begin{cases}
\beta_+(r_j), & \mbox{for $ j $ even},
\\
\beta_-(r_j), & \mbox{for $ j $ odd}.
\end{cases}
\end{equation}

We also set
\begin{equation}
r  :=  \max\limits_{\scriptscriptstyle 1\leq j \leq N+1} r_j =
\frac{\varepsilon}{\min_j \ell_j},
\end{equation}
and
\begin{equation}\label{eq: alphabeta}
\alpha(r)  :=  \max\limits_{\scriptscriptstyle 1\leq j \leq N+1}
\alpha^j, \qquad\mbox{and}\qquad \beta(r)  :=
\max\limits_{\scriptscriptstyle 1\leq j \leq N+1} \beta^j.
\end{equation}

From the first estimate in \eqref{eq: 1 Lemma 7.7 CP} and the
zero boundary values in \eqref{eq: Bistabe BVP} we deduce that $
|\phi| \leq 2 $ on $ [-  \tfrac{\ell}{2},  \tfrac{\ell}{2}]. $
Therefore, for each $ j=1,\ldots,N+1, $
\begin{equation}\label{eq: phi leq 2}
|\phi^j| \leq 2 \qquad\mbox{on}\quad
 \bigl[m_j -  \tfrac{\ell_j}{2}, \, m_j + \tfrac{\ell_j}{2}\bigr],
\end{equation}
and as a consequence of the definition \eqref{eq: def of uh}, $
u^h $ is uniformly bounded on $ [0, 1], $ thus $ f(u^h) $ and $
f'(u^h) $ are uniformly bounded too.

Similarly, from the second estimate in \eqref{eq: 1 Lemma 7.7 CP}
we get that
\begin{equation}\label{eq: phi x leq ceps}
|\phi^j_x| \leq C\varepsilon^{-1} \qquad \mbox{on} \quad \bigl[m_j -  \tfrac{\ell_j}{2},\, m_j + \tfrac{\ell_j}{2}\bigr].
\end{equation}

By \eqref{eq: phij solves BS} and \eqref{eq: phi leq 2} we get
\begin{equation}\label{eq: phi xx leq ceps}
|\phi^j_{xx}| \leq C\varepsilon^{-2} \qquad\mbox{on}\quad
 \bigl[m_j -  \tfrac{\ell_j}{2},\, m_j + \tfrac{\ell_j}{2}\bigr],
\end{equation}
and a differentiation of \eqref{eq: phij solves BS} together with
\eqref{eq: phi leq 2}, \eqref{eq: phi x leq ceps} yields
\begin{equation}\label{eq: phi xxx leq ceps}
|\phi^j_{xxx}| \leq C\varepsilon^{-3} , \qquad \mbox{on} \quad \bigl[m_j -  \tfrac{\ell_j}{2},\, m_j + \tfrac{\ell_j}{2}\bigr],
\end{equation}
and in general, we may see that, see also in
\cite{AntonBlomkerKarali}
\begin{equation}\label{eq: phi nx leq ceps}
\Big| \partial_x^n \phi^j \Big| \leq C \varepsilon^{-n}, \qquad
\quad \mbox{on}\quad \bigl[m_j -  \tfrac{\ell_j}{2}, \; m_j +
\tfrac{\ell_j}{2}\bigr].
\end{equation}

By Proposition \ref{lemma7.5 CP}, we have
\begin{subequations}\label{ineq: lemma 7.5 CP}
\begin{eqnarray}
|\phi^j(x) - (-1)^j| & \leq & C \beta(r) , \label{ineq:a lemma
7.5 CP}
\\
|\phi^j_x(x)| & \leq & C \varepsilon^{-1} \beta(r),
\qquad\qquad\mbox{for}\quad |x-m_j| < 2 \varepsilon.
\label{ineq:b lemma 7.5 CP}
\end{eqnarray}
\end{subequations}
As a consequence of \eqref{ineq:a lemma 7.5 CP},  we have
\begin{equation}\label{ineq: cons1 of lemma 7.5 CP}
|f(\phi^j)| \,=\, |f(\phi^j) - f((-1)^j)| \leq C \beta(r), \qquad\qquad \mbox{for} \quad |x-m_j|<2\varepsilon,
\end{equation}
which on its turn together with \eqref{eq: phij solves BS} implies
\begin{equation}\label{ineq: cons2 of lemma 7.5 CP}
| \phi_{xx}^j | \leq C \varepsilon^{-2} \beta(r), \qquad\qquad
\mbox{for} \quad |x-m_j|<2\varepsilon.
\end{equation}

\begin{proposition}[{\cite[Lemma 8.2]{CarrPego}}]\label{lemma:8.2 CP}Let $ r_0 > 0 $ be sufficiently small. There exist
constants $ C_1, C_2 $ such that if we assume that $ \varepsilon/\ell_j < r_0 $ and $ \varepsilon/\ell_{j+1} < r_0 $
for $ j \in \{1, 2, \ldots, N\}, $ then
\begin{subequations}\label{eq: lemma 8.2 CP}
\begin{eqnarray}
\bm{\big|} \phi^j(x) \, - \, \phi^{j+1}(x) \bm{\big|} & \leq &
C_1\, \bm{\big|} a^j \, - \, a^{j+1} \bm{\big|}, \label{eq: a
lemma 8.2 CP}
\\
\bm{\Big|} \phi^j_x(x) \, - \, \phi^{j+1}_x(x) \bm{\Big|} & \leq
& C_2 \,\varepsilon^{-1}\, \bm{\Big|} a^j \, - \, a^{j+1}
\bm{\Big|}, \qquad\qquad \mbox{for} \quad |x-h_j|<\varepsilon.
\label{eq: b lemma 8.2 CP}
\end{eqnarray}
\end{subequations}

\end{proposition}

Moreover, by \eqref{eq: phij solves BS}, mean value theorem and
\eqref{eq: a lemma 8.2 CP} we have, with some $ \theta_x $
between $ \phi^j(x)$ and $ \phi^{j+1}(x), $
\begin{eqnarray}
\bm{|} \phi^j_{xx}(x) - \phi^{j+1}_{xx}(x)\bm{|} & = &
\varepsilon^{-2} \bm{|} f(\phi^j(x))  -f(\phi^{j+1}(x)) \bm{|}
\nonumber
\\
&=& \varepsilon^{-2} \bm{|} f'(\theta_x) \bm{|} \,
\bm{|}\phi^{j+1} - \phi^j\bm{|} \nonumber
\\
& \leq & C \varepsilon^{-2} \bm{\big|} a^j \, - \, a^{j+1}
\bm{\big|}, \qquad\qquad \mbox{for} \quad |x-h_j|<\varepsilon.
\label{ineq: estimate diff phij}
\end{eqnarray}

By differentiating \eqref{eq: phij solves BS}, we may proceed
recursively to get,  for $ n=1,2,3, \cdots,$
\begin{eqnarray}\label{ineq: estim gen diff phij}
\bm{|} \partial^n_x\phi^j - \partial^n_x\phi^{j+1}\bm{|} & \leq &
C \varepsilon^{-n} \bm{\big|} a^j \, - \, a^{j+1} \bm{\big|},
\qquad\qquad \mbox{for} \quad |x-h_j|<\varepsilon. \label{eq:
11.3a flow in channel}
\end{eqnarray}

Considering the smooth cut-off function $ \chi^j $ let us
notice that, for $ n=1,2,3,\cdots, $
\begin{equation}\label{eq: deriv cutoff}
\Big| \frac{d^n }{dx^n} \chi^j \Big| \leq C \varepsilon^{-n}.
\end{equation}
By \eqref{eq: derivative of uh}, \eqref{eq: phi x leq ceps} and
\eqref{eq: deriv cutoff} we have
\begin{equation}\label{eq: uhx leq eps}
|u^h_{x}| \leq C \varepsilon^{-1}, \qquad \mbox{on} \quad [0, 1],
\end{equation}
and by \eqref{eq: uxxh}, \eqref{eq: phi xx leq ceps}, \eqref{eq:
lemma 8.2 CP}, \eqref{eq: deriv cutoff} we easily get
\begin{equation}\label{eq: uhxx leq eps}
|u^h_{xx}| \leq C \varepsilon^{-2}, \qquad \mbox{on} \quad [0, 1].
\end{equation}

Differentiating \eqref{eq: uxxh} we immediately get
\begin{equation}\label{eq: u xxx form}
u_{xxx}^h =
\begin{cases}
\; \phi_{xxx}^j  &\quad  , m_j \leq x \leq h_j-\varepsilon ,
\\
\chi_{xxx}^j \, \big(\phi^{j+1}-\phi^j\big)  + 3 \chi_{xx}^j \,
\big(\phi^{j+1}_x-\phi^j_x\big)  + 3 \chi_x^j \,
\big(\phi_{xx}^{j+1}-\phi_{xx}^j\big)  &
\\
\phantom{ \chi_{xxx}^j \, \big(\phi^{j+1}-\phi^j\big) } +
(1-\chi^j)\phi^j_{xxx}
 + \chi^j\phi^{j+1}_{xxx} &\quad   , |x-h_j| < \varepsilon ,
\\
\; \phi_{xxx}^{j+1}  &\quad   , h_j +\varepsilon \leq x \leq
m_{j+1}.
\end{cases}
\end{equation}

By \eqref{eq: phi xxx leq ceps}, \eqref{ineq: estim gen diff
phij}, \eqref{eq: deriv cutoff}, \eqref{eq: u xxx form}, we
easily obtain
\begin{equation}\label{ineq: uh3x leq eps3}
|u^h_{xxx}| \leq C \varepsilon^{-3}, \qquad \mbox{on} \quad [0, 1].
\end{equation}

Also for a smooth function $ \mathcal{F} = \mathcal{F}(s), \, s
\in [0,1], $ it is straightforward to show that the remainder $
R(\chi) $ of the linear Lagrange interpolation of $ \mathcal{F} $
at $ s = 0 $ and $ s = 1, $
\begin{equation}\label{eq: Lagrange interp}
R(x) \,:=\, \mathcal{F}(x) \,-\, (1-x) \mathcal{F}(0) - x
\mathcal{F}(1), \qquad x \in [0,1],
\end{equation}
is given by
\begin{equation}\label{eq: Lagrange interp2}
R(x) \;=\; (1-x) \int_0^{\chi} s \mathcal{F}''(s)\, ds \;+\; x
\int_{x}^1 (1-s) \mathcal{F}''(s)\, ds.
\end{equation}

We now use \eqref{eq: derivative of uh} and employ \eqref{eq:
Lagrange interp}-\eqref{eq: Lagrange interp2} for the function
\begin{equation}\label{eq: appl F Lagr interp}
\mathcal{F}(s) := f((1-s)\phi^j + s \phi^{j+1}),
\end{equation}
to get
\begin{equation}\label{eq: appl F Lagr interp2}
\mathscr L^b(u^h) = \varepsilon^2 \chi_{xx}^j \,
\big(\phi^{j+1}-\phi^j\big) \; + \; 2 \varepsilon^2 \chi_x^j \,
\big(\phi_x^{j+1}-\phi_x^j\big) \;+\; R^j, \qquad \mbox{for} \quad |x-h_j| < \varepsilon,
\end{equation}
where
\begin{equation}\label{eq: rem Lagr interp}
R^j = \big(\phi^{j+1}-\phi^j\big)^2 \left[ (1-\chi^j)
\int_0^{\chi^j} s f''\big(\theta(s)\big) \, ds \; + \; \chi^j
\int_{\chi^j}^1 (1-s) f''\big(\theta(s)\big) \, ds \right],
\end{equation}
and $ \theta(s):=(1 - s) \phi^j + s \phi^{j+1}.$

We then combine \eqref{eq: appl F Lagr interp2}-\eqref{eq: rem
Lagr interp} with \eqref{eq: lemma 8.2 CP}, \eqref{eq: deriv
cutoff} to conclude that
\begin{equation}\label{ineq: Lbuh leq alpha}
|\mathscr L^b(u^h)| \leq C \alpha(r), \qquad  \mbox{for} \quad
|x-h_j| < \varepsilon ;
\end{equation}
cf. \cite[Theorem 3.5]{CarrPego}.

At this point let us recall \eqref{eq: uh sat BS}, i.e. $
\mathscr L^b(u^h)=0, $ for  $ |x-h_j| \geq \varepsilon, $ which
together with boundary values \eqref{eq: zero Neumann BV uh} and
\eqref{ineq: Lbuh leq alpha} show that $ u^h $ ``almost" satisfy
the steady-state problem \eqref{BVP: steady states}.

\begin{remark}\label{remujh}\upshape To show that
\begin{equation}\label{eq:approx der}
 u_j^h \sim -u_x^h, \qquad \mbox{as} \quad r \to 0, \qquad \mbox{uniformly on} \quad I_j :=[m_j,\, m_{j+1}]
\end{equation}
first notice that only $ \phi^j $ and $ \phi^{j+1} $ depend on $
h_j, $ so the support of $ u_j^h $ is contained in $
[h_{j-1}-\varepsilon, \, h_{j+1} + \varepsilon]. $

Applying \eqref{eq:Lemma 7.8 CP} for the translate $ \phi^j $ of
$ \phi, $
$$
\phi^j(x) := \phi\big(x - \tfrac{h_{j-1} + h_j}{2}, \, h_j -
h_{j-1}, \, (-1)^j\big) , \qquad x \in [h_{j-1} - \varepsilon, \,
h_j+\varepsilon],
$$
we have, for $ x \in [h_{j-1} - \varepsilon, \, h_j+\varepsilon],
$
\begin{eqnarray}
 \frac{\partial }{\partial h_j} \phi^j
& = & \phi^j_x \, \frac{\partial }{\partial
h_j}\left(x-\tfrac{h_{j-1} + h_j}{2}\right) \, + \, \phi^j_\ell
\,\frac{\partial }{\partial h_j} \big(h_j - h_{j-1}\big) \nonumber
\\
& = & - \, \frac{1}{2} \phi^j_x \, - \, \frac{1}{2}
\operatorname{sgn}(x-m_j)\, \phi_x^j \,+\,  \mathrm{w}^j ,
\label{eq:a derivative phi j wrt hj}
\end{eqnarray}
therefore
\begin{equation}\label{eq:b derivative phi j wrt hj}
 \frac{\partial }{\partial h_j} \phi^j
 \;=\;
- \phi^j_x \,+\,  \mathrm{w}^j , \qquad\mbox{in} \quad
I_j:=[m_j,\, h_j+\varepsilon] ,
\end{equation}
since $ \operatorname{sgn}(x-m_j) > 0. $

Similarly, we obtain
\begin{equation}\label{eq: deriv phi jplus1 wrt hj on Ij}
 \frac{\partial }{\partial h_j} \phi^{j+1}
 \,=\,
 -  \phi^{j+1}_x
\,-\,  \mathrm{w}^{j+1} , \qquad\mbox{for} \quad x \in [h_j, \,
h_{j+1}],
\end{equation}
and thus
\begin{equation}\label{eq: 7 deriv diff wrt h}
 \frac{\partial }{\partial h_j} \big(\phi^{j+1} -  \phi^j\big)
\, = \,
 \phi^j_x
\, - \, \phi^{j+1}_x \, - \,  \mathrm{w}^j \,- \,
\mathrm{w}^{j+1} , \qquad\mbox{in} \quad I_j.
\end{equation}

So recalling that $ u^h = \big(1 - \chi^j\big) \, \phi^j + \chi^j
\, \phi^{j+1}, $ and noticing that $ \chi^j_x = - \chi_{h_j}^j, $
it is straightforward to see that
\begin{eqnarray}
\frac{\partial }{\partial h_j} u^h & = & \chi^j_x \big( \phi^j
\;-\; \phi^{j+1} \big) \; + \;  \chi^j \frac{\partial }{\partial
h_j}\big(\phi^{j+1} - \phi^j\big) \;+\; \frac{\partial }{\partial
h_j} \phi^j \nonumber
\\
&\stackrel{\eqref{eq: 7 deriv diff wrt h}}{=\joinrel=}& \chi^j_x
\big( \phi^j \;-\; \phi^{j+1} \big) \; + \;  \chi^j \big[
\phi^j_x \, - \, \phi^{j+1}_x \, - \,  \mathrm{w}^j \, - \,
\mathrm{w}^{j+1}\big] \, - \, \phi^j_x \, + \,  \mathrm{w}^j
\nonumber
\\
& = & \chi^j_x \big( \phi^j \;-\; \phi^{j+1} \big) \; + \; \chi^j
\big[ \phi^j_x \, - \, \phi^{j+1}_x\big] \;-\; \chi^j
\big[\mathrm{w}^j \,+ \,  \mathrm{w}^{j+1}\big] \, - \, \phi^j_x
\, + \,  \mathrm{w}^j \nonumber
\\
&\stackrel{\eqref{eq: derivative of uh}}{=\joinrel=}& - u_x^h \;
+ \; \big(1-\chi^j\big) \mathrm{w}^j \;-\;  \chi^j
\mathrm{w}^{j+1} \qquad  \mbox{for} \quad |x-h_j|<\varepsilon.
\label{eq: uhj 1}
\end{eqnarray}

For the translate $ \phi^{j+1} $ of $ \phi, $
$$
\phi^{j+1}(x) := \phi\big(x - \tfrac{h_j + h_{j+1}}{2}, \,
h_{j+1} - h_j ; \, (-1)^{j+1}\big) \qquad x \in [h_j -
\varepsilon, \, h_{j+1} + \varepsilon]
$$
we have, by \eqref{eq:Lemma 7.8 CP},
\begin{eqnarray}
 \frac{\partial }{\partial h_j} \phi^{j+1}
& = & \phi^{j+1}_x \, \frac{\partial }{\partial
h_j}\left(x-\tfrac{h_j + h_{j+1}}{2}\right) \, + \,
\phi^{j+1}_\ell \,\frac{\partial }{\partial h_j} \big(h_{j+1} -
h_j\big) \nonumber
\\
& = & - \, \frac{1}{2} \phi^{j+1}_x \, + \, \frac{1}{2}
\operatorname{sgn}(x-m_{j+1})\, \phi_x^{j+1} \, - \,
\mathrm{w}^{j+1} \;=\;
 - \,  \mathrm{w}^{j+1} ,
 \qquad \mbox{in} \quad [m_{j+1} , \, h_{j+1} + \varepsilon],
\label{derivative phi j+1 wrt hj on Ij+1}
\end{eqnarray}
since $ \operatorname{sgn}(x-m_{j+1}) > 0. $

Recall that $$u^h = \big(1 - \chi^{j+1}\big) \, \phi^{j+1} +
\chi^{j+1} \, \phi^{j+2}, \qquad x \in [m_{j+1}, \; m_{j+2}],  $$
so using \eqref{derivative phi j+1 wrt hj on Ij+1} and noticing
that $ \chi_{h_j}^{j+1}=0=\phi_{h_j}^{j+2}, $ it is
straightforward to see that
\begin{equation}\label{derivative uh wrt hj on Ij+1}
\frac{\partial }{\partial h_j} u^h \;=\;  \big(1 -
\chi^{j+1}\big) \frac{\partial }{\partial h_j} \phi^{j+1}
\stackrel{\eqref{derivative phi j+1 wrt hj on Ij+1}}{=\joinrel=}
 - \,  \big(1 - \chi^{j+1}\big)  \mathrm{w}^{j+1}, \qquad x \in [m_{j+1}, \; h_{j+1}+\varepsilon].
\end{equation}

Analogously, taking into account that $ u^h = \big(1 -
\chi^{j-1}\big) \, \phi^{j-1} + \chi^{j-1} \, \phi^j \, $ for $ x
\in [m_{j-1}, \; m_j],  $ using \eqref{eq:a derivative phi j wrt
hj} and noticing that $ \chi_{h_j}^{j-1}=0=\phi_{h_j}^{j-1}, $
we obtain that
\begin{equation}\label{derivative uh wrt hj on Ij-1}
\frac{\partial }{\partial h_j} u^h \;=\; \chi^{j-1}
\frac{\partial }{\partial h_j} \phi^j \; = \;
  \chi^{j-1}  \mathrm{w}^j, \qquad x \in [h_{j-1}-\varepsilon, \, m_j].
\end{equation}
Gathering \eqref{eq: uhj 1}, \eqref{derivative uh wrt hj on
Ij+1}, \eqref{derivative uh wrt hj on Ij-1} we have that
\begin{equation}\label{eq: derivative uh wrt h}
 \frac{\partial }{\partial h_j} u^h
\, = \,
\begin{cases}
\chi^{j-1} \mathrm{w}^j , &  \mbox{for} \quad
h_{j-1}-\varepsilon \leq x \leq m_j,
\\
 - u^h_{x}
\,+\,  \mathrm{w}^j,  &  \mbox{for} \quad  m_j \leq x \leq
h_j-\varepsilon ,
\\
- u_x^h \; + \; \big(1-\chi^j\big) \mathrm{w}^j \;-\; \chi^j
\mathrm{w}^{j+1}, & \mbox{for} \quad  |x-h_j| < \varepsilon ,
\\
-  u^h_{x} \,-\,  \mathrm{w}^{j+1} , &  \mbox{for} \quad h_j
+\varepsilon \leq x \leq m_{j+1},
\\
- \, (1 - \chi^{j+1}) \mathrm{w}^{j+1} , &  \mbox{for} \quad
m_{j+1} \leq x \leq h_{j+1} +\varepsilon.
\end{cases}
\end{equation}
Then \eqref{eq:approx der} follows from \eqref{eq: derivative uh
wrt h} combined with \eqref{eq: 1 Lemma 7.9 CP}-\eqref{eq: 2
Lemma 7.9 CP}. See also in \cite{AntonBlomkerKarali}, for some
analogous results.
\end{remark}

\subsection{Apriori estimates for the problem (\ref{eq:deltaCH-AC})-(\ref{eq:BC1p})-(\ref{eq:BC2p})-(\ref{eq:MC})}\label{sectionWP}
For the well-posedness of the initial and boundary value problem we may argue as in \cite[\S2]{ElliottSongmu}; next we derive the estimates needed in our case where we have replaced the b.c. at $ x = 1 $ with the mass conservation condition and added the Allen-Cahn lower order term in the pde.

Local in time existence may be proved by fixed-point theory, applying a Picard-type iteration scheme. In order to prove global existence, i.e. existence on $ [0, T] $ for any $ T > 0, $ we need to derive certain apriori uniform estimates on $ u. $ 
To this aim, first notice that by \eqref{eq:MC}, \eqref{eq:deltaCH-AC2} and \eqref{eq:BC1p}-\eqref{eq:BC2p} we have
\begin{eqnarray*}
0
& = &
\frac{d}{dt}
\int_0^1 u(x, t) \, dx
\; = \;
\int_0^1 u_t(x, t) \, dx
\\
& = &
- \deps
\int_0^1
\big(\varepsilon^2 u_{xx} - W'(u)\big)_{xx}
\, dx
\; + \;
\meps
\int_0^1
\big(\varepsilon^2 u_{xx}  - W'(u)\big)
\, dx
\\
& = &
- \deps \varepsilon^2 u_{xxx}(1,t)
\; - \;
\meps  \int_0^1 W'(u) \, dx
\end{eqnarray*}
so we have
\begin{equation}\label{CompatCond2}
\meps  \int_0^1 W'(u) \, dx = - \deps \varepsilon^2 u_{xxx}(1,t).
\end{equation}
Also, as in the proof of Lemma \ref{lemma 4.1 BX}, we can see that for differentiable $ \upsilon $ and any positive $ \epsilon_1, $
\begin{equation}\label{inequ1square}
\upsilon^2(1,t)  \leq \|\upsilon\|_\infty^2
\leq
2\epsilon_1 \|\upsilon_x\|^2
\, + \,
\frac{2}{\epsilon_1} \,
\|\upsilon\|^2.
\end{equation}
For the special case of
\begin{equation}\label{eq:specialValue} W(u) = \frac{1}{4}(u^2-1)^2 , \qquad\mbox{thus}\qquad W'(u) = u^3 - u ,
\end{equation}
we see that,
\begin{equation*}
W'(u) \leq c_1 W(u) + c_2 , \qquad \forall u \in \mathbb R,
\end{equation*}
for some positive constants $ c_1, c_2 $ independent of $ u, $ and so
\begin{equation}\label{L1WprimeleqL1W}
\int_0^1 |W'(u)|\, dx \leq c_1 \int_0^1 W(u) \, dx + c_2.
\end{equation}

\textit{Growth estimate for energy:} We set
\begin{equation}
E(t) := \int_0^1 \frac{\eps^2}{2} u_x^2 + W(u)\, dx
\end{equation}
and we have
\begin{eqnarray*}
\frac{d}{dt}
E(t)
& = &
\int_0^1 \eps^2 u_x \, (u_t)_x \, + \, W'(u)\,u_t \, dx
\\
&=&
-\int_0^1 \big(\eps^2 u_{xx} \, - \, W'(u)\big) \, u_t \, dx,
\end{eqnarray*}
where we integrated by parts the first term and applied \eqref{eq:BC1p}. Then, by the pde \eqref{eq:deltaCH-AC2} and integrations by parts combined with  \eqref{eq:BC1p}-\eqref{eq:BC2p} we get
\begin{flalign}
\frac{d}{dt}
E(t)
=&
\int_0^1 \big(\eps^2 u_{xx} \, - \, W'(u)\big)\, \Big[\deps \, \big(\varepsilon^2 u_{xx} - W'(u)\big)_{xx} \; - \;   \meps \, \big(\varepsilon^2 u_{xx}  - W'(u)\big)\Big] \, dx
\nonumber
\\
=&
-  \deps\int_0^1\left[\big(\eps^2 u_{xx} \; - \;W'(u)\big)_x\right]^2\, dx\; - \meps \int_0^1\left[\eps^2 \, u_{xx}
\; - \;W'(u)\right]^2\, dx
\nonumber
\\
&
+
\deps \, \eps^2 \, u_{xxx}(1,t) \,
\Big[
\eps^2 \, u_{xx}(1,t)
\;-\;
W'\big(u(1,t)\big)\Big]
\nonumber
\\
\stackrel{\scriptscriptstyle\eqref{CompatCond2}}{=} &
- \deps
\left\|
\big(\eps^2 u_{xx}
\; - \;
W'(u)\big)_x
\right\|^2
\; -
\meps
\left\|
\eps^2 \, u_{xx}
\; - \;
W'(u)
\right\|^2
\nonumber
\\
&
- \meps \, \Big[\eps^2 \, u_{xx}(1,t) \;-\;W'\big(u(1,t)\big)\Big]
\, \int_0^1 W'(u) \, dx\nonumber
\\
\leq &- \deps\left\|\big(\eps^2 u_{xx} \; - \;W'(u)\big)_x\right\|^2\; - \; \meps \left\|\eps^2 \, u_{xx}
\; - \;W'(u)
\right\|^2\nonumber
\\
&
+\frac{\meps}{4 \epsilon}
\;\Big[\eps^2 \, u_{xx}(1,t) \;-\;W'\big(u(1,t)\big)\Big]^2\; + \; \epsilon \, \meps \left(\int_0^1 W'(u)\,dx\right)^2 .
\label{step1enest2}
\end{flalign}
By \eqref{eq:BC1p} we have
\begin{equation}\label{step1enest4}
\int_0^1 W'(u)\,dx
=
\int_0^1 \big(W'(u) - \eps^2 u_{xx}\big) \,dx
\leq
\left\| W'(u) - \eps^2 u_{xx}\right\|
\end{equation}
and by \eqref{inequ1square} for $ \upsilon=W'(u) - \eps^2 u_{xx} $ therein, we have
\begin{equation}\label{step1enest5}
\frac{\meps}{4 \epsilon}
\upsilon^2(1,t)
\leq
\frac{\meps \, \epsilon_1}{2 \epsilon}
\|\upsilon_x\|^2
\, + \,
\frac{\meps}{2 \epsilon \, \epsilon_1} \,
\|\upsilon\|^2
\end{equation}
so, choosing $ \epsilon, \epsilon_1, $ so that
\begin{equation}\label{step1enest6}
\frac{\meps \, \epsilon_1}{2 \epsilon}
\leq
\deps
\qquad\mbox{and}\qquad
0\leq 1 -\epsilon - \frac{1}{2 \epsilon \, \epsilon_1}
\end{equation}
we substitute \eqref{step1enest4}, \eqref{step1enest5} into \eqref{step1enest2} to get that
$$ \frac{d}{dt}  E(t) \leq 0.$$
Therefore, $ E(t) \leq E(0) $ that is
\begin{equation}\label{step1enest7}
\frac{\eps^2}{2}
\|u_x\|^2
\, + \,
\int_0^1 W(u) \, dx
\leq
E_0:= \int_0^1 \frac{\eps^2}{2} (u_0)_x^2 + W\big(u_0\big)\, dx
\end{equation}
with $ u_0(x):=u(x,0), $ and so
\begin{equation}\label{step1enest8}
\frac{\eps^2}{2} \|u_x\|^2 \leq E_0
\end{equation}
and
\begin{equation}\label{step1enest9}
\int_0^1 W(u) \, dx \leq E_0.
\end{equation}
Furthermore, integrating \eqref{step1enest2} we get
\begin{equation}\label{step1enest9.1}
\int_0^t\left\|\eps^2 \, u_{xx}
\; - \;W'(u)
\right\|_{H^1}^2 \, d\tau
\leq C,\qquad 0\leq t \leq T,
\end{equation}
for some constant $ C $ depending on $ u_0, T $ and $ \deps, \meps,  \eps.$

\textit{Remark:} A trivial calculus shows that the weakest condition \eqref{step1enest6} for $ \deps, \meps $ is attained by choosing $ \epsilon_1/\epsilon = 27/8 $ for $ \epsilon=2/3, $ and so $ \tfrac{27}{16} \,\meps \,\leq \,\deps. $ Let us also emphasize that the condition $ c \, \meps \,\leq \,\deps $
for some $ c > 0, $ is weaker than the assumption \eqref{eq:a assumption B} for establishing the slow evolution within the channel \eqref{eq: def ourslow
ch} (Theorem \ref{mtat}); e.g. take $ \deps, \meps $ such that $ \eps^{-3/2}\meps \,\ll \,\deps = \mathcal O(\eps^8).$

\textit{Growth estimate for} $\mathit{\|u\|^2:} $  Multiply \eqref{eq:deltaCH-AC2} by $ u, $ then integrate with respect to $ x $ and apply \eqref{eq:BC1p}-\eqref{eq:BC2p} to get
\begin{flalign}
\frac{1}{2}\frac{d}{dt}
\|u\|^2
\;
&
+ \;
\deps \, \varepsilon^2 \, \|u_{xx}\|^2
\; + \; \meps \, \varepsilon^2 \,  \|u_{x}\|^2
\nonumber
\\[0.5em]
=
&
\;
-
\deps \, \eps^2 \, u_{xxx}(1,t) \, u(1,t)
\; - \;
\deps
\int_0^1 \big(W'(u)\big)_x \, u_{x} \, dx
\;-\;
\meps
\int_0^1 W'(u) \, u\, dx
\nonumber
\\[0.5em]
=
&
\;
-
\deps \, \eps^2 \, u_{xxx}(1,t) \, u(1,t)
\; - \;
\deps
\int_0^1 W''(u) \, u^2_{x} \, dx
\;-\;
\meps
\int_0^1 W'(u) \, u\, dx
\nonumber
\\[0.5em] 
\stackrel{\scriptscriptstyle W''\geq -1}{\leq} &
\;
-
\deps \, \eps^2 \, u_{xxx}(1,t) \, u(1,t)
\; + \;
\deps
\int_0^1 u^2_x \, dx
\;-\;
\meps
\int_0^1 W'(u) \, u\, dx
\nonumber
\\
\stackrel{\scriptscriptstyle\eqref{CompatCond2}}{=} &
\;
\deps \, \|u_x\|^2
\;+\;
\meps\left(
u(1,t)
\int_0^1 W'(u) \, dx
\;-\;
\int_0^1 W'(u) \, u\, dx
\right)
\nonumber
\\[0.5em]
\leq  &
\;
\deps \, \|u_x\|^2
\;+\;
\meps
\,
\left(
\|u\|_\infty^2
\;+\;
\left\|W'(u)\right\|^2_1
\right)
\nonumber
\\[0.5em]
\stackrel{\scriptscriptstyle\eqref{inequ1square}}{\leq}  &
\; \deps \, \|u_x\|^2
\;+\;
\meps
\left(
2\, \|u_x\|^2
\, + \,
2 \,
\|u\|^2
\;+\;
\left\|W'(u)\right\|^2_1
\right) .
\label{step2enest1}
\end{flalign}
Regarding the term $ \left\|W'(u)\right\|_1 $ in \eqref{step2enest1}, we combine  \eqref{L1WprimeleqL1W} and \eqref{step1enest9} to see that
\begin{equation}\label{step2enest2}
\left\|W'(u)\right\|_1 \leq c_1 \left\|W(u)\right\|_1 + c_2 \leq C.
\end{equation}
In view of \eqref{step1enest8} and \eqref{step2enest2}, the estimate \eqref{step2enest1} yields
\begin{equation}\label{step2enest3}
\frac{1}{2}\frac{d}{dt}
\|u\|^2
\; + \;
\deps \, \eps^2 \, \|u_{xx}\|^2
\; + \; \meps \, \eps^2 \,  \|u_{x}\|^2
\leq
C_1 \|u\|^2 \,+\, C_2
\end{equation}
the constants $ C_1, C_2 $ depending only on $ u_0 $ and $ \deps, \meps.$

In particular, \eqref{step2enest3} implies
$$
\frac{1}{2} \frac{d}{dt}
\|u\|^2
\leq
C_1 \|u\|^2 \,+\, C_2
$$
and integrating this inequality we get
\begin{equation}\label{step2enest3.1}
\|u(\cdot,t)\|^2 \leq e^{2 C_1 t} \|u_0\|^2 + C_2 (e^{2 C_1 t} - 1) / C_1
\end{equation}
and so
$$ 
\|u(\cdot,t)\|^2\leq c_1 \|u_0\|^2 + c_2, \qquad 0\leq t\leq T,
$$
with $ c_1= e^{2C_1 T},\, c_2=C_2 (e^{2C_1 T} - 1)/C_1, \, $ that is
\begin{equation}\label{step2enest4}
\|u\|\leq C, \qquad 0\leq t \leq T,
\end{equation}
with a constant $ C $ depending only on $ u_0, T $ and $ \deps, \meps. $

By \eqref{step1enest8} and \eqref{step2enest4} we get that
\begin{equation}\label{step2enest5}
\|u(\cdot,t)\|_\infty \leq C, \qquad 0\leq t \leq T.
\end{equation}
Now we return to \eqref{step2enest3}, ignore the positive term $ \|u_{xx}\|, $ then integrate and employ \eqref{step2enest3.1} to get
\begin{equation*}\left\|u(\cdot,T)\right\|^2 \; + \; 2 \, \meps \, \eps^2 \, \int_0^t\|u_x\|^2 \, d\tau \leq \|u_0\|^2 \, e^{2 C_1 t} \, + \, C_3 \big(e^{2 C_1 t}-1\big)\end{equation*} therefore
\begin{equation}\label{step2enest7}
\int_0^t\|u_{x}\|^2 \, d\tau
\leq C,\qquad 0\leq t \leq T,
\end{equation}
for some positive constant $ C $ depending only on $ u_0, T $ and $ \deps, \meps, \eps.$

Returning once more to \eqref{step2enest3}, we get as above that
\begin{equation}\label{step2enest8}
\int_0^t\|u_{xx}\|^2 \, d\tau
\leq C,\qquad 0\leq t \leq T,
\end{equation}
as well.

For improving the regularity of the weak solution to be a classical one we may use a bootstrap argument; see e.g. (2-20)-(2-25) of \cite{ElliottSongmu}. Let us also remark that in view of \eqref{eq:tilde u}, the $H^k$-regularity of $ u $ implies the $H^{k+1}$-regularity for the solution $ \tilde{u} $ of the integrated problem.

\textit{Uniqueness:} Let $ u, v $ be solutions of the problem \eqref{eq:deltaCH-AC2}-\eqref{eq:BC1p}-\eqref{eq:BC2p}-\eqref{eq:MC} 
and consider the difference $ \mathrm{v} = u - v. $ In view of \eqref{eq:deltaCH-AC2}, we have
\begin{equation}\label{Uniq1}
\mathrm{v}_t
=
-
\deps \,
\varepsilon^2 \mathrm{v}_{xxxx}
\, + \,
\deps \,
\big(
W'(u)
\,-\,
W'(v)
\big)_{xx}
\; + \;
\meps \,
\varepsilon^2 \mathrm{v}_{xx}
\, - \,
\meps \,
\big(
W'(u)
\, - \,
W'(v)
\big),
\end{equation}
the \eqref{eq:BC1p}-\eqref{eq:BC2p} yield the boundary conditions
\begin{flalign}
&  &
\mathrm{v}_{x}(0, \, t) =  \mathrm{v}_{x}(1 , \, t) & \; = 0 , 
& 
\label{Uniq2}
\\
& &
\mathrm{v}_{xxx}(0, \, t) & \; = 0, 
&\label{Uniq3}
\end{flalign}
and \eqref{eq:MC} implies
\begin{equation}\label{Uniq4}
\int_0^1 \mathrm{v}(x, t) \, dx = 0, \qquad t>0.
\end{equation}
Multiply the pde \eqref{Uniq1} by $ \mathrm{v}, $ then integrate with respect to $ x $ and apply \eqref{Uniq2}-\eqref{Uniq3} to get
\begin{multline}\label{Uniq5}
\frac{1}{2} \frac{d}{dt} \|\mathrm{v}\|^2
\;+\;
\deps \, \varepsilon^2 \, \|\mathrm{v}_{xx}\|^2
\; + \; \meps \, \varepsilon^2 \,  \|\mathrm{v}_x\|^2
\; = \;
-
\deps \, \eps^2 \, \mathrm{v}_{xxx}(1,t) \, \mathrm{v}(1,t)
\\[0.5em]
+ \;
\deps
\int_0^1 \big(
W'(u)
\, - \,
W'(v)
\big) \, \mathrm{v}_{xx} \, dx
\;-\;
\meps
\int_0^1 \big(
W'(u)
\, - \,
W'(v)
\big) \, \mathrm{v}\, dx.
\end{multline}
Let us next estimate the terms in the rhs of \eqref{Uniq5}. To this aim, we set $ K_{\scriptscriptstyle T} := \sup\big\{\|u(\cdot,t)\|_\infty, \|\upsilon(\cdot,t)\|_\infty :\, 0\leq t \leq T\big\} $ and $ L = \max\big\{|W''(w)|: \, |w|\leq K_{\scriptscriptstyle T}\big\}. $ In view of \eqref{step2enest5} we have $ K_{\scriptscriptstyle T} < \infty, $ and $ L $ depends on $ u, \upsilon, W, T,  $ but it is independent of $ t. $

We then have that
\begin{equation}\label{Uniq6}
\left|
W'\big(\upsilon(x,t)\big)
\,-\,
W'\big(u(x,t)\big)
\right|
\leq
L \, |\upsilon(x,t) - u(x,t)| , \qquad 0 \leq t \leq T.
\end{equation}
Regarding the first term in the rhs of \eqref{Uniq5}, we clearly have
$$
\mathrm{v}(1,t)
=
\int_y^1 \mathrm{v}_x \,dx
\,+\,
\mathrm{v}(y,t)
\leq
\int_0^1 |\mathrm{v}_x|\,dx
\,+\,
\mathrm{v}(y,t)
$$
and integrate this inequality with respect to $ y, $ to get, by virtue of \eqref{Uniq4},
\begin{equation}\label{Uniq7}
\mathrm{v}(1,t)
 \leq
\int_0^1 |\mathrm{v}_x|\,dx
\,+\,
\int_0^1 \mathrm{v}(x,t) \, dx
\; = \;
\left\|\mathrm{v}_x(\cdot,t)\right\|_1 .
\end{equation}
Moreover, by \eqref{CompatCond2} and \eqref{Uniq6} we get
\begin{equation}
\deps \, \varepsilon^2 \, \mathrm{v}_{xxx}(1,t)
=
\meps
\int_0^1
\big(
W'(\upsilon)
\,-\,
W'(u)
\big) \, dx
\;\leq\;
L \,
\meps
\int_0^1 |\upsilon - u| \, dx
=
L \,
\meps \,
\left\|\mathrm{v}\right\|_1 .\label{Uniq8}
\end{equation}
Consequently, by \eqref{Uniq7}-\eqref{Uniq8} we obtain that
\begin{eqnarray}\label{Uniq9}
\deps \, \varepsilon^2 \, \mathrm{v}_{xxx}(1,t) \, \mathrm{v}(1,t)
& \leq &
\frac{L^2 \, \meps}{4 \epsilon}
\,
\left\|\mathrm{v}\right\|_1^2
\, + \,
\epsilon\, \meps
\,
\left\|\mathrm{v}_x\right\|_1^2
\nonumber
\\[0.5em]
& \leq &
\frac{L^2 \, \meps}{4 \epsilon}
\,
\left\|\mathrm{v}\right\|^2
\, + \,
\epsilon\, \meps
\,
\left\|\mathrm{v}_x\right\|^2
\end{eqnarray}
for an arbitrary positive $ \epsilon < 1. $

As for the second term in the rhs of \eqref{Uniq5}, again we use \eqref{Uniq6} to see that
\begin{eqnarray}\label{Uniq10}
\deps
\int_0^1 \big(
W'(u)
\, - \,
W'(v)
\big) \, \mathrm{v}_{xx} \, dx
& \leq &
\deps \, L \,
\int_0^1
\big|u-v\big|
\,
|\mathrm{v}_{xx}| \, dx
\nonumber
\\[0.5em]
& \leq &
\frac{L^2 \, \deps}{4 \epsilon}
\,
\left\|\mathrm{v}\right\|^2
\, + \,
\epsilon\, \deps
\,
\left\|\mathrm{v}_{xx}\right\|^2 ,
\end{eqnarray}
and for the last term in \eqref{Uniq5}, estimate \eqref{Uniq6} yield the bound
\begin{equation}\label{Uniq11}
\meps
\int_0^1 \big(
W'(u)
\, - \,
W'(v)
\big) \, \mathrm{v}\, dx
\; \leq \;
\meps \, L \,
\int_0^1
\big|u-v\big|
\,
|\mathrm{v}| \, dx
\;=\;
\meps \, L \,\left\|\mathrm{v}\right\|^2.
\end{equation}
We apply \eqref{Uniq9}, \eqref{Uniq10}, \eqref{Uniq11} into \eqref{Uniq5} to obtain
\begin{equation}\label{Uniq12}
\frac{1}{2} \frac{d}{dt} \|\mathrm{v}\|^2
\; + \;
\deps \, (\varepsilon^2 - \epsilon) \, \|\mathrm{v}_{xx}\|^2
\; + \;
\meps \, (\varepsilon^2 - \epsilon) \,  \|\mathrm{v}_x\|^2
\leq
\left(
\frac{L^2 \, \meps}{4 \epsilon}
+ \frac{L^2 \, \deps}{4 \epsilon}
+ \meps \, L
\right)
\,
\left\|\mathrm{v}\right\|^2.
\end{equation}
Therefore
\begin{equation}
\frac{d}{dt} \|\mathrm{v}\|^2
\leq
c \,
\left\|\mathrm{v}\right\|^2 ,
\qquad 0 \leq t \leq T,\label{Uniq13}
\end{equation}
for some constant $ c $ depending on $ \deps, \meps, T, $ but independent of $ t $ and integrating with respect to $ t, $ we obtain
\begin{equation*}
\left\|\mathrm{v}(\cdot,t)\right\|^2
\leq
e^{ct} \,
\left\|\mathrm{v}(\cdot,0)\right\|^2
\;=\;
0,
\qquad 0 \leq t \leq T,
\end{equation*}
that is $ u \equiv \upsilon, $ so the solution of \eqref{eq:deltaCH-AC2}-\eqref{eq:BC1p}-\eqref{eq:BC2p}-\eqref{eq:MC} is unique.

\section*{Acknowledgment}
The research work was supported by the Hellenic Foundation for Research and Innovation (H.F.R.I.) under the First Call for H.F.R.I. Research Projects to support Faculty members and Researchers and the procurement of high-cost research  equipment grant (Project Number: HFRI-FM17-45).

\end{document}